\documentclass{amsart}
\usepackage{amsmath,amssymb} 
\usepackage{amsbsy}
\usepackage{latexsym}
\usepackage{mathrsfs}
\usepackage{graphicx}
\usepackage{youngtab}
\usepackage{bbm}
\usepackage{tikz}
\usetikzlibrary{arrows}
\usetikzlibrary{patterns}
\usetikzlibrary{shapes}
\usepackage{mathtools} 
\mathtoolsset{showonlyrefs,showmanualtags} 
\usepackage{bm}
\usepackage{hyperref}
\hypersetup{
   colorlinks=true,       
    linkcolor=blue,          
    citecolor=blue
  }

\newcommand{\R}{\mathbb{R}}
\newcommand{\C}{\mathbb{C}}
\newcommand{\E}{\mathbb{E}\,}

\newcommand{\PP}{\mathbb{P}}
\newcommand{\RP}{\mathbb{R}\mathrm{P}}
\newcommand{\CP}{\mathbb{C}\mathrm{P}}

\newcommand{\be}{\begin{equation}}
\newcommand{\ee}{\end{equation}}

\def\s{\sigma}
\def\ti{\times}
\def\ot{\otimes}
\def\e{\varepsilon}
\def\diag{\mathrm{diag}}
\newcommand{\sv}{\mathrm{sv}}

\newcommand{\edeg}{\mathrm{edeg}\,} 
\newcommand{\Hy}{\mathcal{H}}   
\newcommand{\rdeg}{\mathrm{rdeg}}
\newcommand{\GL}{\mathrm{GL}}
\newcommand{\sgn}{\mathrm{sgn}}
\newcommand{\Prob}{\mathrm{Prob}}
\newcommand{\KB}{\mathcal{K}}
\newcommand{\cZ}{\mathcal{Z}}  
\newcommand{\cM}{\mathcal{M}}  
\newcommand{\cN}{\mathcal{N}}  
\newcommand{\Reg}{\mathrm{Reg}}  
\newcommand{\Hom}{\mathrm{Hom}} 
\newcommand{\Sing}{\mathrm{Sing}} 
\newcommand{\Tu}{\mathcal{T}^\perp}
\newcommand{\pb}{}
\newcommand{\as}{\bar{\sigma}}
\newcommand{\codim}{\mathrm{codim}}

\newtheorem{thm}{Theorem}
\newtheorem{lemma}[thm]{Lemma}
\newtheorem{cor}[thm]{Corollary}
\newtheorem{prop}[thm]{Proposition}

\newtheorem*{cori}{Corollary}
\newtheorem*{thmi}{Theorem}
\newtheorem*{defii}{Definition}

\newtheorem{defi}[thm]{Definition}

\theoremstyle{remark} 
\newtheorem{remark}[thm]{Remark}
\newtheorem{example}[thm]{Example}

\numberwithin{equation}{section}
\numberwithin{thm}{section}

\title{Probabilistic Schubert Calculus}
\date{\today}

\author{Peter B\"urgisser}
\thanks{First author partially supported by DFG grant BU 1371/2-2.}
\address{Institute of Mathematics, Technische Universit\"at Berlin}
\email{pbuerg@math.tu-berlin.de}

\author{Antonio Lerario}
\address{SISSA (Trieste)}
\email{lerario@sissa.it}



\begin{document}

\maketitle

\begin{abstract}
We initiate the study of \emph{average} intersection theory in real Grassmannians.
We define the \emph{expected degree}~$\edeg G(k,n)$ of the real Grassmannian $G(k,n)$
as the average number of real $k$-planes meeting nontrivially $k(n-k)$ random subspaces of $\R^n$, 
all of dimension $n-k$, where these subspaces are sampled uniformly and independently from $G(n-k,n)$. 
We express $\edeg G(k,n)$ in terms of  
the volume of an invariant convex body 
in the tangent space to the Grassmanian, 
and prove that for fixed $k\ge 2$ and $n\to\infty$,
$$ 
\edeg G(k,n) = \deg G_\C(k,n)^{\frac12 \e_k + o(1)},  
$$
where $\deg G_\C(k,n)$ denotes the degree of the corresponding complex Grassmannian
and $\e_k$ is monotonically decreasing with $\lim_{k\to\infty} \e_k = 1$.
In the case of the Grassmannian of lines, we prove the finer asymptotic
\begin{equation*}
 \edeg G(2,n+1) = \frac{8}{3\pi^{5/2}\sqrt{n}}\, \left(\frac{\pi^2}{4} \right)^n 
 \left(1+\mathcal{O}(n^{-1})\right).
\end{equation*}
The expected degree 
turns out to be the key quantity governing questions of the random enumerative geometry of flats. 
We associate with a semialgebraic set
$X\subseteq\RP^{n-1}$ of dimension $n-k-1$ 
its Chow hypersurface $Z(X)\subseteq G(k,n)$,
consisting of the $k$-planes~$A$ in $\R^n$ such 
whose projectivization intersects $X$. 
Denoting $N:=k(n-k)$, we show that 
$$
\E\#\left(g_1Z(X_1)\cap\cdots\cap g_N Z(X_N)\right)
 = \edeg G(k,n)  \cdot \prod_{i=1}^{N} \frac{|X_i|}{|\RP^{m}|},
$$ 
where each $X_i$ is of dimension $m=n-k-1$, the expectation is taken with respect to independent uniformly distributed 
$g_1,\ldots,g_m\in O(n)$ and $|X_i|$ denotes the $m$-dimensional volume of $X_i$.
\end{abstract}


\section{Introduction}

\subsection{Motivation}\label{se:motivation}

Classical enumerative geometry deals with questions like: 
{\em``How many lines intersect four lines in three-dimensional space in general position?''} or more generally:  
{\em``How many lines intersect four curves of degrees $d_1, \ldots, d_4$ in three-space in general position?''}

The answer to this type of questions is provided by a beautiful, sophisticated machinery, 
which goes under the name of \emph{Schubert calculus}. 
The problem is reduced to a computation in the cohomology ring of the Grassmann manifold: 
the set of lines intersecting a given line (or curve) represents a codimension one cycle (a Schubert variety) 
in the complex Grassmannian $G_\C(2,4)$ 
of $2$-dimensional vector subspaces in $\C^4$ 
(i.e., lines in $\CP^3$),  and one has to count the number of points of intersection of four generic copies of this cycle.
When counting complex solutions, the answer to the first question is~$2$; for the second 
the answer is  $2 \cdot d_1\cdots d_4$.

We refer to the survey~\cite{kleiman-laksov:72} and 
any of the monographs~\cite{gri-ha:94,fulton-intersection:98,manivel:01,eisenbud-harris:16} 
for a treatment of Schubert calculus.

Over the real numbers this approach fails---there is no generic number of real solutions, which 
can already be seen in the basic problem of counting the real solutions to a polynomial equation.   
Understanding even the possible outcomes for higher dimensional versions of the problem becomes increasingly complicated. 
A main question considered in real enumerative geometry is whether the number of complex solutions 
can be realized over the reals, in which case one speaks of a {\em fully real problem}, 
see Sottile~\cite{sottile:03,sottile:11}.
Schubert calculus is known to be fully real~\cite{sottile-99,vakil:06}. 
Recently Finashin and Kharlamov \cite{abundance} have identified a class of enumerative problems for which one can get an analog
of Schubert calculus over the reals, where solutions are counted with signs (these are Schubert problems which
are expressible in terms of what in \cite{abundance} is called Euler-Pontryagin ring). 

The goal of our work is to study enumerative problems over the reals in a probabilistic way, by answering questions like:
\be \label{eq:real4}
\textrm{\emph{``On average, how many \emph{real} lines intersect four \emph{random} lines in \emph{$\RP^3$}?''}}.
\ee
This approach offers a broad point of view to the subject, by seeking typical properties of random arrangements.
As we will see, the theory we will present will blend ideas from (real) algebraic geometry with integral geometry 
and the theory of random polytopes.


Classical Schubert calculus deals with the intersection of Schubert varieties 
in general position. 
Our paper is a first attempt at developing such a theory over the reals from the probabilistic point of view. 
By {\em probabilistic Schubert calculus} we understand the 
investigation of the expected number of points of intersection of
real Schubert varieties in random position.
We confine our study to special Schubert varieties, thus ignoring
the more complicated flag conditions. For the sake of simplicity, 
we also restrict our presentation to the codimension one case, i.e., intersection of 
hypersurfaces in the real Grassmannian, even though 
our methods work in more generality, as pointed out 
in different remarks throughout the text. 
It may be interesting to note that many results in Schubert calculus, 
going back to Schubert and Pieri, were first shown in the special cases of codimension one,
for $G_\C(2,n+1)$, and for special Schubert varieties, before reaching complete generality. 
In that sense, our paper fits in nicely with that tradition.
We plan to treat the case of higher codimension 
in a future publication.

\subsection{Expected degree of real Grassmannians} 

Generalizing the question~\eqref{eq:real4}, 
the number of lines  meeting $2n-2$ generic projective subspaces of dimension $n-2$ in $\CP^n$
equals the degree of the complex Grassmann manifold $G_\C(2,n+1)$, which is known to be the 
Catalan number $\frac{1}{n} {2n-2 \choose n-1}$ and asymptotically, when $n\to \infty$, behaves as:
\be\label{eq:degasy} 
  \deg G_\C(2,n+1) = \frac{4^{n-1}}{n^{3/2}\sqrt{\pi}}\left(1+\mathcal{O}(n^{-1})\right). 
\ee
Similarly, the degree of $G_\C(k,n)$ equals the number of $k$-planes in $\C^n$ meeting nontrivially $k(n-k)$ generic subspaces, 
all of dimension $n-k$.

For the study of the corresponding problem over the reals, we can introduce the following probabilistic setup. 
Recall that for every $(k,n)$ there is a unique probability distribution on the real Grassmannian $G(k,n)$, 
which is invariant under the action of $O(n)$; we call this distribution the \emph{uniform distribution}. 
A random $k$-dimensional linear subspace of $\R^n$ is obtained by sampling from this distribution. 

\begin{defii}
We define the \emph{expected degree} of $G(k,n)$ as the average number $\edeg G(k,n)$ of real $k$-planes 
meeting nontrivially $k(n-k)$ many random, real, independent subspaces
of $\R^n$, all of dimension $n-k$.
\end{defii}

More specifically, 
we define the special Schubert variety $\Sigma(k,n)\subseteq G(k,n)$ as the set of $k$-planes in $\R^n$ 
intersecting a fixed linear subspace~$V$ 
of dimension $n-k$ nontrivially. Then $\Sigma(k,n)$ is a codimension-one real algebraic subset of $G(k,n)$, 
and the expected degree equals 
\be  
 \edeg G(k,n) = \E \#\left(g_1\Sigma(k,n)\cap \cdots \cap g_{N}\Sigma(k,n)\right),
\ee
where $N:=k(n-k)=\dim G(k,n)$, 
and the expectation is taken over independent uniformly distributed $g_1, \ldots, g_{N}\in O(n)$.

The expected number of lines meeting four random lines in $\RP^3$ equals 
$\edeg(2,4) = 1.7262...$ (the exact value is not known, see Proposition \ref{prop:twolines} 
for an explicit expression involving an iterated integral). 
More generally, for the Grassmannian of lines, 
we prove the following asymptotic (Theorem \ref{thm:edlines}), 
which should be compared with its complex analogue \eqref{eq:degasy}.

\begin{thmi}
The expected degree of the Grassmannian of lines equals
\be
  \edeg G(2,n+1) = \frac{8}{3\pi^{5/2}\sqrt{n}}\, \Big(\frac{\pi^2}{4}\Big)^n 
         \left(1+\mathcal{O}(n^{-1})\right).
\ee
\end{thmi}

For the general case, we define $\rho_k :=\E\|X\|$, where $X\in\R^k$ is a standard normal Gaussian $k$-vector. 
(It is known that $\rho_k \le \sqrt{k}$, which is asymptotically sharp; see~\eqref{eq:def-rho}.)  
We show the following explicit upper bound 
\be 
 \edeg G(k,n) \ \le\ \frac{|G(k,n)|}{|\RP^{k(n-k)}|}\cdot 
    \left(\sqrt{\frac{\pi}{2}}\, \frac{\rho_k}{\sqrt{k}}\right)^N .
\ee
Note that this specializes to $\edeg G(2,n) \le (\pi^2/4)^{n-2}$. 
Moreover, we prove that (in the logarithmic scale) this upper bound is asymptotically sharp, deducing the following 
(see Theorem~\ref{thm:ased}).

\begin{thmi}
For fixed $k$, as $n\to \infty,$ the following asymptotic holds: 
\be 
\log\edeg G(k,n) = kn\log\left(\frac{\sqrt{\pi}\,\Gamma\left(\frac{k+1}{2}\right)}{\Gamma\left(\frac{k}{2}\right)}\right) 
    -\mathcal{O}(\log n).
\ee
\end{thmi}


A general, guiding theme of our work is to compare the number of complex solutions 
with the expected number of real solutions. With this regard, we obtain: 

\begin{cori}
For fixed $k\ge 2$, we have  
$\edeg G(k,n) = \deg G_\C(k,n)^{\frac12 \e_k + o(1)}$ for $n\to\infty$,
where $\e_k$ is monotonically decreasing with $\lim_{k\to\infty} \e_k = 1$. 
(See Corollary~\ref{cor:exponent} for an explicit expression for $\e_k$.) 
\end{cori}

This means that for large $n$, the expected degree of the real Grassmannian
exceeds the square root of the degree of the corresponding complex Grassmannian,
and $\frac12(\e_k-1)$ measures the deviation in the exponent. 
If also $k\to\infty$, we have an {\em asymptotic square root law},  
$\edeg G(k,n) = \deg G_\C(k,n)^{\frac12 +o(1)}$,
saying that the expected number of solutions is roughly the 
square root of the number of complex solutions. 

\subsection{Random incidence geometry over the reals} 

The expected degree of Grassmannians turns out to be the key quantity governing questions of random incidence geometry, 
as we discuss now. Given semialgebraic subsets $X_1, \ldots, X_N\subset \RP^{n-1}$ of dimension $n-k-1$, 
where $N=k(n-k)$, we will say that they are in \emph{random position} if they are randomly translated 
by elements $g_1, \ldots, g_N$ sampled independently from the orthogonal group $O(n)$ with the uniform distribution.  
The problem of counting the number of $k$-planes whose projectivization intersects $X_1, \ldots, X_N$ 
in random position can be geometrically described as follows. 

We associate with a semialgebraic subset $X\subseteq \RP^{n-1}$ of dimension $n-k-1$ 
the \emph{Chow hypersurface} 
\be 
  Z(X) := \{\textrm{$k$-planes in $G(k,n)$ whose projectivization intersects $X$}\} \subseteq G(k,n) ,
\ee
which is the real analog of the associated Chow variety in complex algebraic geometry.
The special real Schubert variety $\Sigma(k,n)$ is obtained by taking for $X$ a linear space. 
In the following, we denote by $|X|$ the $(n-k-1)$-dimensional volume of 
the set of smooth points of $X$. 

A translation of $X$ by an element $g\in O(n)$ corresponds to a translation of the associated Chow hypersurface $Z(X)$ 
by the induced action of the same element on $G(k,n)$. 
In the geometry of the Grassmannian, the random incidence problem is equivalent to the computation of 
the average number of intersection points of random translates of the sets $Z(X_1), \ldots, Z(X_N)$.

The following theorem decouples the problem into the computation of the volume of $X_1, \ldots, X_N$  
and the determination of the expected degree of the Grassmann manifold 
(see Theorem~\ref{thm:RIG}). Note that over the complex numbers, 
the same decoupling result is a consequence of the ring structure of the cohomology of the Grassmannian 
(compare \S\ref{se:IS-C} below).

\begin{thmi}
The average number of $k$-planes whose projectivization intersects semialgebraic sets $X_1, \ldots, X_N$ 
of dimension $n-k-1$ in random position in $\RP^{n-1}$ equals
\be\label{eq:interi}  
 \E\#\left(g_1Z(X_1)\cap\cdots\cap g_N Z(X_N)\right)=\edeg G(k,n)  \cdot \prod_{i=1}^{N} \frac{|X_i|}{|\RP^{n-k-1}|} .
\ee
\end{thmi}

\begin{example}[Lines intersecting four random curves in three-space]
Let us consider the problem of counting the number of lines intersecting four random curves in $\RP^3$. 
A possible model for random algebraic varieties, which has attracted a lot of attention 
over the last years~\cite{buerg:07,Letwo,LeLu:gap,FyLeLu, GaWe1,GaWe3,GaWe2}, 
is the so called \emph{Kostlan model}, which we describe now for the case of space curves. 
First let us recall that a random Kostlan polynomial $f\in \R[x_0, \ldots, x_3]_{(d)}$ is defined as
\be 
 f(x)=\sum_{|\alpha|=d}\xi_\alpha x_0^{\alpha_0}\cdots x_3^{\alpha_3},
\ee
where the $\xi_\alpha$ are independent centered gaussian variables with variance $\frac{d!}{\alpha_0!\cdots \alpha_3!}$. 
This model is invariant under the action of $O(4)$ by change of variables (there are no preferred points or directions 
in projective space). A random Kostlan curve is defined as the zero set of two independent Kostlan polynomials
\be  
 C=\{f(x)=g(x)=0\}.
\ee 
It is well known \cite{kostlan:93} that the expectation of the length of $C$ equals 
$\left(\deg f \deg g\right)^{\frac{1}{2}}|\RP^1|$. 
Using this fact in the above theorem, applied to four random independent curves $C_1, \ldots, C_4$, 
and combined with the $O(4)$-invariance of the model, 
it is easy to show that the expected number of lines intersecting the four curves equals
\be 
 \E\#\{\textrm{lines intersecting $C_1, \ldots, C_4$}\} 
 = \edeg G(2,4) \cdot \prod_{i=1}^4\left(\deg f_i \deg g_i\right)^{\frac{1}{2}}.
\ee
Observe that the randomness in this problem does not come from putting the curves in random position, but rather from sampling them from a distribution which is invariant under random translations: this allows to exchange the order of the expectations without changing the result.
(More generally a similar argument can be applied to other invariant models and in higher dimensions.)
\end{example}

\subsection{Iterated kinematic formula and volume of Schubert varieties}

A key technical ingredient of our work is a generalization of the kinematic formula  
for homogeneous spaces in Howard~\cite{Howard} to multiple intersections (Theorem~\ref{thm:GIGF}). 
This generalization is novel, but the proof is technically inspired by~\cite{Howard},
so that we have moved it to the Appendix~\ref{se:GIGF}. 
What one should stress at this point is that there is no ``exact'' kinematic formula for the Grassmann manifold $G(k,n)$, 
meaning that the evaluation of kinematic integrals depends on the class of submanifolds we consider. 
This is in sharp contrast with the case of spheres or projective spaces, and is essentially due to the fact that 
the stabilizer $O(k)\times O(n-k)$ of the action is too small to enforce transitivity at level of tangent spaces. 
In the context of enumerative geometry, however, one can still produce an explicit formula, 
thanks to the observation that the orthogonal group acts transitively 
on the tangent spaces of Chow hypersurfaces (and in particular Schubert varieties).

We use the resulting formula for the derivation of \eqref{eq:interi}. 
In practice this requires the computation of the volume of the associated Chow hypersurfaces (Theorem \ref{thm:IGF}), 
as well as the evaluation of an average angle between them (see \S\ref{se:bodies} below).

By means of classical integral geometry in projective spaces, in Proposition \ref{prop:chow} we prove that
\be\label{eq:volratioR}
  \frac{|Z(X)|}{|\Sigma(k,n)|} =\frac{|X|}{\textrm{\emph{$|\RP^{\dim X}|$}}} .
\ee
We note this is analogous to the corresponding result over $\C$;  
see \eqref{eq:chow-volratioC}. 
The computation of the volume of the special Schubert variety $\Sigma(k,n)$ requires more work, 
and constitutes a result of independent interest (Theorem~\ref{prop:vspe}), 
due to its possible applications in numerical algebraic geometry~\cite{bu-co-is}. 

\begin{thmi}\label{thmi:volsch}
The volume of the special Schubert variety $\Sigma(k,n)$ satisfies:
\be 
 \frac{|\Sigma(k,n)|}{|G(k,n)|} 
   = \pi \cdot \frac{|\RP^{k-1}|}{|\RP^k|} \cdot \frac{|\RP^{n-k-1}|}{|\RP^{n-k}|} ,
\ee
where the volume $|G(k,n)|$ of the Grassmann manifold is given by \eqref{eq:Gr-vol}.
\end{thmi}

This result shows a striking analogy with the corresponding result over $\C$; see~\eqref{eq:HlinC}.

\begin{remark}
A known result in integral geometry, \cite[Eq.~(17.61)]{Santalo}, gives the measure of the set of 
projective $(k-1)$-planes meeting a fixed spherically convex set $Q$ in $\RP^{n-1}$.
A natural idea would be to apply this result to the $\epsilon$-neighborhood $Q_\epsilon$ in $\RP^{n-1}$ 
around a fixed $(n-k-1)$-plane and to consider the limit for $\epsilon\to 0$ (after some scaling).
However, this argument is flawed since neighborhoods $Q_\epsilon$ are not spherically convex!
\end{remark}


Finally, let us mention that the integral geometry arguments can 
also be applied to the classical setting over $\C$. 
There, the situation is considerably simpler, since the scaling constant 
$\alpha_\C(k,m)$ appearing in the corresponding 
integral geometry formula (Theorem~\ref{thm:IGFC})
can be explicitly determined. On the other hand, 
we do not have a closed formula for its 
real version $\alpha(k,m)$, which is captured 
by the expected degree of Grassmannians. 
As a consequence, it turns out that classical results 
of enumerative geometry over $\C$ can be obtained 
via the computation of volumes and the evaluation of integrals.
We illustrate this general observation by the computation of 
the degree of $G_\C(k,n)$ in the proof of Corollary~\ref{cor:degGC}. 

\begin{remark}
It is interesting to observe that over the complex numbers the degree of $G_\C(k,n)$ can be obtained in two different ways: 
(1) intersecting the Grassmannian in the Pl\"ucker embedding with $N=k(n-k)$ generic hyperplanes; 
(2) intersecting $N$ many generic copies of the Schubert variety $\Sigma_\C(k,n)$ inside $G_\C(k,n)$ itself. 
The second method is equivalent to intersecting the image of the Pl\"ucker embedding with $N$ very nongeneric hyperplanes.

By contrast, over the real numbers the two procedures give very different answers: 
(1) averaging the intersection of the real Grassmannian in the Pl\"ucker embedding with $N$ random hyperplanes gives 
$\frac{|G(k,n)|}{|\RP^N|}$ (by classical integral geometry);
(2)  averaging the intersection of $N$ random copies of $\Sigma(k,n)$ produces an \emph{intrinsic answer} and gives the expected degree.
\end{remark}

\subsection{Link to random convex bodies}\label{se:bodies}

The proof of our results on the expected degree employs an 
interesting and general connection between expected absolute random determinants 
and random polytopes that are zonoids (Vitale~\cite{Vitale}). 
This allows to express the average angle between random Schubert varieties in terms of the volume of 
certain convex bodies, for which we coined the name 
{\em Segre zonoids}~$C(k,m)$, see \S\ref{se:coiso-zonoid}.
These are zonoids of matrices in $\R^{k\ti m}$ that are invariant under the action 
of the group $O(k)\ti O(m)$.
Via the singular value decomposition, the Segre zonoid
fibers over a zonoid $D(k)$ of singular values in $\R^k$ (if $k\le m$) 
and its volume can be studied this way. 
Via a variant of Laplace's method~\cite{FulksSather} 
we prove that, asymptotically for $m\to\infty$ 
and up to a subexponential factor $e^{o(m)}$, 
the volume of $C(k,m)$ equals the volume of the 
smallest ball including $C(k,m)$. 
The volume of this ball is not the same, but 
related to what one gets when carrying out 
the analogous argument for computing the 
degree of the complex Grassmann manifold.

\subsection{Related work}
A square root law (in the form of an equality) was for the first time discovered by 
Kostlan~\cite{kostlan:93} and Shub and Smale~\cite{Bez2}, 
who found a beautiful result on the expected number of real solutions 
to random polynomial systems. This work has strongly inspired ours. 
Paul Breiding~\cite{breiding:16} has recently discovered results in the same spirit 
for the (expected) number of (real) eigenvalues of tensors. 
Related to this work, the second author of the present paper, together with S. Basu, E. Lundberg and C. Peterson have recently obtained results on counting real lines on random real hypersurfaces \cite{BLLP}, obtaining a similar square root law (already conjectured by the first author of the current paper).
When the zero set of a random system of polynomial equations is not 
zero dimensional, its size may be measured either in terms of volume 
or Euler characteristic, which was investigated by the first author of the current paper in~\cite{buerg:07}. 
We refer to the papers \cite{dhost:13} and \cite{draisma-horobet:14} 
that investigate the average number of real solutions in different contexts.
More recently, the subject of \emph{random algebraic geometry}, meant as the study of topological properties of random real algebraic sets, 
has become very popular. Asymptotic square root laws have been found for Betti numbers of Kostlan hypersurfaces  
by Gayet and Welschinger \cite{GaWe1,GaWe3,GaWe2} (and in general for invariant models by the second author of the current paper, 
Fyodorov and Lundberg \cite{FyLeLu}); similar results for Betti numbers of intersection of random quadrics have been obtained 
by the second author of the current paper and Lundberg \cite{ Letwo,LeLu:gap}.

\subsection{Acknowledgments}
We are very grateful to Frank Sottile who originally suggested this line of research. 
We thank Paul Breiding, Kathl\'en Kohn, Chris Peterson, and Bernd Sturmfels for discussions. 
We also thank the anonymous referees for their comments.







\section{Preliminaries}

\subsection{Real Grassmannians}
\label{se:GrassR}



Suppose that $E$ is a Euclidean vector space of dimension~$n$.
We have an inner product on the exterior algebra
$\Lambda(E)=\bigoplus_{k=0}^{\infty}\Lambda^k(E)$,  
which in coordinates is described as follows.  
Let $e_1,\ldots,e_n$ be an orthonormal basis of $E$ 
and put $e_I := e_{i_1}\wedge\cdots\wedge e_{i_k}$ for 
$I=\{i_1,\ldots,i_k\}$ with $i_1 < \cdots < i_k$. Then 
the $e_I$ form an orthonormal basis of $\Lambda^k(E)$, 
where $I$ runs over all subsets of~$[n]$ with $|I| =k$.  
Hence, for $a_I\in\R$, 
$$
 \| \sum_U a_I e_I \|^2 = \sum_I |a_ I|^2 .
$$
Let $v_i =\sum_{j=1}^p x_{ij} e_j$ with $x_{ij}\in\R$. 
Then 
$v_1 \wedge\cdots\wedge v_k = \sum_I \det(X_I) e_I$,
where $X_I$ is the $k\times k$ submatrix obtained from $X=(x_{ij})$ 
by selecting the columns indexed by the numbers in~$I$. 
Hence 
\be\label{eq:norm-fomula}
 \|v_1 \wedge\cdots\wedge v_k \|^2= \sum_I |\det(X_I)|^2 
   = \det (\langle v_i, v_j \rangle)_{1\le i,j\le k} ,
\ee
where the second equality is the well known Binet-Cauchy identity. 
(Clearly, the above quantity does not change if we substitute the 
$v_i$ by $\sum_{j=1}^k g_{ij}v_j$ with $(g_{ij})\in\mathrm{SL}_k$.) 
If $w_i =\sum_{j=1}^p y_{ij} e_j$ with $y_{ij}\in\R$, we obtain 
for the scalar product between two simple $k$-vectors
(compare~\cite{Kozlov}) 
\be\label{eq:scalar} 
  \langle v_1\wedge\cdots\wedge v_k, w_1\wedge\cdots\wedge w_k\rangle 
  = \det (\langle v_i, w_j\rangle)_{1\le i,j\le k}  .
\ee


Let $G^+(k,E)$ denote the set of oriented $k$-planes in $E$.
Let $L\in G^+(k,E)$ and $v_1,\ldots,v_k$ be an oriented orthonormal basis of $E$. 
Then $v_1\wedge\cdots\wedge v_k$ is independent of the choice of the basis. 
Via the (injective) Pl\"ucker embedding 
$G^+(k,E) \to \Lambda^k(E), L \mapsto v_1\wedge\cdots\wedge v_k$
we can identify $G^+(k,E)$ with the set of norm-one, simple $k$-vectors in $\Lambda(E)$.
We note that the orthogonal group $O(E)$ acts isometrically on $\Lambda(E)$ 
and restricts to the natural action on $G^+(k,E)$ that, for $g\in O(E)$, 
assigns to an oriented $k$-plane $L$ its image $gL$. 
By this identification, $G^+(k,E)$ is a smooth submanifold of $\Lambda^k(E)$ 
and inherits an $O(E)$-invariant Riemannian metric from the ambient euclidean space. 
We denote by $|U|$ the volume of a measurable subset of $G^+(k,E)$. 

The {\em real Grassmann manifold} $G(k,E)$ is defined as the set of $k$-planes in $E$. 
We can identify $G(k,E)$ with the image of $G^+(k,E)$ under the quotient map 
$q\colon S(\Lambda(E))\to \PP(\Lambda(E))$ from the unit sphere of $\Lambda(E)$, 
which forgets the orientation.
The Riemannian metric on $G(k,E)$ is defined as the one for which 
$q\colon G^+(k,E) \to G(k,E)$ is a local isometry. 
(Note that this map is a double covering.)

The {\em uniform probability} measure on $G(k,E)$ is defined 
by setting: \be\Prob(V) := \frac{|V|}{|G(k,E)|}, \quad \textrm{for all measurable $V\subseteq G(k,E)$.}\ee
We just write $G(k,n):=G(k,\R^n)$ and $G^+(k,n):=G^+(k,\R^n)$
where $E=\R^n$ has the standard inner product. 

The Pl\"ucker embedding allows to view the tangent space 
$T_A G(k,E)$ at $A\in G(k,n)$ as a subspace of $\Lambda^k(E)$. 
While this is sometimes useful for explicit calculations, it is often helpful 
to take a more invariant viewpoint. We canonically have 
$T_A G(k,n) = \Hom(A,A^\perp)$, compare~\cite[Lect.~16]{harris}.
To $\alpha\in\Hom(A,A^\perp)$ there corresponds the tangent vector 
$\sum_{i=1}^k a_1\wedge\cdots\wedge a_{i-1}\wedge \alpha(a_i)\wedge a_{i+1}\wedge\cdots\wedge a_k$, 
where $A=a_1\wedge\cdots\wedge a_k$. 
Moreover, if $g_*\colon G(k,E)\to G(k,E)$ denotes the action corresponding 
to $g\in O(E)$, then its derivative 
$D_A g_*\colon T_A G(k,E)\to T_{g(A)} G(k,E)$ at $A\in G(k,E)$
is the map  
$\Hom(A,A^\perp) \to \Hom(g(A),g(A)^\perp),\,  
\alpha\mapsto g\circ \alpha\circ g^{-1}$. 
It is easy to verify that the norm of a tangent vector
$\alpha\in\Hom(A,A^\perp)$ in the previously defined 
Riemannian metric equals the Frobenius norm of $\alpha$.

\begin{remark}
It is known that (see \cite[Section 6.4]{Kozlov}) that for $(k,n)\neq (2,4)$, there is a unique (up to multiples) 
$O(n)$-invariant Riemannian metric on $G^{+}(k,n)$; 
moreover for all $k,n$ there is a unique (up to multiples)
$O(n)$-invariant volume form and consequently 
a unique $O(n)$-invariant probability distribution on $G^{+}(k,n)$. 
\end{remark}

\subsection{Regular points of semialgebraic sets}\label{sec:semialg} \pb

We refer to \cite{BCR:98} as the standard reference for real algebraic geometry
and recall here some not so well known notion from real algebraic geometry. 
Let $M$ be a smooth real algebraic variety, e.g., 
$\R^n$ or a real Grassmannian, and 
let $S\subseteq M$ be a semialgebraic subset. 
A point $x\in S$ is called {\em regular of dimension~$d$} 
if $x$ has an open neighborhood $U$ in $M$ such that 
$S\cap U$ is a smooth submanifold of $U$ of dimension~$d$. 
The dimension of $S$ can be defined as the maximal dimension of 
regular points in $S$ (which is well defined since regular points 
always exist). The points which fail to be regular are called 
{\em singular}. We denote the set of singular points by $\Sing(S)$
and say that $S$ is {\em smooth} if $\Sing(S)=\emptyset$.
Moreover, we write $\Reg(S):= S\setminus\Sing(S)$. 
The following is stated in~\cite[\S 4.2]{dimca:87} 
without proof, refering to~\cite{lovasi:65}. 
A proof can be found in~\cite{stasica:03}. 

\begin{prop}\label{prop:singular}
The set $\Sing(S)$ is semialgebraic and $\dim \Sing(S) < \dim S$. 
In particular, $S\ne \Sing(S)$.
\end{prop}

\pb
In the sequel it will be convenient to say that generic points 
of a semialgebraic set $S$ satisfy a certain property if this property 
is satisfied by all points except in a semialgebraic subset of positive codimension in $S$. 
(For example, in view of the previous proposition, 
generic points of a semialgebraic set $S$ are regular points.) 


\subsection{Complex Grassmannians}	
\label{se:GrassC}

The {\em complex Grassmann manifold} $G_\C(k,n)$ is defined as the set of complex $k$-planes in $\C^n$. 
It can be identified with its image under the (complex) Pl\"ucker embedding 
$\iota\colon G_\C(k,n) \to \PP(\Lambda(\C^n))$, that is defined 
similarly as over $\R$ in \S\ref{se:GrassR}. 
We note that the unitary group $U(n)$ acts isometrically on $\Lambda(\C^n)$ 
and restricts to the natural action on $G_\C(k,n)$ that, for $g\in U(n)$, 
assigns to a complex $k$-plane $L$ its image $gL$. 
Since the image of $\iota$ is Zariski closed, we can view $G_\C(k,n)$ 
as a complex projective variety, which inherits an $U(n)$-invariant 
Hermitian metric from the ambient space. The real part of this Hermitian 
metric defines a Riemannian metric on $G_\C(k,n)$.
We denote by $|V|$ the volume of a measurable subset $V$ of $G_\C(k,n)$. 
The degree of $G_\C(k,n)$ as a projective variety is well known, 
see Corollary~\ref{cor:degGC}.

Let $\Hy\subseteq G_\C(k,n)$ be an irreducible algebraic hypersurface in $G_\C(k,n)$. 
It is known that the vanishing ideal of $\Hy$ in the homogeneous
coordinate ring of $G_\C(k,n)$ is generated by a single equation 
in the Pl\"ucker coordinates; cf.~\cite[Chap.~3, Prop.~2.1]{GKZ}. 
We shall call its degree the {\em relative degree} $\rdeg\Hy$ of $\Hy$. 
This naming is justified by the following observation. 
The image $\iota(\Hy)$ of $\Hy$ under the Pl\"ucker embedding $\iota$ 
is obtained by intersecting 
$\iota(G_\C(k,n))$ with an irreducible hypersurface of degree~$\rdeg\Hy$ 
and hence, by B\'ezout's theorem, 
\begin{equation}\label{eq:rdegH}
 \deg\iota(\Hy) = \rdeg\Hy\cdot \deg G_\C(k,n) .
\end{equation}


For dealing with questions of incidence geometry, we introduce the 
{\em Chow variety} $Z_\C(X)\subseteq G_\C(k,n)$ associated with 
an irreducible complex projective variety $X\subseteq \CP^{n-1}$ of dimension $n-k-1$.
One defines $Z_\C(X)$ as the set of $k$-dimensional linear subspaces~$A\subseteq\C^n$  
such that their projectivization satisfies $\PP(A) \cap X \ne\emptyset$. 
Is it known that $Z(X)$ is an irreducible algebraic hypersurface in $G_\C(k,n)$; 
moreover $\rdeg Z_\C(X)= \deg X$; cf.~\cite[Chap.~3, \S2.B]{GKZ}. 
If we choose $X$ to be a linear space, then $Z_\C(X)$ is a 
{\em special Schubert variety} of codimension one that we denote by $\Sigma_\C(k,n)$ 
(notationally ignoring the dependence on $X$). 

\subsection{Degree and volume}

A fundamental result in complex algebraic geometry due to Wirtinger states that 
degree and volume of an irreducible projective variety~$X\subseteq \CP^{n}$ are linked via 
\be\label{th:degvol} 
  |X| = \deg X \cdot |\CP^m|, \quad\mbox{where $m=\dim X$} ;
\ee 
see \cite[\S5.C]{mumford} or \cite[Chap.~VIII, \S4.4]{shaf-2}. 
This is a key observation, since the interpretation of the degree as a volume ratio paves 
the way for arriving at analogous results over the reals. 
%
%

Equation~\eqref{th:degvol}, 
combined with the fact that the Pl\"ucker embedding is isometric, implies 
\be\label{eq:vol-deg-G}
 |G_\C(k,n)| = \deg G_\C(k,n) \cdot |\CP^{N}| ,
\ee
where $N :=\dim G_\C(k,n) = k(n-k)$. 
If $X\subseteq \CP^{n-1}$ is an irreducible projective variety of dimension $n-k-1$, 
we obtain for the Chow hypersurface $Z_\C(X)$ by the same reasoning, using~\eqref{eq:rdegH}, 
\be\label{eq:A}
|Z_\C(X)| = \deg \iota(Z_\C(X)) \cdot |\CP^{N-1}| 
   = \rdeg Z_\C(X) \cdot \deg G_\C(k,n) \cdot |\CP^{N-1}| .
\ee
In particular, taking for $X$ a linear space, we get for the special Schubert variety: 
\be\label{eq:chow-C} 
 |\Sigma_\C(k,n))|  = \deg G_\C(k,n) \cdot |\CP^{N-1}| .
\ee
Dividing this equation by \eqref{eq:vol-deg-G} we obtain for the volume of $\Sigma_\C(k,n)$, 
cf.~\cite{bu-co-is}:
\be\label{eq:HlinC}
 \frac{|\Sigma_\C(k,n)|}{ |G_\C(k,n))|} = \frac{|\CP^{N-1}|}{|\CP^{N}|} = \frac{N}{\pi} 
 = \pi \cdot \frac{|\CP^{k-1}|}{|\CP^{k}|} \cdot \frac{|\CP^{n-k-1}|}{|\CP^{n-k}|} .
\ee
We remark that this formula can also by obtained in a less elegant way 
by calculus only (see first version of \cite{bu-co-is}).
In Theorem~\ref{prop:vspe} 
we derive a real analogue of~\eqref{eq:HlinC},
based on the calculation of the volume of the tube around the special Schubert 
variety $\Sigma(k,n).$

We can easily derive more conclusions from the above:
dividing the equations \eqref{eq:A} and \eqref{eq:chow-C}, we get:
\be\label{eq:chow-volratioC}
 \frac{|Z_\C(X)|}{ |\Sigma_\C(k,n))|} =  \rdeg Z_\C(X) = \deg X = \frac{|X|}{|\CP^{\dim X}|} .
\ee
In Proposition~\ref{prop:chow} we derive an analogous result over $\R$.

\subsection{Intersecting hypersurfaces in complex Grassmannians}
\label{se:IS-C}

For the sake of comparison, we describe the classical intersection theory 
in the restricted setting of intersecting hypersurfaces of~$G_\C(k,n)$.


Many results in enumerative geometry over $\C$ can be obtained from the following 
well known consequence of B\'ezout's theorem. 
In Theorem~\ref{thm:RIG} we provide a real analogue of this result. 

\begin{thm}\label{th:intersect_C}
Let $\Hy_1,\ldots,\Hy_N$ be irreducible algebraic hypersurfaces in $G_\C(k,n)$, 
where $N:=k(n-k)$.  
Then we have:  
$$
 \#(g_1\Hy_1\cap\ldots\cap g_N\Hy_N) 
  = \deg G_\C(k,n) \cdot \rdeg\Hy_1\cdot \ldots \cdot\rdeg\Hy_N 
$$
for almost all $(g_1,\ldots,g_N)\in (\GL_{n})^N$. 
In particular, if $X_1,\ldots,X_N\subseteq \PP^n(\C)$ are irreducible projective subvarieties 
of dimension $n-k-1$, then: 
$$
 \#(g_1 Z_\C(X_1)\cap\ldots\cap g_N Z_\C(X_N)) 
 = \deg G_\C(k,n) \cdot \deg X_1\cdot\ldots\cdot\deg X_N  
$$
for almost all $(g_1,\ldots,g_N)\in (\GL_{n})^N$.
\end{thm}

\begin{proof}
A general result by Kleiman~\cite{kleiman:74} implies that
$g_1\Hy_1,\ldots,g_N\Hy_N$ meet transversally, for almost all $g_1,\ldots,g_n$. 
Recall the Pl\"ucker embedding 
$\iota\colon G_\C(k,n) \hookrightarrow \PP(\Lambda^{k}\C^{n})$. 
Note that $\iota(\Hy_i)$ is obtained by intersecting 
$\iota(G_\C(k,n))$ with an irreducible hypersurface $H_i$ in $\PP(\Lambda^{k}\C^{n})$ 
of degree $\rdeg\Hy_i$. 
We thus have 
$\iota(g_1\Hy_1)\cap\ldots\cap \iota(g_1\Hy_N) = \iota(G_\C(k,n)) \cap g_1\Hy_1\cap\ldots\cap g_N\Hy_N$
and the assertion follows from B\'ezout's theorem.
\end{proof}

\subsection{Convex bodies}
\label{se:CB}

For the following see \cite{Schneider}.  
By a {\em convex body} $K$ we understand a nonempty, compact, convex subset of $\R^d$ with nonempty interior. 
The {\em support function} $h_K$ of $K$ is defined as: 
\be 
 h_K(u) := \max \big\{\langle x,u\rangle \mid x \in K \big\} \quad \mbox{ for $u\in\R^d$.}
\ee
One calls 
$\{ x\in\R^d \mid \langle x,u \rangle = h_K(u) \}$ 
the {\em supporting hyperplane} in the direction $u\in\R^d\setminus\{0\}$. 
The support function $h_K$ is a positively homogeneous, subadditive function that, by 
the hyperplane separation theorem, characterizes $K$ as follows:
\be 
  x \in K \Longleftrightarrow \forall u\in \R^d\quad \langle x,u\rangle \le h_K(u) .
\ee
If $0\in K$ one defines the {\em radial  function} $r_K$ of $K$ by: 
\be
 r_K(u) := \max\big\{ t\ge 0 \mid tu \in K \big\} \quad \mbox{ for $u\in\R^d\setminus\{0\}$.}
\ee
The {\em radius} $\|K\|$ of $K$ is defined as the maximum of the function~$r_K$ 
on the unit sphere $S^{d-1}$. 
		
\begin{lemma}\label{lemma:r=h}
Let $K\subseteq \R^d$ be a convex body containing the origin. 
Then the radial function~$r_K$ and the support function $h_K$ of $K$ 
have the same maximum on $S^{d-1}$. 
Moreover, a direction $u\in S^{d-1}$ is maximizing 
for $r_K$ if and only $u$ is maximizing for $h_K$. 
\end{lemma}

\begin{proof}
Let $\overline{r}$ denote the maximum of $r_K$. 
Then $K\subseteq B(0, \overline{r})$ and hence $h_{K}\leq h_{B(0, \overline{r})}$,
which implies 
$\max h_K \le \max h_{B(0, \overline{r})} = \overline{r}$.
For the other direction let $u\in S^{d-1}$. Then 
$\langle x,u\rangle\leq h_{K}(u)$ for all $x\in K$.  
In particular, setting $x := r_K(u)u$, which lies in $K$, 
we obtain 
$r_K(u)=\langle r_K(u)u, u\rangle \leq h_{K}(u)$, 
and consequently $\max r_K\leq \max h_K$. 
The claim about the maximizing directions follows easily
by tracing our argument.
\end{proof}

We recall also the following useful fact \cite[Corollary 1.7.3]{Schneider}.

\begin{prop}\label{prop:nabla}
Let $K\subseteq\R^d$ be a convex body and $u\in\R^d\setminus\{0\}$. 
Then the support function~$h_K$ of $K$ is differentiable at $u$ 
if and only if the intersection of $K$ with 
the supporting hyperplane of~$K$ in the direction~$u$ 
contains only one point $x$. 
In this case $x=(\nabla h_K)(u)$. 
In particular, if $S$ 
denotes the set of nonzero points at which $h_K$ is smooth, 
we have $\nabla h_K(S)\subseteq \partial K$. 
\end{prop}

The set $\KB^d$ of convex bodies in $\R^d$ can be turned into a metric space 
by means of the Hausdorff metric. The {\em Hausdorff distance}
of the sets $K,L\in\KB^d$ is defined by: 
\be 
 \delta(K,L) := \min\big\{ t\ge 0 \mid K\subseteq L + tB^d \mbox{ and } L\subseteq K + tB^d \big\} ,
\ee
where $B^d$ denotes the closed unit ball and $+$ the Minkowski addition.

We introduce now special classes of convex bodies that will naturally arise in our work. 
A {\em zonotope} is the Minkowski sum of finitely many line segments. 
A {\em zonoid} is obtained as the limit (with respect to the Hausdorff metric)
of a sequence of zonotopes. It is easy to see that closed balls are zonoids.

\subsection{The zonoid associated with a probability distribution}
\label{se:zonoid}


The following is from Vitale~\cite{Vitale}.
A {\em random convex body} $X$ is a Borel measurable map from 
a probability space to $\KB^d$. Suppose that the radius $\|X\|$ 
of $X$ has a finite expectation. Then one can associate with $X$ 
its {\em expectation} $\E X\in \KB^d$, which is characterized by 
its support function as follows:
\be\label{eq:suppEP}
  h_{\E X}(u) = \E h_X(u) \quad \mbox{ for all $u\in\R^d$.}
\ee
Suppose now $Y$ is a random vector taking values in $\R^d$.
We associate with $Y$ the {\em zonoid} $C_Y:=\E [0,Y]$ in $\R^d$, 
which is defined if $\E\|Y\| < \infty$. In the following we 
assume that $Y$ and $-Y$ have the same distribution. 
By \eqref{eq:suppEP}, the support function $h$ of $C_Y$ is given by: 
\be\label{eq:support-h}
  h_{C_Y}(u) = \frac12\,\E |\langle u,Y\rangle| \quad\mbox{ for $u\in\R^d$.}
\ee 

\begin{example}\label{ex:zono}
Suppose that $Y\in\R^d$ is standard Gaussian. The support function~$h_1$ 
of $C_Y$ satisfies by~\eqref{eq:support-h}:
$$
h_1(e_1) = \frac12\, \E |\langle e_1,Y\rangle| 
 = \frac12\, \E\, |Y_1| = \frac{1}{\sqrt{2\pi}} .
$$ 
The function $h_1$ is $O(d)$-invariant and 
hence $h_1(u) =  \|u\|/\sqrt{2\pi}$. 
It follows that the associated zonoid~$C_Y$ 
is the ball $\frac{1}{\sqrt{2\pi}}\, B^d$ with radius $1/\sqrt{2\pi}$ 
and centered at the origin.

The random variable $\tilde{Y}:= Y/\|Y\|$ is uniformly distributed 
in the sphere~$S^{d-1}$ and $\tilde{Y}$ and $\|Y\|$ are independent. 
Hence the support function~$h_2$ of $C_{\tilde{Y}}$ satisfies:
$$
 h_2(u) = \frac12 \,\E\, \frac{|\langle u,Y\rangle|}{\|Y\|} = \frac{1}{\rho_d} h_1(u) .
$$
Hence  the zonoid associated with $\tilde{Y}$ equals the ball 
$\frac{1}{\rho_d}\frac{1}{\sqrt{2\pi}} B^d$. 
Therefore, if we sample uniformly in the sphere of radius~$\rho_d$ 
around the origin, the associated zonoid is the same as for the 
standard normal distribution, namely $\frac{1}{\sqrt{2\pi}} B^d$. 
\end{example}

Let $M_Y\in\R^{d\ti d}$ be the random matrix whose columns are 
i.i.d.\ copies of $Y$. Our goal is to analyze the expectation 
of the absolute value of the determinant of $M_Y$. 
The following result due to Vitale~\cite{Vitale} is 
crucial for our analysis in \S\ref{se:RCB}. 

\begin{thm}\label{thm:vitale}
If $Y$ is a random vector taking values in $\R^d$ such that 
$\E\|Y\| < \infty$, then 
$$
 \E | M_Y | = d!\, |C_Y| .
$$
\end{thm}

More information on this can be found in Schneider's book~\cite[\S3.5]{Schneider}, 
where $C_Y$ is called the zonoid whose generating measure is the distribution of $Y$. 
(This is under the assumption that $Y\in S^{d-1}$.)
More results in the spirit of Theorem~\ref{thm:vitale} can be 
found in \cite[\S5.3]{Schneider}. 





\subsection{Various volumes}
\label{se:vols}


In order to express the volume of orthogonal groups etc.,
we introduce the {\em multivariate Gamma function} $\Gamma_k(a)$, defined by:
\be 
 \Gamma_k(a) := \int_{A>0} e^{-\textrm{tr} A}(\det A)^{a-\frac{k+1}{4}}\,dA ,
\ee
for $\Re(a) > \frac12 (k-1)$, 
where the integral is over the cone of positive definite $k\times k$ real symmetric matrices. 
This clearly generalizes the classical Gamma function: $\Gamma_1(a)=\Gamma(a)$. 
It is known that~\cite[Theorem 2.1.12]{Muirhead}:
\be 
 \Gamma_k(a) = \pi^{k(k-1)/4}\, \Gamma(a) \Gamma(a-1/2) \cdots \Gamma(a-(k-1)/2) .
\ee

The orthogonal group 
$O(k) :=\{Q\in\R^{k\times k} \mid Q^TQ=\textbf{1}\}$ 
is a compact smooth submanifold of~$\R^{k\times k}$. 
Its tangent space at $\textbf{1}$ is given by the space of skew-symmetric matrices
$T_{\textbf{1}} O(k) = \{A\in\R^{k\times k} \mid A^T+ A = 0 \}$, 
on which we define the inner product 
$\langle A,B\rangle := \frac12 \mathrm{tr}(A^TB)$ 
for $A,B\in T_{\textbf{1}} O(k)$. 
We extend this to a Riemannian metric on $O(k)$
by requiring that the left-multiplications are isometries 
and call the resulting metric the canonical one. 
(Note that this differs by a factor $\frac12$ from the metric 
on $O(k)$ induced by the Euclidean metric on $\R^{k\times k}$.) 
Then, for $v$ in the unit sphere $S^{k-1} := \{ x\in \R^{k} \mid \|x\| = 1 \}$,
the orbit map 
$O(k)\to S^{k-1},\, Q\mapsto Qv$ 
is a Riemannian submersion. 

Similarly, if $G(k,n)$ is given the Riemannian metric defined above, it is easy to verify that the quotient map $O(n)\to G(k,n)$ is a Riemannian submersion.
We also define a Riemannian metric on the Stiefel manifold 
$S(k,m) := \{Q\in\R^{m\ti k} \mid  Q^TQ = \textbf{1}\}$ 
for $1\le k\le m$ such that 
$O(m) \to S(k,m),\, Q=[q_1,\ldots,q_m]\mapsto [q_1,\ldots,q_k]$ 
is a Riemannian submersion, 
where $q_i$ denotes the $i$th column of $Q$. 
The following is well known \cite[Thm.~2.1.15, Cor.~2.1.16]{Muirhead}:
\be\label{eq:Ovol}
 |O(k)|=\frac{2^k \pi^{\frac{k^2+k}{4}}}{\Gamma(k/2)\Gamma((k-1)/2)\cdots\Gamma(1/2)}
  = \frac{2^k \pi^{\frac{k^2}{2}}}{\Gamma_k\left(\frac{k}{2}\right)}, \quad
  |S(k,m)|=\frac{2^k \pi^{\frac{km}{2}}}{\Gamma_k\left(\frac{m}{2}\right)}. 
\ee
In particular, we note that 
$|S^{k-1}| = 2\pi^{\frac{k}{2}}/\Gamma\left(\frac{k}{2}\right)$. 


The unitary group 
$U(k) :=\{Q\in\C^{k\times k} \mid Q^*Q=\textbf{1}\}$ 
is a compact smooth submanifold of $\C^{k\times k}$. 
Its tangent space at $\textbf{1}$ is given by the space of skew-hermitian matrices
$$
 T_{\textbf{1}} U(k) = \{A\in\C^{k\times k} \mid A^* + A = 0 \}.
$$
We define an inner product on $T_{\textbf{1}}U(k)$ 
by setting for $A=[a_{ij}], B=[b_{ij}]$: 
\be\label{eq:ug-can-metric} 
\langle A,B\rangle := \sum_{i} \Im(a_{ii}) \Im(b_{ii}) 
  + \frac12 \sum_{i\ne j} a_{ij} \bar{b}_{ij} 
  = \sum_{i} \Im(a_{ii}) \Im(b_{ii})  + \sum_{i< j} a_{ij} \bar{b}_{ij} ,
\ee 
where $\Im(z)$ denotes the imaginary part of $z\in\C$. 
We extend this to a Riemannian metric on $U(k)$
by requiring that the left-multiplications are isometries 
and call the resulting metric the canonical one. 
It is important to realize that this metric is essentially different from the 
Riemannian metric on $U(k)$ that is induced by the Euclidean metric 
of $\C^{k\times k}$.  
(The reason is the contribution in~\eqref{eq:ug-can-metric}
from the imaginary elements on the diagonal; 
in the analogous situation of the orthogonal group, the Riemannian metrics 
differ by a constant factor only.) 
A useful property of the canonical metric (whose induced volume form is proportional to the Haar measure) is that for 
$v$ in the unit sphere $S(\C^k)$ of $\C^k$, the orbit map 
$U(k) \to S(\C^k),\, Q\mapsto Qv$ 
is a Riemannian submersion. 
This implies $|U(k)| = |U(k-1)|\cdot |S(\C^k)|$ 
and we obtain from this: 
\be\label{eq:Uvol}
|U(k)| = \frac{2^k \pi^{\frac{k^2+k}{2}}}{\prod_{i=1}^{k-1} i!}  .
\ee
We also state the following well known formulas:
\be\label{eq:Gr-vol}
 |G(k,n)| = \frac{|O(n)|}{|O(k)|\cdot |O(n-k)|}, \quad  |G_\C(k,n)| = \frac{|U(n)|}{|U(k)|\cdot |U(n-k)|} .
\ee

\subsection{Some useful estimates}\label{se:useful-estim}

We define $\rho_k :=\E\|X\|$, where $X\in\R^k$ is a standard normal Gaussian $k$-vector. 
It is well known that:
\be\label{eq:def-rho}
 \rho_k = \frac{\sqrt{2}\, \Gamma(\frac{k+1}{2})}{\Gamma(\frac{k}{2})} .
\ee
E.g., $\rho_2= \sqrt{\frac{\pi}{2}}$. 
For the following estimate see \cite[Lemma 2.25]{BuCu}: 
\be\label{eq:rho-estim}
 \sqrt{\frac{k}{k+1}}\cdot \sqrt{k} \ \le\ \rho_k \ \le\  \sqrt{k} .
\ee
Hence $\rho_k$ is asymptotically equal to $\sqrt{k}$. 

\begin{lemma}\label{le:volG}
For fixed $k$ and $n\to\infty$ we have:
\begin{eqnarray*}
\log |G(k,n)| &=& -\frac12 kn \log n + \frac12 kn \log (2e\pi)+\mathcal{O}(\log n),\\
\log |G_\C(k,n)| &=& -kn\log n + kn\log (e\pi)+\mathcal{O}(\log n),\\
\log \deg G_\C(k,n) &=& kn\log k + \mathcal{O}(\log n). 
\end{eqnarray*}
\end{lemma}


\begin{proof}
By \eqref{eq:Gr-vol} we have $|G(k,n)| = \frac{|O(n)|}{|O(k)|\cdot |O(n-k)|}$.
From \eqref{eq:Ovol} we get:
\be 
\frac{|O(n)|}{|O(n-k)|}
  = 2^k \pi^{\frac{kn}{2}+\frac{k-k^2}{4}} 
      \cdot \prod_{j=n-k+1}^n \frac{1}{\Gamma(j/2)} .
\ee
Consequently, using the asymptotic 
$\log\Gamma(x)=x\log x-x+\mathcal{O}(\log x)$ as $x\to \infty$, we get:
\begin{align} 
 \log\frac{|O(n)|}{|O(n-k)|} 
        & = \frac12 kn \log\pi-k\log\Gamma(n/2)+\mathcal{O}(\log n).\\
	& = \frac12 kn \log\pi-k\left(\frac{n}{2}\log\left(\frac{n}{2}\right)-\frac{n}{2}\right)+\mathcal{O}(\log n)\\
		&=\frac{kn}{2} \log\pi - \frac12 kn \log n + \frac12 kn\log 2 +\frac12 kn+\mathcal{O}(\log n),
\end{align}
from which the assertion on $\log |G(k,n)|$ follows.

For $|G_\C(k,n)|$ we argue similarly. Using \eqref{eq:Uvol}, we get: 
\be 
\frac{|U(n)|}{|U(n-k)|}
  = 2^k \pi^{kn + \frac{k-k^2}{2}} \prod_{i=0}^{k-1} \frac{1}{(n-k+i)!} .
\ee
and the assertion on $\log |G_\C(k,n)|$ follows. 

The third assertion follows easily from the formula for $\deg G_\C(k,n)$ in Corollary~\ref{cor:degGC}.  
\end{proof}

\section{Integral geometry of real Grassmannians} 

\subsection{Joint density of principal angles  between random subspaces}

Let $A\simeq\R^k$ and $B\simeq\R^\ell$ be two subspaces of $\R^n$ with corresponding orthonormal bases 
$(a_1, \ldots, a_k)$ and $(b_1, \ldots, b_\ell)$. By slightly abusing notation we denote by 
$A=[a_1, \ldots, a_k]$ and $B=[b_1, \ldots, b_\ell]$ 
also the matrices with the columns $a_i$ and $b_j$, respectively.
By the singular value decomposition theorem, there exist $U\in O(k)$ and $V\in O(\ell)$ such that:
\be\label{eq:principal} 
 U^TA^TBV=\left[
	\begin{array}{cc}
	D &0 \\ 
	0 &0
	\end{array}\right]
\ee
where $D=\textrm{diag}(\sigma_1, \ldots, \sigma_r)$ and $r=\min\{k,\ell\}$; 
we adopt the convention $1\geq \sigma_1\geq\cdots\geq \sigma_r\geq 0$.
We recall that the {\em principal angles} $0\leq\theta_1\leq \cdots\leq \theta_r\leq \pi/2$ between $A$ and $B$ 
are defined by
$\sigma_i=\cos \theta_i$ for $i=1,\ldots, r$. 
Clearly, they are independent of the choice of the orthonormal bases. 
We note that $m:=\dim(A\cap B)$ equals the number of zero principal angles between $A$ and $B$.


\begin{lemma}\label{le:bases}
There exist orthonormal bases $(a_1, \ldots, a_k)$ and $(b_1, \ldots,
b_\ell)$ of $A$ and $B$, respectively, such that
$a_i=b_i$ for $i\le m$ and 
$\langle a_i, b_j\rangle =\delta_{ij}\cos \theta_i$ 
for all $i$ and $j$.
\end{lemma}

\begin{proof}
Put $C:=A\cap B$.
If $C=0$, then the columns of the matrices $AU$ and $BV$ satisfy the desired property. 
In the general case, we apply this argument to the orthogonal complement 
of~$C$ in $A$ and $B$. 
\end{proof}


The following theorem is a generalization of \cite[Equation~(3)]{AEK}
and can be found in ~\cite[Appendix D.3]{amelunxen-diss}. 
We present a different and considerably shorter proof in Appendix~\ref{app:density}.
	
\begin{thm}
\label{thm:density}
Let $k\le\ell$ and $k+\ell \le n$. We fix an $\ell$-dimensional subspace $B\subseteq \R^n$ 
and sample $A\in G(k,n)$ with respect to to the uniform distribution.
Then the joint probability density of the principal angles $0\leq \theta_1\leq \cdots\leq \theta_k\leq \pi/2$ between $A$ and $B$ 
is given by
\be\label{eq:density} 
 p(\theta_1, \ldots, \theta_k) 
  = c_{k, l, n}\prod_{j=1}^k(\cos \theta_j)^{l-k}(\sin \theta_j)^{n-l-k}\prod_{i<j}
    \left((\cos \theta_i)^2-(\cos \theta_j)^2\right) ,
\ee
where
\be\label{eq:gengamma} 
c_{k,l,n} :=\frac{2^k \pi^\frac{k^2}{2} \Gamma_{k}\left(\frac{n}{2}\right)}
   {\Gamma_{k}\left(\frac{k}{2}\right)\Gamma_{k}\left(\frac{l}{2}\right)\Gamma_{k}\left(\frac{n-l}{2}\right)}.
\ee
\end{thm}
\begin{example}
For a fixed $2$-plane $B\subseteq \R^4$, the joint density of the principal angles $0\leq\theta_1\leq\theta_2\leq \pi/2$ between $B$ 
and a uniform $A\in G(2, 4)$ is given by $p(\theta_1, \theta_2)=2((\cos\theta_1)^2 -(\cos\theta_2)^2)$ and we can easily verify that
\be 
  \int_{0}^{\pi/2}\int_{\theta_1}^{\pi/2}2\left((\cos\theta_1)^2-(\cos\theta_2)^2\right)d\theta_2d\theta_1=1.
\ee
More generally, the joint density $p(\theta_1, \theta_2)$ of the principal angles $0\leq\theta_1\leq \theta_2\leq \pi/2$ 
between a random $2$-plane in $G(2, n+1)$ and a fixed subspace of $\R^{n+1}$ of dimension $n-1$ is given by
\be 
  p(\theta_1, \theta_2)=(n-1)(n-2)(\cos \theta_1)^{n-3}(\cos \theta_2)^{n-3} \left((\cos \theta_1)^2-(\cos \theta_2)^2\right).
\ee
\end{example}

\subsection{Subvarieties of Grassmannians with transitive action on tangent spaces}
\label{se:TATS}

The orthogonal group $O(n)$ acts transitively and isometrically on $G(k,n)$. 
For an element $g\in O(n)$, let $g_*\colon G(k,n)\to G(k,n)$ denote the corresponding action and 
$D_A g_*\colon T_A G(k,n)\to T_{g_*(A)} G(k,n)$ its derivative at $A\in G(k,n)$. 
See \S\ref{sec:semialg} for the definition of generic points. 

\begin{defi}\label{def:trans-act}
Let $\cZ\subseteq G(k,n)$ be a semialgebraic set.  
We say that $\cZ$ has {\em transitive action on its tangent spaces} 
if, for generic regular points $A_1,A_2$ of $\cZ$, there exists $g\in O(n)$ such that 
$g_*(A_1)=A_2$ and $D_{A_1} g_*(T_{A_1}\cZ) = T_{A_2}\cZ$. 
\end{defi} 


\begin{example}\label{ex:Glines}
Consider the subvariety 
$\Omega := \{ A\in G(k,n) \mid \R^{k-1}\subseteq A \subseteq \R^{k+1} \}$, 
which is isomorphic to a real projective line. For determining the tangent space 
of $\Omega$ at $A$ we can assume w.l.og.\ that $A=\R^k$. 
Then the tangent vectors $\alpha\in\Hom(A,A^\perp)$ to $\Omega$ are 
characterized by $\alpha(\R^{k-1})= 0$ and $\alpha(e_k) \in \R e_{k+1}$, so that 
$T_A\Omega = \R\, e_k \ot e_{k+1}$. This implies that $\Omega$ has 
transitive action on its tangent spaces.\end{example}

Tangent spaces of Grassmannians have a product structure: in fact
we have a canonical isometry (cf.\ \S\ref{se:GrassR}):
\be 
 T_A G(k,n) \simeq \Hom(A,A^\perp) \simeq A\ot A^\perp 
\ee
and note that $\dim A =k$ and $\dim A^\perp = n-k$. 

We will focus on a special type of hypersurfaces, whose normal spaces exploit 
the product structure of $T_A G(k,n)$.

\begin{defi}\label{re:coisotropic}
A semialgebraic subset $\cM\subseteq G(k,n)$ of codimension one 
is called a {\em coisotropic hypersurface}, 
if, for generic points $A\in\cM$, 
the normal space $N_A \cM$ is spanned by a rank one vector 
(see~\cite[\S4.3]{GKZ} and \cite{kohn:16}). 
\end{defi}

\begin{lemma}\label{le:factor-trans}
A coisotropic hypersurface
has transitive action on its tangent spaces. 
\end{lemma}

\begin{proof}\pb 
Let $A_1,A_2 \in\cM$ be regular points. 
By assumption, there are unit length vectors $u_i \in A$ 
and $v_i \in A^\perp$ for $i=1,2$ such that
$N_{A_1}\cM =\R u_1\ot v_1$ and $N_{A_2}\cM = \R u_2\ot v_2$.
There is $g\in O(n)$ such that 
$g(A_1) = A_2$, $g(u_1)= u_2$, and $g(v_1)=v_2$. 
Let $g_*\colon G(k,n)\to G(k,n)$ denote the multiplication with $g$. 
Then $D_{A_1}g_*$ maps $N_{A_1}\cM$ to $N_{A_2}\cM$,  
and hence it maps $T_{A_1}\cM$ to $T_{A_2}\cM$ 
(compare \S\ref{se:GrassR}).
\end{proof}

\begin{remark}
Proposition \ref{pro:Chow-TA} below implies that codimension one Schubert varieties 
(and more generally Chow hypersurfaces) have transitive actions on their tangent space.
One can show that the same is true for all Schubert varieties. \pb 
This shows that the method of this paper in principle extends to include random intersections of general Schubert varieties, 
as we plan to discuss in a future work.
\end{remark}

The integral geometry formula to be discussed involves a certain average scaling factor between 
random subspaces, that we introduce next. 

\subsection{Relative position of real subspaces}

The relative position of two subspaces $V,W$ 
of a Euclidean vector space~$E$ can be quantified by a volume like quantity that we define now. This quantity is crucial in the study of integral geometry in homogeneous spaces (see \cite{Howard}).

\subsubsection{Relative position of two real subspaces}
\label{se:RP-2}

Suppose that $E$ is a Euclidean vector space of dimension~$n$.
We have an induced inner product on $\Lambda(E)$, cf.~\S\ref{se:GrassR}.  
Let $V,W$ be linear subspaces of $E$ of dimensions $k,m$, respectively, 
such that $k+m\le n$. 
We define the following quantity:
\begin{equation} \label{eq:def-sigma}
 \s(V,W) :=\|v_1\wedge\ldots \wedge v_k\wedge w_1\wedge\ldots\wedge w_m \|,
\end{equation}
where $v_1,\ldots,v_k$ and $w_1,\ldots,w_m$ are orthonormal bases 
of $V$ and $W$, respectively. (This is clearly independent of the 
choice of the orthonormal bases.) 
Note that if $k=m$ and oriented versions of $V$ and $W$ are 
interpreted as simple $k$-vectors in $\Lambda^k(E)$, then 
$\s(V,W)$ is the sine of the angle between $V$ and $W$ (viewed as vectors in the Euclidean space $\Lambda^k(E)$).
So the quantity $\s(V,W)$ in a sense measures the relative position of $V,W$. 
We can make this more precise by expressing $\s(V,W)$ in terms 
of the principal angles between $V$ and $W$. 

\begin{lemma}\label{le:pos-subspace-R} 
The following is true:
\begin{enumerate}
\item $\s(W,V)=\s(V,W)$.
\item $\s(V,W) = \sin\theta_1\cdot\ldots\cdot\sin\theta_k$ if $k\le m$ and 
the $\theta_i$ are the principal angles between $V$ and~$W$. 
\item $0\le\s(V,W)\le 1$. 
\item We have $\s(V,W)=0$ iff $V\cap W\ne 0$ and 
$\s(V,W)=1$ iff $V$ and $W$ are orthogonal.
\end{enumerate}
\end{lemma}

\begin{proof}
The first assertion is obvious. For the second we use the well known fact 
(cf. \cite[Corollary ~2.4]{bu-co-is}) that there is an orthonormal basis 
$e_1,f_1,\ldots,e_k,f_k,g_1,\ldots,g_{n-2k}$ of $E$ such that 
$V$ is the span of $e_1,\ldots,e_k$ and 
$$
 W= \mathrm{span}\{e_1\cos\theta_1 + f_1\sin\theta_1,\ldots,
        e_k\cos\theta_k + f_k\sin\theta_k,g_1,\ldots,g_{m-k}\} .
$$
This implies the second claim. 
The remaining assertions follow immediately from this. 
\end{proof}


\subsubsection{Relative position of two complex subspaces}
\label{se:CP-2} 

We briefly discuss how to extend the previous discussion to complex spaces.
Suppose that $F$ is a hermitian vector space with $n=\dim_\C F$ and 
$V,W\subseteq F$ are $\C$-subspaces of $\C$-dimensions $k$ and $m$, respectively. 
We note that an analogue of Lemma~\ref{le:pos-subspace-R} holds. 
As in~\eqref{eq:def-sigma} we define 
$\s_\C(V,W):= \|v_1\wedge\ldots \wedge v_k\wedge w_1\wedge\ldots\wedge w_m \|$ 
where $v_1,\ldots,v_k$ and $w_1,\ldots,w_m$ are orthonormal bases of $V$ and $W$, 
respectively. On the other hand, we can view $F$ as a Euclidean vector space 
of dimension $2n$ with respect to the induced inner product 
$\langle v,w\rangle_\R := \Re \langle v,w\rangle$. 
Then $v_1,iv_1,\ldots,v_k,iv_k$ is an orthonormal basis of $V$, seen as 
an $\R$-vector space (note that $v$ and $iv$ are orthogonal for all $v\in F$).  
We conclude that 
\begin{equation}\label{eq:sC-q} 
\s_\C(V,W) = \s(V,W)^2 .
\end{equation}



\subsubsection{Average relative position of two subspaces in a tensor product}
\label{se:av-pos-tp}
We assume now that $E$ is a product $E=E_1 \ot E_2$ of Euclidean vector spaces $E_i$
such that $\langle x\ot y, x'\ot y'\rangle = \langle x,x' \rangle \langle y, y' \rangle$
for $x,x' \in E_1$, $y,y'\in E_2$.   
The product of orthogonal groups $K :=O(E_1)\ti O(E_2)$ acts isometrically on $E$.
Clearly, $K$ is a compact group with a uniform invariant measure.

If $V$ and $W$ are given linear subspaces of $E$, we measure their average 
relative position by moving $V$ and $W$ with uniformly random transformations $k_1,k_2\in K$ 
and taking the average value of $\sigma(k_1V,k_2W)$. Of course, it suffices to move 
one of the subspaces. 

\begin{defi}\label{def:av-sigma}
Let $V$ and $W$ be linear subspaces of $E$. 
The {\em average scaling factor} of $V$ and~$W$ 
is defined as $\as(V,W) := \E_{k\in K} \s(kV,W)$, 
taken over a uniform $k\in K=O(E_1)\times O(E_2)$. 
\end{defi} 

From the definition it is clear that $\as(V,W)$ only depends on the 
$K$-orbits of $V$ and $W$, respectively.

\begin{remark}\label{re:factoring-spaces}
In many cases, $V$ and $W$ factor themselves into subspaces as 
$V=V_1\ot V_2$ and $W=W_1\ot W_2$. 
In this case, it is clear that $\as(V,W)$ only depends on the dimensions 
of the involved spaces $V_i,W_i,E_i$, $i=1,2$.  
\end{remark}

\subsection{Intersecting random semialgebraic subsets of Grassmannians} 

We now provide tailor made formulations of general results in
Howard~\cite{Howard} in the situation of real Grassmannians. 

\subsubsection{Intersecting two random semialgebraic sets} 
We begin with a general observation. 
Let $A\in G(k,n)$ be the subspace spanned by the first $k$ 
standard basis vectors, $A:=e_1\wedge\ldots\wedge e_k\in G(k,n)$, 
and note that the tangent space 
$E := T_A G(k,n) \simeq A \ot A^\perp$ can be seen as a tensor product of vector spaces. 
The stabilizer subgroup of $A$ in $O(n)$ can be identified with 
$K:=O(A)\times O(A^\perp)$ and acts on $T_A G(k,n)$. 
Recall that 
$D_A g_*\colon T_A G(k,n)\to T_A G(k,n)$ denotes the derivative of 
the action $g_*\colon G(k,n)\to G(k,n)$ induced by $g\in O(n)$. 
Moreover, $N_A \cM\subseteq T_A G(k,n)$ denotes the normal subspace 
of a semialgebraic subset $\cM\subseteq G(k,n)$ at a regular point $A$. 

\begin{lemma}\label{le:as-welldefined}
Suppose that $\cM$ is a semialgebraic subset of $G(k,n)$
with transitive action on its tangent spaces. 
Let $B_i\in\cM$ be generic regular points of $\cM$ and 
$g_i\in O(n)$ such that $B_i$ is mapped to~$A$, 
for $i=1,2$. 
Then $D_{B_1} {g_1}_* (N_{B_1} \cM)$ and $D_{B_2} {g_2}_* (N_{B_1} \cM)$ 
are subspace of $T_A G(k,n)$ that lie in the same $K$-orbit.
\end{lemma}

\begin{proof}
There is $h\in O(n)$ such that $h B_1 = B_2$ and 
$D_{B_1} h_*$ maps $T_{B_1} \cM$ to $T_{B_2} \cM$, 
since $\cM$ has transitive action on its tangent spaces. 
We have $g_1 = k g_2 h$, where
$k:= g_1 (g_2 h)^{-1} \in K$.
Moreover, 
$$
 D_{B_1} {g_1}_* (T_{B_1}\cM) = D_A k_* (D_{B_2} {g_2}_* (D_{B_1} h_* (T_{B_1} \cM))) 
  = D_A k_* (D_{B_2} {g_2}_* (T_{B_2} \cM)) .
$$ 
Therefore, $D_A k_*$, which can be seen as an element of $K$, 
maps $D_{B_2} {g_2}_* (N_{B_2}\cM)$ to $D_{B_1} {g_1}_* (N_{B_1}\cM)$.  
\end{proof}

We can now define an important quantity that enters 
the integral geometry formula to be discussed. 

\begin{defi}\label{def:av-SF}
Let $\cM$ and $\cN$ be semialgebraic subsets of $G(k,n)$ 
with transitive action on their tangent spaces. 
We pick a regular point $B$ of $\cM$, a regular point $C$ of $\cN$, 
and $g,h\in O(n)$ such that $g(B)=A$ and $h(C)=A$.  
Then we define 
$\as(\cM,\cN) := \as\Big(D_B g_* (N_B\cM), D_C h_* (N_C\cN)\Big)$.   
\end{defi}

It is important to note that $\as(\cM,\cN)$ does not depend on the 
choice of the regular points $B,C$ and the maps $g,h$ 
and hence is well defined. This is a consequence of Lemma~\ref{le:as-welldefined}.

\begin{remark}\label{re:on-as}
Suppose that the normal spaces $N_B \cM$ of $\cM$ 
at regular points $B$ of $\cM$ are tensor products 
$N_B \cM = V_B \ot W_B$, 
with subspaces $V_B\subseteq B$ and $W_B\subseteq B^\perp$ 
of a fixed dimension $\mu_1$ and~$\mu_2$ respectively. 
Similarly, we assume that at regular points $C$ of $\cN$ 
we have $N_C \cN = V'_C \ot W'_C$ with subspaces  
$V'_C\subseteq C$ and $W'_C\subseteq C^\perp$ of 
fixed dimensions $\nu_1$ and $\nu_2$, respectively. 
Then $\as(\cM,\cN)$ only depends on the 
dimensions $\mu_1,\mu_2,\nu_1,\nu_2$ and $k,n$; 
see Remark~\ref{re:factoring-spaces}. 
For instance, this assumption is satisfied if we are 
considering semialgebraic subsets of the projective 
space $\RP^{n-1} =G(1,n)$.
\end{remark}

The following result follows by combining Howard~\cite[Theorem 3.8]{Howard}
with Lemma~\ref{le:factor-trans}. 

\begin{thm}\label{th:is-2-howard}
Let $\cM$ and $\cN$ be semialgebraic subsets  of $G(k,n)$ of codimension $\mu$ and $\nu$ respectively, 
with transitive action on their tangent spaces.
Then we have: 
$$
 \E_{g\in O(n)} |\cM\cap g\cN|  
  = \overline\s(\cM, \cN) \cdot \frac{1}{|G(k,n)|}\cdot |\cM|\cdot |\cN| .
$$
Here, $|\cM\cap g\cN|$ denotes the volume in the expected dimension 
$k(n-k) -\mu-\nu$, $|\cM|$ and $|\cN|$ denote 
the volumes in the dimensions of $\cM$ and $\cN$, respectively. 
\end{thm}


\begin{remark}
Howard~\cite[Theorem 3.8]{Howard} states his result under the assumption of 
two smooth compact submanifolds of $G(k,n)$ (possibly with boundary).  
However, the compactness assumption is not needed so that we can apply his result to the regular loci $\Reg(\cM)$ and $\Reg(\cN)$
of $\cM$ and $\cN$, respectively. For this, 
note that the singular locus $\Sing(\cM)$ has dimension strictly 
smaller than $\cM$ (see \S\ref{sec:semialg}). Moreover, 
$\E_{g\in G} |\cM\cap g\cN| = \E_{g\in G} |\Reg(\cM)\cap g\Reg(\cN)|$
since, almost surely, the dimension of the difference 
$(\cM\cap g\cN) \setminus (\Reg(\cM)\cap g\Reg(\cN))$ is strictly smaller than 
the typical dimension of $\cM\cap g\cN$.
\end{remark}

\begin{example}\label{ex:IGproj-space}
Here is a typical application of Theorem~\ref{th:is-2-howard} 
in the special case $G(1,n)=\RP^{n-1}$. It is clear that 
any semialgebraic subset $\cM$ of $G(1,n)$ has transitive action on its 
tangent spaces. Let $m:=\dim\cM$. 
Applying Theorem~\ref{th:is-2-howard} 
to $\cM$ and $\cN := \RP^{n-m-1}$, we obtain
\be 
  \E_{g\in O(n)} \#(\cM\cap g\RP^{n-m-1})  
  = \overline\s \cdot \frac{1}{|G(1,n)|}\cdot |\cM|\cdot |\RP^{n-m-1}| 
\ee
where the average scaling factor $\overline\s := \overline\s(\cM, \cN)$
only depends on $m$ and $n$; see Remark~\ref{re:on-as}. 
In particular, this gives for $\cM= \RP^{m}$, 
\be 
  \E_{g\in O(n)} \#(\RP^{m}\cap g\RP^{n-m-1}) 
  = \overline\s \cdot \frac{1}{|G(1,n)|}\cdot |\RP^{m}|\cdot |\RP^{n-m-1}| ,
\ee
the left-hand side of which clearly equals one. 
Solving this for $\overline\s$ and plugging it in in the previous equation gives 
the well known formula (compare \cite[A.55]{BuCu})
\be\label{eq:IG-proj} 
  \E_{g\in O(n)} \#(\cM\cap g\RP^{n-m-1})  
  = \frac{|\cM|}{|\RP^{m}|} . 
\ee
\end{example}

\subsubsection{Intersecting many random coisotropic hypersurfaces} 

The previous discussion extends to the situation of intersecting many semialgebraic sets; 
we focus here on 
the special case of intersecting coisotropic hypersurfaces 
in the Grassmannian, which is particularly simple. The general case is discussed in the Appendix. 

Let $E$ be a Euclidean vector space of dimension~$n$. 
Let $V_1,\ldots,V_s$ be linear subspaces of the dimensions $m_1,\ldots,m_s$, 
respectively, such that $\sum_{j=1}^s m_j \le n$. 
We define the quantity
\begin{equation} \label{eq:def-sigma-many}
 \s(V_1,\ldots,V_s) :=\|v_{11}\wedge\ldots \wedge v_{1m_1}\wedge\ldots
     \wedge v_{s1}\wedge\ldots\wedge v_{sm_s} \|,
\end{equation}
where $v_{j1},\ldots,v_{jm_j}$ are orthonormal bases of $V_j$. 
(This is clearly independent of the choice of the orthonormal bases.) 
Note that $\s(V_1,\ldots,V_s) \le 1$.
In the special case where $m_j=1$ for all $j$, we can interpret 
$\s(V_1,\ldots,V_s)$ as the volume of the parallelepiped spanned by 
the unit vectors $v_{11},\ldots,v_{s1}$. 
In order to deal with the coisotropic case, we introduce the following definition.
\begin{defi}\label{def:alpha}
For $k,m\ge 1$ we define the {\em (real) average scaling factor} 
$\alpha(k,m)$ as 
\be
 \alpha(k,m) := \E \| (u_1\ot v_1) \wedge \ldots \wedge (u_{km} \ot v_{km}) \|,
\ee
where $u_j \in S(\R^k)$ and $v_j \in S(\R^{m})$, for $1\le j\le km$, 
are independently and uniformly chosen at random. 
\end{defi}

Note that $\alpha(k,m)$ is nothing but the generalization to the case of many coisotropic submanifolds 
of the quantity $\overline\sigma(\cM, \cN)$ introduced in Definition~\ref{def:av-SF}. 
It has the special property that in the coisotropic case, it does not depend on the choice of the hypersurfaces; 
compare Remark~\ref{re:on-as}.  

We will need the following result in the spirit of Theorem~\ref{th:is-2-howard} 
about intersecting codimension many coisotropic hypersurfaces in $G(k,n)$. 

\begin{thm}
\label{thm:IGF}
Let $\Hy_1, \ldots, \Hy_{N}$ be coisotropic hypersurfaces of $G(k,n)$, 
where $N :=k(n-k)$. 
Then we have:
\be 
\E_{(g_1,\ldots,g_N)\in O(n)^N} 
\#\left(g_1\Hy_1\cap\cdots\cap g_N \Hy_N \right)dg_1\cdots dg_N 
 =  \alpha(k,n-k) \cdot |G(k,n)| \cdot\prod_{i=1}^{N} \frac{|\Hy_i|}{|G(k,n)|} .
\ee
\end{thm}

This theorem is a consequence of a generalized Poincar\'e formula in homogeneous spaces,  
stated as Theorem~\ref{thm:GIGF} in the Appendix~\ref{se:GIGF}. 
Its statement and proof is very similar to Howard~\cite[Theorem 3.8]{Howard}.
Since we have been unable to locate this result in the literature, 
we provide a proof in the Appendix.


\section{Probabilistic enumerative geometry} 

\subsection{Special Schubert varieties}\label{se:RChow}

By the {\em special real Schubert variety} 
associated with a linear subspace $B\in G(n-k,n)$
we understand the subvariety 
\be 
 \Omega(B) := \big\{ A \in G(k,n) \mid A \cap B \ne 0\big\} .
\ee 
If the dependence on~$B$ is not relevant, which is mostly the case, 
we write $\Sigma(k,n) := \Omega(B)$.
The Schubert cell\footnote{We observe, as pointed out by an anonymous referee, 
that in general $e(B)$ is not a topological cell, but instead a vector bundle over a product of Grassmannians, 
see \cite[Theorem 2.1]{WongY}.} 
associated with~$B$ is the open subset
\be 
 e(B) := \big\{ A \in G(k,n) \mid \dim(A \cap B) = 1\big\}
\ee 
of $\Omega(B)$.
It is a well known fact that $e(B)$ is the regular locus of $\Omega(B)$, 
e.g., see \cite{milnor-stasheff, WongY}. 
Moreover, $e(B)$ is a hypersurface in $G(k,n)$. 
Its normal spaces can be described by the following result, 
which is a special case of~\cite[\S2.7]{sottile-pieri}. 

\begin{lemma}\label{le:prop-Sigmakn}
Let $A\in e(B)$ and $a_1,a_2,\ldots,a_k$ be an orthonormal basis of $A$ such that 
$A\cap B =\R a_1$. Moreover, let 
$(A+B)^\perp = \R f$ with $\|f\| =1$. 
Then $\alpha\in \Hom(A,A^\perp)$ 
lies in the normal space $N_A \Omega(B)$ iff 
$\alpha(a_1)\in\R f$ and $\alpha(a_i) = 0$ for $i>1$.
In other words, 
$f\wedge a_2\wedge\cdots\wedge a_k$
spans the normal space $N_A \Omega(B)$, when interpreted 
as a subspace of $\Lambda^k \R^n$ as in~\S\ref{se:GrassR}. 
\end{lemma}

\subsection{Volume of special Schubert varieties}
\label{se:vol-SSV} 

We can determine now the volume of the special Schubert  varieties $\Sigma(k,n)$.
It is remarkable that the result is completely analogous to the corresponding result 
\eqref{eq:HlinC} over~$\C$. 

\begin{thm}\label{prop:vspe}
The volume of the special Schubert variety $\Sigma(k,n)$ satisfies:
\be 
 \frac{|\Sigma(k,n)|}{|G(k,n)|} 
   =\frac{\Gamma\left(\frac{k+1}{2}\right)}{\Gamma\left(\frac{k}{2}\right)} \cdot
      \frac{\Gamma\left(\frac{n-k+1}{2}\right)}{\Gamma\left(\frac{n-k}{2}\right)} 
   = \pi \cdot \frac{|\RP^{k-1}|}{|\RP^k|} \cdot \frac{|\RP^{n-k-1}|}{|\RP^{n-k}|} .
\ee
\end{thm}

\begin{proof}
The isometry $G(k,n)\to G(n-k,n),\, A\mapsto A^\perp$ maps 
$\Omega(B)$ to $\Omega(B^\perp)$. 
Therefore, we may assume that $k\le n/2$ without loss of generality.  

The proof relies on a general principle.
Let $M$ be a compact Riemannian manifold, 
$H$ be a smooth hypersurface of $M$, 
and $K$ be a nonempty compact subset of $H$.  
We define the {\em $\epsilon$-tube} $\Tu(K, \epsilon)$ 
of~$K$ in $M$ by 
\be 
 \Tu(K, \epsilon) := \big\{\exp_x(\nu) \mid x\in K,\ \nu\in N_x H,\ \|\nu\| \le \e\big\} ,
\ee 
where $\exp_x\colon T_xM\to M$ denotes the exponential map of $M$ at~$x$ 
and $N_x H$ is the (one dimensional) orthogonal complement of the tangent space $T_xH$ in $T_xM$;  
cf.~\cite[\S21.2]{BuCu} and~\cite{spivak-1}.
Let $m:=\dim H$. 
The $m$-dimensional volume $|K|$ of $K$ can be computed from the 
$m+1$-dimensional volumes of the tubes $\Tu(K,\epsilon)$ as follows (see~\cite{gray:81}):
\be
  \label{eq:volumes} |K| = \lim_{\epsilon\to 0}\,\frac{1}{2\epsilon}\, |\Tu(K,\epsilon)| .
\ee
We shall apply this formula to a Schubert cell $H=e(B)$ for a fixed $B\in G(n-k,n)$
in the Grassmann manifold $M=G(k,n)$ and $K$ belonging to an increasing family of compact sets exhausting $e(B)$ (see below).

Let us denote by $\vartheta_1,\vartheta_2\colon G(k,n)\to \R$ the functions giving, respectively, 
the smallest and the second smallest principal angle between $A\in G(k,n)$ and the fixed~$B$.
We note that 
$e(B) = \{A\in G(k,n) \mid \vartheta_1(A) = 0\}$. 
For $0<\delta<\pi/2$, we consider the following compact subset of $e(B)$: 
\be 
  e(B)_\delta := \big\{ A\in G(k,n) \mid  \vartheta_1(A)=0,\ \vartheta_2(A) \ge \delta  \big\} .
\ee
Note that 
$e(B)_{\delta_2}\subseteq e(B)_{\delta_1}$ if $\delta_1<\delta_2$ and $e(B)=\bigcup_{\delta>0}e(B)_\delta$. 
Therefore, we have $|\Omega(B)|=|e(B)|=\lim_{\delta\to 0}|e(B)_\delta|$. 
Combining this with~\eqref{eq:volumes}, we can write the volume of the Schubert variety as 
the double limit
\be
 \label{eq:volume2}|\Omega(B)|=\lim_{\delta\to 0}|e(B)_\delta|\\
 =\lim_{\delta\to 0}\lim_{\epsilon\to 0}\, \frac{1}{2\epsilon}\, |\Tu(e(B)_\delta, \epsilon)| .
\ee
We need the following technical result, whose proof will 
be provided in Appendix~\ref{app:tech-lemma}. 

\begin{lemma}\label{le:tube} 
In the above situation, we have for $0<\e \le \delta< \pi/2$ that 
\be 
 \Tu(e(B)_\delta, \epsilon) = \{ A\in G(k,n) \mid \vartheta_1(A) \le \epsilon,\ \vartheta_2(A) \ge\delta \big\} .
\ee
\end{lemma}

Let now $p(\theta_1,\ldots,\theta_k)$ denote the joint density of 
the (ordered) principal angles $\theta_1\le\cdots\le\theta_k$ 
between $A\in G(k,n)$ and $B$, as in Theorem~\ref{thm:density}.  
Due to Lemma~\ref{le:tube} we can write for $\epsilon\le\delta <\pi/2$, 
\be
  \frac{|\Tu(e(B)_\delta, \epsilon) )|}{|G(k,n)|} 
= \Prob\big\{ A\in G(k,n) \mid \vartheta_1(A)\leq \epsilon,\ \vartheta_2(A) \geq \delta \big\} 
   = \int_{R(\epsilon, \delta)} p(\theta)\, d\theta,
\ee 
where $\Prob$ refers to the uniform distribution on $G(k,n)$ and 
$$
 R(\epsilon, \delta) := \big\{\theta\in\R^k \mid 0\le\theta_1\le\ldots\le\theta_k\le\pi/2,\ 
  0\le\theta_1\le \epsilon,\ \delta\leq \theta_2\big\} .
$$
Thus, from \eqref{eq:volume2} 
we see that
\be 
  \frac{|\Omega(B)|}{|G(k,n)|}
   =\frac{1}{2}\lim_{\delta\to 0}\lim_{\epsilon\to 0}\frac{1}{\epsilon}\int_{R(\epsilon, \delta)} p(\theta)\, d\theta,
\ee 
and we have reduced the problem to the evaluation of the last limit.
We have 
\begin{eqnarray*} 
 \lim_{\delta\to 0}\lim_{\epsilon\to 0} \frac{1}{\epsilon} \int_{R(\epsilon, \delta)} p(\theta)\, d\theta 
 &=&  \lim_{\delta\to 0}\lim_{\epsilon\to 0} \frac{1}{\epsilon} \int_0^\epsilon \left(
  \int_{\theta_1\le\delta\leq\theta_2\le\ldots\le\theta_k\le\pi/2} p(\theta)\,d\theta_2\cdots d\theta_k\right) d\theta_1 \\
 &=&  \lim_{\delta\to 0}\int_{0\le\delta\leq\theta_2\le\ldots\le\theta_k\le\pi/2} p(0,\theta_2,\ldots,\theta_k)\,d\theta_2\cdots d\theta_k\\
  &=& \int_{0\leq\theta_2\le\ldots\le\theta_k\le\pi/2} p(0,\theta_2,\ldots,\theta_k)\,d\theta_2\cdots d\theta_k .
\end{eqnarray*}
Theorem~\ref{thm:density} on the joint density~$p$ gives 
\begin{eqnarray*} 
 p(0,\theta_2,\ldots,\theta_k) 
  &=& c_{k,n-k,n} \cdot \prod_{j=2}^k (\cos \theta_j)^{n-2k} 
          \prod_{j=2}^k \big(1- (\cos\theta_j)^2\big) \prod_{2\le i<j\le k} 
           \left((\cos \theta_i)^2-(\cos \theta_j)^2\right) \\
 &=& \frac{c_{k,n-k,n}}{c_{k-1,n-k+1,n}} \cdot \tilde{p}(\theta_2,\ldots,\theta_k),
\end{eqnarray*}
where $\tilde{p}(\theta_2,\ldots,\theta_k)$ is the joint density of the principal 
angles between a random $(k-1)$-plane and a fixed $(n-k-1)$-plane in $\R^n$. 
We conclude that 
\be
 \frac{|\Omega(B)|}{|G(k,n)|} = \frac12 \frac{c_{k,n-k,n}}{c_{k-1,n-k-1,n}} 
 = \frac{\Gamma\left(\frac{k+1}{2}\right)}{\Gamma\left(\frac{k}{2}\right)} \cdot
      \frac{\Gamma\left(\frac{n-k+1}{2}\right)}{\Gamma\left(\frac{n-k}{2}\right)} ,
\ee
where the last equality follows by a tedious calculation after 
plugging in the formulas for the constants $c_{k,n-k,n}$ and $c_{k-1,n-k-1,n}$
from Theorem~\ref{thm:density}.
\end{proof}


\subsection{Chow hypersurfaces}

We generalize now the definition of special Schubert varieties of codimension one.
Let $X\subseteq \RP^{n-1}$ be a semialgebraic subset of dimension $n-k-1$. 
We associate with $X$ the set of $(k-1)$-dimensional projective subspaces of $\RP^{n-1}$ 
that intersect~$X$.
We interpret this as a subset of $G(k,n)$ and define  
\be 
  Z(X) :=\{A\in G(k, n)\,|\, \PP(A)\cap X\neq \emptyset\}\ .
\ee
In analogy with the situation over $\C$, we call $Z(X)$ the \emph{Chow hypersurface} 
associated with $X$. 
(The fact that $Z(X)$ is a hypersurface, i.e., a semialgebraic subset of codimension one, will be proved 
in Lemma~\ref{pro:TangSpaceChow} below.) 
We note that $Z(X) = \Omega(B)$ if $X$ is projective linear and $B$ the corresponding linear space.

\begin{example}\label{ex:Chow-not-alg}
Let $X\subseteq\RP^3$ be the real twisted cubic, i.e., 
the image of the map 
$\RP^1 \to \RP^3, (s:t)\mapsto (s^3:s^2t:st^2:t^3)$.
We claim that $Z(X)$ is {\em not an algebraic subset} 
of~$G(2,4)$. For seeing this, we intersect the affine part 
$X_{\mathrm{aff}} := \{(t,t^2,t^3) \mid t\in\R\}$ of $X$ with the 
open subset 
$U:= \{L_{u_1 u_2 v_1 v_2} \mid (u_1,u_2,v_1,v_2) \in \R^4 \}\subseteq G(2,4)$ 
of lines 
$$
 L_{u_1 u_2 v_1 v_2} := \left\{ (0,u_1,u_2) + s(1,v_1,v_2) \mid s\in\R\right\} .
$$
We have 
\be
 X_{\mathrm{aff}} \cap U \ne\emptyset \Longleftrightarrow 
\exists t\in\R\quad  
  t^2 -v_1 t -u_1 =0,\ 
  t^3 -v_2 t -u_2 = 0 .
\ee
The right-hand condition is equivalent to 
$v_2^2 + 4u_2 \ge 0, \mathrm{res}=0$, 
where 
$$\mathrm{res} := 
 u_1 v_1^2 v_2 - u_2 v_1^3 - u_1^3 +2 u_1^2 v_2 - 3 u_1 u_2 v_1 -u_1 v_2^2+u_2 v_1 v_2 +u_2^2 = 0,
$$ 
is the resultant of the above quadratic and cubic equation. 
Intersecting with the 2-plane $u_2=v_2=0$ gives 
$\{(u_2,v_2)\in\R^2 \mid v_2^2 + 4u_2 \ge 0 \}$, 
which is not algebraic. 
\end{example}

It is essential that the tangent spaces of Chow hypersurfaces coincide with the 
tangent spaces of closely related special Schubert varieties. 
We postpone the somewhat technical proof to Appendix~\ref{app:Chow}. 

\begin{lemma}\label{pro:TangSpaceChow}
The Chow hypersurface $Z(X)\subseteq G(k,n)$ is semialgebraic of codimension one; 
if $X$ is compact, then $Z(X)$ is compact and if $X$ is connected, then $Z(X)$ is connected.
Moreover, for generic points $A\in Z(X)$ the following is true: the intersection $\PP(A)\cap X$ consists of one point only; 
let us denote this point by $p$; then $p$ is a regular point of $X$ and, denoting by $B\subseteq\R^n$ 
the linear space of dimension $n-k$ corresponding to the tangent space $T_p X$, we have:
\be\label{eq:Tchow} 
 T_A Z(X) = T_A \Omega(B).
\ee
\end{lemma}




The following result makes sure that we can apply the methods from 
integral geometry to Chow hypersurfaces.  

\begin{prop}\label{pro:Chow-TA}
Chow hypersurfaces are coisotropic and hence 
have transitive action on their tangent spaces. 
In particular, this applies to special Schubert varieties of codimension one.
\end{prop} 

\begin{proof}
For codimension one special Schubert varieties, the assertion follows from 
Lemma~\ref{le:prop-Sigmakn}. The assertion on Chow hypersurfaces 
is a consequence of Lemma~\ref{pro:TangSpaceChow}.
\end{proof}


\begin{remark}\label{re:Schubert-hc}\label{pro:HAV-factor}
Special Schubert varieties of higher codimension are defined by 
$\Omega(B,m) := \big\{ A \in G(k,n) \mid \dim(A \cap B)\ge m\big\}$,
where $B$ is an $\ell$-dimensional subspace such that 
$k+\ell -n \le m \le \ell$. The codimension of $\Omega(B,m)$ 
equals $m(n+m-k-\ell)$. 
For incorporating touching conditions, one can 
extend the definition of the Chow hypersurface
and study {\em higher associated semialgebraic sets},  
\be
Z_m(X) :=\{A\in G(k, n)\,|\, \exists p\in\Reg(X)\ 
  \dim T_pX \cap T_p (\PP(A)) \ge m\} .
\ee 
This is done in \cite[\S3.2, \S4.3]{GKZ}, 
\cite{sturmfels:14}, and \cite{kohn:16} 
for {\em complex} projective varieties $X$ 
in the case $\dim X= n-k-1+m$, 
where $Z_m(X)$ typically is a hypersurface.
Note that $Z_0(X) = Z(X)$.  
The proof of Proposition~\ref{pro:Chow-TA}
generalizes in a straightforward way 
to show that $\Omega(B,m)$ and $Z_m(X)$  
have transitive action on their tangent spaces.


\end{remark} 

As an application, we derive the following result that 
relates the volume of a semialgebraic subset $X\subseteq \RP^{n-1}$ 
of dimension $n-k-1$ with the volume of the associated Chow hypersurface $Z(X)\subseteq G(k,n)$.
(This is analogous to equation~\eqref{eq:chow-volratioC} holding over $\C$.) 

\begin{prop}\label{prop:chow}
Let $X$ be a semialgebraic subset of  $ \RP^{n-1}$ of dimension $n-k-1$.
Then we have: 
\be 
  \frac{|Z(X)|}{|\Sigma(k,n)|} =\frac{|X|}{\textrm{\emph{$|\RP^{\dim X}|$}}} .
\ee
\end{prop}

\begin{proof}
By~\eqref{eq:IG-proj} we can express the volume ratio 
$|X|/|\RP^{\dim X}|$ as the expectation 
$\E_{g\in O(n)} \#(X\cap g\RP^{k})$. 
Consider now the subvariety
$\Omega := \{ A\in G(k,n) \mid \R^{k-1}\subseteq A \subseteq \R^{k+1} \}$  
from Example~\ref{ex:Glines}, which has transitive action on 
its tangent spaces. 
For almost all $g$, we have a finite intersection
$X\cap g\RP^{k}=\{p_1,\ldots,p_d\}$. 
On the other hand,
$A\in g\Omega$ implies $A\subseteq g\R^{k+1}$, hence 
$\PP(A)\cap X = \{p_1,\ldots,p_d\}$. 
Therefore, we have for almost all $g$,
\be 
  \#(X\cap g\RP^{k} ) = \#(Z(X) \cap g\Omega) ,
\ee
and it remains to determine the expectation of the right-hand side. 

The Chow hypersurface $Z(X)$ has transitive action on 
its tangent spaces by Proposition~\ref{pro:Chow-TA}. 
Applying Theorem~\ref{th:is-2-howard} to $Z(X)$ and 
$\Omega$, we obtain: 
$$
 \E_{g\in O(n)} \#(Z(X)\cap g\Omega )  
  = \overline\s(Z(X), \Omega) \cdot \frac{1}{|G(k,n)|}\cdot |Z(X)|\cdot |\Omega| .
$$
Applying this formula to a linear space $X$ yields:
\be\label{eq:siint}
 \E_{g\in O(n)} \#( \Sigma(k,n)\cap g\Omega ) 
  = \overline\s(\Sigma(k,n), \Omega)\cdot \frac{1}{|G(k,n)|}\cdot |\Sigma(k,n)|\cdot |\Omega| .
\ee
However, it is clear that the left-hand side equals one. 
According to Definition~\ref{def:av-SF}, the average scaling factor $\overline\s(Z(X), \Omega)$
is defined in terms of tangent (or normal) spaces of $\s(Z(X))$ and $\Omega$. 
Taking into account Lemma \ref{pro:TangSpaceChow}, we see that 
\be  
 \overline\s(Z(X), \Omega)= \overline\s(\Sigma(k,n), \Omega),
\ee
hence we can solve for this quantity and plugging this into \eqref{eq:siint} 
yields the assertion.
\end{proof}

\subsection{Random incidence geometry}\label{se:rand_inc-geo}

%

As already indicated in the introduction, we study the following problem. 
We fix semialgebraic subsets $X_1,\ldots,X_{N} \subseteq \RP^{n-1}$
of the same dimension $\dim X_i = n-k-1$, where $0 < k < n$ and $N:=k(n-k)$ 
and we ask how many $(k-1)$-dimensional projective linear 
subspaces intersect all random translations of the $X_1,\ldots,X_{k(n-k)}$.
More specifically, the task is to determine the expectation
$\E\#\left(g_1Z(X_1)\cap\cdots\cap g_N Z(X_N)\right)$ 
with respect to independent uniformly distributed $g_i\in O(n)$. 
In the case where all $X_i$ are linear subspaces, we call the answer the 
expected degree of $G(k,n)$. 

\begin{defi}\label{def:edeg} 
The {\em expected degree} of $G(k,n)$ is defined as the average number of 
intersection points of $N=k(n-k)$ many random copies of $\Sigma(k,n)$, i.e., 
$$
 \edeg G(k,n) := \E \#(g_1\Sigma(k,n)\cap\ldots \cap g_N \Sigma(k,n))
$$
with respect to independent uniformly distributed $g_i\in O(n)$. 
\end{defi}

We note that $\edeg G(1,n) = 1$. 
In Theorem~\ref{thm:ased} we will provide an asymptotically sharp upper bound 
on $\edeg G(k,n)$. 

Theorem~\ref{thm:IGF} combined with the fact that 
$\Sigma(k,n)$ is coisotropic implies (recall that $\alpha(k,n-k)$ was defined in Definition \ref{def:alpha})
\begin{align} \label{eq:ided}
\edeg(G(k,n))&=\alpha(k,n-k)\cdot |G(k,n)|\cdot\left(\frac{|\Sigma(k,n)|}{|G(k,n)|}\right)^{k(n-k)}\\
&=\alpha(k,n-k)\cdot |G(k,n)|\cdot\left(\frac{\Gamma\left(\frac{k+1}{2}\right)}{\Gamma\left(\frac{k}{2}\right)} 
        \frac{\Gamma\left(\frac{n-k+1}{2}\right)}{\Gamma\left(\frac{n-k}{2}\right)}\right)^{k(n-k)},
\end{align}
where the last line follows from Theorem \ref{prop:vspe}.
Combining Theorem~\ref{thm:IGF} with Proposition~\ref{pro:Chow-TA}
and Proposition~\ref{prop:chow},
we obtain a real version of Theorem~\ref{th:intersect_C}.

\begin{thm}\label{thm:RIG}
Let $X_1, \ldots, X_{N}$ be semialgebraic subsets of $\RP^{n-1}$ of dimension $n-k-1$, 
where $N:=k(n-k)$. 
Then we have
\be 
\E\#\left(g_1Z(X_1)\cap\cdots\cap g_N Z(X_N)\right)
 = \edeg G(k,n)  \cdot \prod_{i=1}^{N} \frac{|X_i|}{|\RP^{n-k-1}|} ,
\ee
with respect to independent uniformly distributed $g_1,\ldots,g_N\in O(n)$. 
\end{thm}

This result allows to decouple a random incidence geometry problem into 
a volume computation in $\RP^{n-1}$ 
and the determination of the expected degree of the Grassmann manifold (the ``linearized'' problem).


\subsection{Classical Schubert calculus revisited}

In Definition~\ref{def:alpha} we defined the {\em real average scaling factor} $\alpha(k,m)$. 
Following the reasonings in \S\ref{se:CP-2}, it is natural to define
its complex variant as follows.

\begin{defi}\label{def:alphaC}
For $k,m\ge 1$ we define the {\em complex average scaling factor}
$\alpha_\C(k,m)$ as 
\be
 \alpha_\C(k,m) := \E \| (u_1\ot v_1) \wedge \ldots \wedge (u_{km} \ot v_{km}) \|^2
\ee
where $u_i \in S(\C^k)$ and $v_i \in S(\C^{m})$ 
are independently and uniformly chosen at random, 
for $i=1,\ldots,km$. 
\end{defi}

Unlike for $\alpha(k,m)$, we can give a closed formula for its 
complex variant $\alpha_\C(k,m)$. 

\begin{prop}\label{pro:alphaC}
We have $\alpha_\C(k,m) = N!/N^N$, where $N=km$.  
\end{prop}

The proof follows immediately from the following lemma.

\begin{lemma}\label{le:Edet2}
Let $Z=[z_{ij}]\in \C^{N\ti N}$ be a random matrix such that 
for all $i,j,k$
$$
 \E \big |z_{ij}|^2 = \frac{1}{N} \quad\mbox{and}\quad
 \E z_{ij} \bar{z}_{ik} = 0 \quad\mbox{if $j\ne k$} . 
$$
Then $\E |\det Z |^2 = N!/N^N$.
\end{lemma}

\begin{proof}
We have 
$$
 |\det Z |^2 = \sum_{\pi,\s\in S_N} \sgn(\pi)\,\sgn(\s) \prod_{i=1}^N z_{i\pi(i)}\, \bar{z}_{i\s(i)} .
$$
Taking expectations yields
$$
 \E |\det Z |^2 
   =\sum_{\pi,\s\in S_N} \sgn(\pi)\,\sgn(\s) \prod_{i=1}^N \E z_{i\pi(i)}\, \bar{z}_{i\s(i)} \\
   = \sum_{\pi\in S_N} \E |z_{i\pi(i)}|^2 = \frac{N!}{N^N} ,
$$
which completes the proof.
\end{proof}

The following theorem is analogous to Theorem~\ref{thm:IGF} and can be 
proved similarly. 
However, we can easily derive this result 
from Theorem~\ref{th:intersect_C},
which even shows that the coisotropy assumption is not needed here.

\begin{thm}\label{thm:IGFC}
Let $\Hy_1, \ldots, \Hy_{N}$ be irreducible algebraic hypersurfaces of $G_\C(k,n)$, 
where $N :=k(n-k)$. 
Then we have:
\be 
 \frac{1}{|U(n)|^{N}} \int_{U(n)^{N}}\#\left(g_1\Hy_1\cap\cdots\cap g_N \Hy_N \right)dg_1\cdots dg_N 
 =  \alpha_\C(k,n-k) \cdot |G_\C(k,n)| \cdot\prod_{i=1}^{N} \frac{|\Hy_i|}{|G_\C(k,n)|} .
\ee
\end{thm}

\begin{proof}
By Theorem~\ref{th:intersect_C},
the left-hand side in the assertion of Theorem~\ref{thm:IGFC} equals
$\deg G_\C(k,n) \cdot\prod_i\rdeg\Hy_i$.
We analyze now the right-hand side. 
We note that 
$|\Hy_i| = \rdeg\Hy_i \cdot |\Sigma_\C(k,n)|$ 
since the proof of Equation~\eqref{eq:chow-volratioC}  
works for any irreducible hypersurface in $G_\C(k,n)$. 
Using \eqref{eq:HlinC} we get 
\begin{equation*}
|\Hy_i| 
 = \rdeg\Hy_i \cdot |G_\C(k,n)| \cdot\frac{|\CP^{N-1}|}{|\CP^{N}|} .
\end{equation*}
Moreover, 
$|G_\C(k,n)| = \deg G_\C(k,n) \cdot |\CP^N|$
by~\eqref{eq:vol-deg-G}.  
Therefore, it suffices to verify that 
$$
 \alpha_\C(k,n-k) = \frac{1}{|\CP^{N}|} \cdot \Big(\frac{|\CP^{N}|}{|\CP^{N-1}|}\Big)^N .
$$
This follows from Proposition~\ref{pro:alphaC}
using $|\CP^N| = \pi^N/N!$.
\end{proof}

This argument showed that complex algebraic geometry 
implies the integral geometry formula of Theorem~\ref{thm:IGFC}.
One can also argue in the reverse direction. 
For instance, we can give a new, integral geometric derivation of 
the formula for $\deg G_\C(k,n)$. 
(This formula was first obtained by Schubert~\cite{schubert:86}.)

\begin{cor}\label{cor:degGC}
We have $\deg G_\C(k,n) = \frac{0!1!\cdots (k-1)!}{(n-k)!(n-k+1)!\cdots (n-1)!}\cdot (k(n-k))!$.
\end{cor}

\begin{proof}
We apply Theorem~\ref{thm:IGFC} in the case $Z_i :=\Sigma_\C(k,n)$ 
and argue similarly as for Theorem~\ref{th:intersect_C}.
Let $\iota$ denote the Pl\"ucker embedding. Then 
$\iota(\Sigma_\C(k,n))$ is obtained by intersecting 
$\iota(G_\C(k,n))$ with a hyperplane $H_i$. 
The translates $g_1\Hy_1,\ldots,g_N\Hy_N$ meet transversally
for almost all $g_1,\ldots,g_n\in U(n)$. 
(This follows from Sard's Theorem as for 
\cite[Prop.~A.18]{BuCu}.)
By B\'ezout's theorem, the left-hand side of 
Theorem~\ref{thm:IGFC} equals $\deg G_\C(k,n)$.

From~\eqref{eq:HlinC} we know that 
$\Sigma_\C(k,n)| = |G_\C(k,n)| \cdot\frac{1}{\pi}\cdot N$
(and we already mentioned that this formula was 
obtained in the first version of \cite{bu-co-is} 
without algebraic geometry).  
Plugging in the right-hand side 
of Theorem~\ref{thm:IGFC} the formula~\eqref{eq:Gr-vol} 
for the volume $G_\C(k,n)$ and the formula for 
$\alpha_\C(k,n-k)$ from Proposition~\ref{pro:alphaC}, we obtain 
after some calculations the claimed expression for $\deg G_\C(k,n)$.
\end{proof}

\section{Random convex bodies}
\label{se:RCB}

In this section we will analyze the average scaling factors $\alpha(k,m)$ 
using properties of the singular value decomposition and 
concepts from the theory of random convex bodies. 

\subsection{The Segre zonoid}\label{se:coiso-zonoid}

We consider the embedding 
$\R^k \ti\R^m \to \R^{k\ti m},\ (u,v)\mapsto uv^T$ 
and assume $k\le m$. 
Recall that in~\S\ref{se:zonoid}, we associated a zonoid~$C_Y$ with a random vector $Y$. 

\begin{defi}\label{def:coiso-zono}
The {\em Segre zonoid} $C(k,m)$ is 
the zonoid associated with the random variable 
$xy^T \in \R^{k\ti m}$,  
where $x\in \R^k$ and $y\in \R^m$ 
are independent and standard Gaussian vectors. 
\end{defi}

If we sample instead $x\in S^{k-1}$ and $y\in S^{m-1}$ in the spheres 
independently and uniformly, 
then the corresponding zonoid equals $(\rho_k \rho_m)^{-1} C(k,m)$, 
compare Example~\ref{ex:zono}.

\begin{cor}\label{cor:edeg-im}
We have 
\be 
\edeg G(k,n) = |G(k,n)| \cdot \frac{(k(n-k))!}{2^{k(n-k)}} \cdot |C(k,n-k)| .
\ee
\end{cor}

\begin{proof}
Recall Equation~\eqref{eq:ided}, which states:  
\be
 \edeg G(k,n) = \alpha(k,n-k)\cdot |G(k,n)| \cdot \left(\frac{|\Sigma(k,n)|}{|G(k,n)|} \right)^{k(n-k)} .
\ee
Moreover, Theorem~\ref {prop:vspe} states that 
$\frac{|\Sigma(k,n)|}{|G(k,n)|} = \frac12 \rho_k \cdot \rho_{n-k}$
(recall the definition of $\rho_k$ in~\eqref{eq:def-rho}). 
By the definition of $\alpha(k,m)$ (Definition~\ref{def:alpha}), 
Theorem~\ref{thm:vitale} (Vitale),  
the definition of $C(k,m)$ (Definition~\ref{def:coiso-zono}), 
and the comment following it 
(exchanging Gaussian by uniform distributions on spheres), we obtain 
\be\label{eq:char-alpha}
 \alpha(k,m) = \frac{(km)!}{(\rho_k\rho_{m})^{km}} \cdot |C(k,m)|.
\ee
Combining all this, the assertion follows.
\end{proof}

The group $O(k)\ti O(m)$ acts on the space $\R^{k\ti m}$ of matrices~$X$  
via $(g,h)X := gXh^T$. 
It is well known that the orbits under this action 
are determined by the singular values of~$X$.
More specifically, 
we denote by $\sv(X)$ the list of singular values of $X$ (in any order). 
Let the matrix $\diag_{k,m}(\s)$ be obtained from 
the diagonal matrix $\diag(\s)$ by 
appending a zero matrix of format $k\ti (m-k)$. 
Then $X$ is in the same orbit as $\diag_{k,m}(\s)$ 
iff $\s$ equals $\sv(X)$ up to a permutation and sign changes. 

The distribution underlying the definition of the Segre zonoid $C(k,m)$ 
is $O(k)\ti O(m)$-invariant, which implies that $C(k,m)$ is invariant under this action.
Therefore, $C(k,m)$ is determined by the following convex set, 
which we call the {\em singular value zonoid}: 
\be\label{def:sv-body}
 D(k) := \big\{ \s\in\R^k \mid \diag_{k,m}(\s) \in C(k,m) \big\} .
\ee
We will see in Proposition~\ref{pro:prop-D} below that 
$D(k)$ does not depend on $m$, which justifies the notation.
Moreover, $D(k)$ can be realized as a projection of $C(k,m)$, 
hence it is indeed a zonoid (see Remark~\ref{re:D(k)asimage}). 
The {\em hyperoctahedral group} $H_k$ is the subgroup of $O(k)$ generated 
by the permutation of the coordinates and by the sign changes 
$x_i\mapsto \e_i x_i$, where $\e_i \in \{-1,1\}$.
Clearly, the convex set $D_k$ is invariant under the action of $H_k$. 
 
We note that $C(k,m)$ is the union of the $O(k)\ti O(m)$-orbits of the diagonal 
matrices $\diag_{k,m}(\s)$, where $\s\in D(k)$. 

We next determine the support function of the convex body $C(k,m)$.
By the invariance, this is an $O(k)\ti O(m)$-invariant function and hence
it can only depend on $\sv(X)$. 
For describing this in detail, we introduce the function $g_k\colon\R^k\to\R$ by 
\be\label{eq:def-g_k}
 g_k(\s_1,\ldots,\s_k) := \E \big(\s_1^2 z_1^2 + \ldots + \s_k^2 z_k^2\big)^\frac12 ,
\ee
where $z_1,\ldots,z_k$ are i.i.d.\ standard normal. 

\begin{lemma}\label{le:on_gk}
The function $g_k$ is a norm on $\R^k$. 
It is invariant under permutation of its arguments.
If $\|\s\|=1$, we have 
$ \rho_1/\sqrt{k} \le g_k(\s) \le 1/\sqrt{k}$
and the maximum of $g_k$ on $S^{k-1}$ 
is attained exactly on the $H_k$-orbit of the point $\s_u:=\frac{1}{\sqrt{k}}(1,\ldots,1)$.
\end{lemma}

\begin{proof}
The Cauchy-Schwarz inequality gives 
$\sum_i \s_i\tau_i z_i^2 \le (\sum_i \s_i^2 z_i^2)^\frac12\cdot (\sum_i \tau_i^2 z_i^2)^\frac12$. 
This easily implies $g_k(\s+\tau) \le g_k(\s) + g_k(\tau)$. It follows that $g_k$ is a norm on $\R^k$. 
The $S_k$-invariance is clear. Suppose now $\|\s\|=1$. Then there exists $i$ with 
$|\s_i|\ge k^{-\frac12}$ and hence $g_k(\s) \ge  k^{-\frac12}\E |z_i| =  k^{-\frac12}\rho_1$. 

We consider now the auxiliary function 
$\tilde{g}_k(\s) := \E (\s_1^2 u_1^2 + \ldots + \s_k^2 u_k^2)^\frac12$, 
where $u\in S^{k-1}$ is chosen uniformly at random. Then we have 
$g_k(\s) = \rho_k \tilde{g}_k(\s)$. Thus we need to show that 
$\max_{\|\s\|=1} \tilde{g}_k(\s) = 1/\sqrt{k}$,
and that the maximum is attained exactly on the $H_k$-orbit of of~$\s_u$.
For this, note first that 
$\tilde{g}(\s_u) = \frac{1}{\sqrt{k}}\,\|\s_u\| = \frac{1}{\sqrt{k}}$. 
Moreover, by the Cauchy-Schwarz inequality, 
$$
 \tilde{g}_k(\s) \le \bigg(\E \sum_{i=1}^k \s_i^2 u_i^2 \bigg)^\frac12 
   = \bigg( \sum_{i=1}^k \s_i^2\, \E u_i^2 \bigg)^\frac12 
   = \frac{\|\s\|}{\sqrt{k}} ,
$$
since $\E u_1^2 = \ldots =\E u_k^2 =1/k$. 
Equality holds iff 
$\s_1^2 u_1^2 + \ldots + \s_k^2 u_k^2$ 
is constant for $u\in S^{k-1}$.
This is the case iff $\s$ is in the 
$H_k$-orbit of of~$\s_u$.
\end{proof}


\begin{remark}
The function $g_k(\s)$ can be given the following interpretation in terms of the ellipsoid 
$E(\sigma) :=\{x\in\R^k\mid\sigma_1^2x_1^2+\cdots +\sigma_k^2x_k^2\leq 1\}$:
\be 
 g_k(\sigma)=
\frac{\rho_k}{k}\cdot\frac{|\partial E(\sigma)|}{|E(\sigma)|}.
\ee
This follows from~\cite[eq.~(3)]{Rivin}. The {\em isoperimetric ratio} $\frac{|\partial E(\sigma)|}{|E(\sigma)|}$ 
can be expressed in terms of certain hypergeometric functions, see~\cite[eq.~(13)]{Rivin}. 
\end{remark}


\begin{lemma}\label{le:C-supp-fct}
The support function of the convex body $C(k,m)$ is given by:
$$
 h_{C(k,m)}(X) = \frac{1}{\sqrt{2\pi}} \,g_k(\sv(X))  
 \quad \mbox{ for $X\in\R^{k\ti m}$} .
$$
\end{lemma}

\begin{proof}
Let $h$ denote the support function of $C(k,m)$ 
and define $\eta(\s) := h(\diag_{k,m}(\s))$. 
We have 
$h(X) = \eta(\sv(X))$ since $X$ is in the same orbit as 
$\diag_{k,m}(\sv(X))$ and $h$ is $O(k)\ti O(m)$-invariant. 
From \eqref{eq:support-h} we obtain
\be 
 \eta(\s) = \frac12\,\E \big|\langle \diag_{k,m}(\s),xy^T\rangle\big|
  = \frac12\,\E \Big|\sum_{i=1}^k \s_i x_i y_i \Big| ,
\ee
where $x\in\R^k$ and $y\in\R^m$ are independent standard normal. 
For fixed~$x$, the random variable 
$\sum_{i=1}^k \s_i x_i y_i$ is normal distributed with 
mean zero and variance $\sum_{i=1}^k \s_i^2 x_i^2$. Therefore, 
the expectation of its absolute value equals 
$(\sum_{i=1}^k \s_i^2 x_i^2)^{\frac12} \rho_1$, 
where we recall that $\rho_1 = \E |z| = \sqrt{2/\pi}$, 
$z\in\R$ denoting a standard Gaussian.  
Summarizing, we obtain 
$$
 \eta(\s) = \frac12\, \rho_1\, \E \Big(\sum_{i=1}^k \s_i^2 x_i^2 \Big)^{\frac12}  
     = \frac{1}{\sqrt{2\pi}}\, g_k(\s) ,
$$
which finishes the proof.
\end{proof}

\begin{cor}\label{cor:int-point-C}
The zero matrix is an interior point of $C(k,m)$.
\end{cor}

\begin{proof}
By Lemma~\ref{le:C-supp-fct} and Lemma~\ref{le:on_gk} we have 
\be
 \min_{\|X\| = 1} h_{C(k,m)}(X) 
  = \frac{1}{\sqrt{2\pi}} \min_{\|\s\|=1} g_k(\s) 
  \ \ge\  \frac{1}{\sqrt{2\pi}}\,\frac{\rho_1}{\sqrt{k}} 
   =  \frac{1}{\pi}\,\frac{1}{\sqrt{k}} .
\ee
Hence $C(k,m)$ contains the ball of radius $\pi^{-1} k^{-\frac12}$ around the origin. 
\end{proof}

 
We need a convexity property of singular values in the spirit of the 
Schur-Horn theorem for eigenvalues of symmetric matrices. 

\begin{thm}\label{conj:SH-sv}
Let $A=(a_{ij})\in\R^{k \ti m}$, $k\le m$, with singular values 
$\s_1\ge\ldots\ge \s_k\ge 0$. Then the diagonal
$(a_{11},\ldots,a_{kk})$ lies in the convex hull of the 
$H_k$-orbit of $(\s_1,\ldots,\s_k)$ of the hyperoctahedral group~$H_k$. 
\end{thm}

\begin{proof}
Let $B\in\R^{k\ti k}$ be obtained from $A\in\R^{k\ti m}$ by 
selecting the first $k$ columns. 
Let $\tau_1\ge\ldots\ge\tau_k\ge 0$
denote the singular values of $B$.
The Courant-Fischer min-max characterization of singular values 
(e.g., see \cite[Exercise 1.3.21]{Tao:12}) implies that 
$\tau_i \le \s_i$ for all $i$. 
Hence the convex hull of the $H_k$-orbit of $\tau:= (\tau_1,\ldots,\tau_k)$ 
is contained in the convex hull of the $H_k$-orbit of $\s:=(\s_1,\ldots,\s_k)$.  

A result by Thompson~\cite{thompson:77} implies that 
the diagonal $(a_{11},\ldots,a_{kk})$ of $B$ lies in the
convex hull of the $H_k$-orbit of $\tau$. 
Hence it lies in the convex hull of the $H_k$-orbit of $\s$. 
\end{proof}

We can now prove that the support functions and radial functions 
of $C(k,m)$ and $D(k)$ are linked in a simple way. 

\begin{prop}\label{pro:prop-D}
The following properties hold:
\begin{enumerate}
\item The support function of $D(k)$ equals $\frac{1}{\sqrt{2\pi}}\, g_k$. 
In particular, $D(k)$ does not depend on $m$. 
\item The radial function of $C(k,m)$ is given by 
$\R^{k\ti m} \to \R,\, X\mapsto r_k(\sv(X))$, where 
$r_k$~denotes the radial function of $D(k)$.
\end{enumerate}
\end{prop}

\begin{proof}
(1) We have $\s\in D(k)$ iff $\diag_{k,m}(\s)\in C(k,m)$, which, 
by Lemma~\ref{le:C-supp-fct}, is equivalent to
\be\label{eq:Y-cond}
 \forall Y \in \R^{k\ti m}\quad \sum_{i=1}^k \s_i Y_{ii} 
   \le \frac{1}{\sqrt{2\pi}}\, g_k(\sv(Y)) .
\ee
We need to prove that this is equivalent to
\be\label{eq:tau-cond}
 \forall \tau \in \R^{k}\quad \sum_{i=1}^k \s_i \tau_i 
   \le \frac{1}{\sqrt{2\pi}} \, g_k(\tau) .
\ee 
One direction being trivial, 
assume that $\s$ satisfies \eqref{eq:tau-cond} 
and let $Y\in\R^{k\ti m}$. Then we have for all $\pi\in H_k$,  
$$
 \langle \s, \pi \sv(Y) \rangle \le \frac{1}{\sqrt{2\pi}}\, g_k(\pi \sv(Y))
  = \frac{1}{\sqrt{2\pi}} \, g_k(\sv(Y)) .
$$
Theorem~\ref{conj:SH-sv} states that 
$(Y_{11},\ldots,Y_{kk})$ lies in the convex hull 
of the $H_k$-orbit of $\sv(Y)$. Therefore, 
\eqref{eq:Y-cond} holds, which 
proves the first assertion. 

(2) For the second assertion, let $R$ denote the radial function of $C(k,m)$. 
By invariance, it is sufficient to prove that 
$R(\diag_{k,m}(\s)) = r_k(\s)$. We have, using the above equivalences, 
\begin{eqnarray*}
R(\diag_{k,m}(\s)) &=& \max \big\{ t \mid t\, \diag_{k,m}(\s) \in C(k,m) \big\} \\
       &\stackrel{\eqref{eq:Y-cond}}{=}& \max \big\{ t \mid \forall Y\  
        \sum_{i=1}^k t \s_i Y_{ii} \le \frac{1}{\sqrt{2\pi}}\, g_k(\sv(Y)) \big\} \\
       &\stackrel{\eqref{eq:tau-cond}}{=}& \max \big\{ t \mid \forall \tau\ \langle t\s, \tau\rangle \le \frac{1}{\sqrt{2\pi}}\, g_k(\tau) \big\}  
               = r_k(\s)
\end{eqnarray*}
and the proof is complete. 
\end{proof}

\begin{remark}\label{re:D(k)asimage}
Proposition~\ref{pro:prop-D} 
easily implies that $D(k)$ is the image of $C(k,m)$ under the projection 
taking the diagonal of a rectangular matrix: 
$$
 \R^{k \ti m} \to \R^k,\, X =(x_{ij}) \mapsto (x_{11},\ldots,x_{kk}) . 
$$
In particular, $D(k)$ is a zonoid in $\R^k$. 
\end{remark}

\begin{lemma}\label{le:radius-D}
The maximum of the radial function $r_k$ of $D(k)$ 
on $S^{k-1}$ equals 
$$
 R_k :=\max_{\|\s\|=1} r_k(\s) =\frac{1}{\sqrt{2\pi}} \frac{\rho_k}{\sqrt{k}}
$$
and the maximum is attained exactly 
on the $H_k$-orbit of 
the point
$\s_u:=\frac{1}{\sqrt{k}}(1,\ldots,1)$.
In particular, $C(k,m)$ is contained in the ball $B(k,m)$ of 
radius $R_k$ around in the origin, and this is the ball of smallest radius 
with this property. 
\end{lemma}

\begin{proof}
According to Proposition~\ref{pro:prop-D}, 
$\frac{1}{\sqrt{2\pi}}\, g_k$ is the support function of $D(k)$. 
According to Lemma~\ref{le:on_gk}, the maximum of $g_k$ 
on $S^{k-1}$ is attained exactly on the $H_k$-orbit of~$\s_u$.
We conclude now with  Lemma~\ref{lemma:r=h}.
\end{proof}

\begin{remark}\label{remark:interior}
Corollary~\ref{cor:int-point-C} implies that $0$ is an interior point of $D(k)$.
This can also be seen in a different way as follows.
A standard transversality argument shows that 
$g_1\Sigma(k,n)\cap\ldots\cap g_N\Sigma(k,n)\ne\emptyset$ 
for all $(g_1,\ldots,g_N)$ in a nonempty open subset of $O(n)^N$; 
cf.~\cite[Prop.~A.18]{BuCu}.
This implies that $\edeg G(k,n)$ is positive. 
Via Theorem~\ref{thm:volconvex} below we conclude that $D(k)$ has full dimension. 
Since $D(k)$ is symmetric with respect to the origin, it again follows that 
$0$ is an interior point of $D(k)$.
\end{remark}
 
\subsection{Volume of the Segre zonoid}

We keep assuming $k\le m$. 
We first state how to integrate a $O(k)\times O(m)$-invariant function 
of matrices in $\R^{k\ti m}$ in terms of their singular values. 
This is certainly known, but we include the proof for lack of a suitable reference.

Let us denote by $S^{km-1}\subseteq \R^{k\times m}$ the unit sphere with respect to the Frobenius norm. 
Moreover, it will be convenient to abbreviate
\be
  S^{k-1}_+ := \{\s \in S^{k-1} \mid \sigma_1\geq\cdots\geq\sigma_k\geq 0\} .
\ee
Also, recall the Stiefel manifold 
$S(k,m) := \{Q\in\R^{m\ti k} \mid  Q^TQ = \textbf{1}\}$.  
We postpone the proof of the following technical result to Appendix~\ref{app:integration}.

\begin{prop}\label{prop:integration}
Let $f\colon S^{km-1}\to \R$ be a continuous $O(k)\times O(m)$-invariant function 
and put $g(\s):=f(\diag_{k,m}(\s))$. 
Then we have 
\be 
\int_{S^{km-1}} f(X)\, dS^{km-1}
  =\frac{|O(k)||S(k,m)|}{2^k}\int_{S^{k-1}_+}
     g(\sigma)\prod_{i=1}^k \sigma_i^{m-k}\prod_{1\leq i<j\leq k}(\sigma_i^2 - \sigma_j^2)\, dS^{k-1} .
\ee
\end{prop}

We can now derive a formula for the volume of $C(k,m)$. 
We introduce the $H_k$-invariant functions $p_k\colon\R^k\to\R$ and $q_k\colon\R^k\to\R$ by 
\be\label{eq:q} 
 p_k(\s_1,\ldots, \s_k) := \prod_{i=1}^k |\s_i| , \quad
 q_k(\s_1,\ldots, \s_k) := p_k(\s_1,\ldots, \s_k)^{-k} 
\cdot\prod_{i<j} |\s_i^2-\s_j^2 | .
\ee

\begin{thm}\label{thm:volconvex}
The volume of $C(k,m)$ equals 
\be 
 |C(k,m)|=\frac{|O(k)|\, |S(k,m)|}{km\, 2^k}
 \int_{S^{k-1}_+} \left( p_k\, r_{k}^k \right)^m q_{k}\, dS^{k-1},
\ee
where $r_k$ denotes the radial function of the 
singular value zonoid $D(k)$.
\end{thm}

\begin{proof}	
According to Lemma~\ref{le:radius-D}, the radius function of $C(k,m)$
is given by $X\mapsto r_k(\sv(X))$. 
The volume of $C(k,m)$ satisfies (see \cite[Equation (1.53)]{Schneider}) 
\be 
 |C(k,m)|=\frac{1}{km} \int_{S^{km-1}} (r_k\circ \sv)^{km}\, dS^{km-1}.
\ee
Since the radial function is continuous and $O(k)\times O(m)$-invariant, we can apply 
Proposition~\ref{prop:integration} to obtain
\begin{align} 
 |C(k,m)|&=\frac{1}{km}\frac{|O(k)||S(k,m)|}{2^k}
    \int_{S^{k-1}_+} r_k(\s)^{km}\, p_k(\s)^{m-k} 
  \prod_{1\leq i<j\leq k}(\sigma_i^2-\sigma_j^2)\,dS^{k-1}\\
  &=\frac{|O(k)||S(k,m)|}{km 2^k}
   \int_{S^{k-1}_+}\left( r_{k}(\s)^k\, p_k(\s) \right)^m q_{k}(\s)\, dS^{k-1} 
\end{align}
as claimed.
\end{proof}

\begin{remark}\label{re:any-invar-body}
The proof of Theorem~\ref{thm:volconvex} works for any 
$O(k)\ti O(m)$-invariant convex body. In particular, 
for the ball 
$B(k,m) := \{ X\in\R^{k\ti m} \mid \|X\|_F \le R_k\}$ 
containing $C(k,m)$ (see Lemma~\ref{le:radius-D}), we obtain:
\begin{align}
 |B(k,m)|=R_k^{km}\frac{\pi^{\frac{km}{2}}}{\Gamma\left(1+\frac{km}{2}\right)}
 =\frac{|O(k)|\, |S(k,m)|}{km\, 2^k}
 \int_{S^{k-1}_+} \left(p_k\, R_k^k\right)^m q_{k} \, dS^{k-1}.
\end{align}
\end{remark}


\section{Asymptotics}

\subsection{Laplace method}\label{se:laplace}

In this section we recall a useful result for the asymptotic evaluation of integrals depending on a large parameter. 
This technique will be crucial in the proofs of Theorem \ref{thm:estim} and Theorem \ref{thm:edlines} and goes under 
the name of \emph{Laplace's method}.
The following statement is classical (see for example \cite[Section II, Theorem 1]{Wong}).

\begin{thm}\label{thm:Laplace}Consider the integral
\be 
 I(\lambda)=\int_{t_1}^{t_2}e^{-\lambda a(t)}b(t)dt,
\ee
where the functions $a, b\colon [t_1, t_2]\to \R$ satisfy the following conditions:
\begin{enumerate}
\item The function $a$ is smooth on a neighborhood of $t_1$ and there exist $\mu>0$ and $a_0\neq 0$ such that as $t\to t_1$
\be 
 a(t)= a(t_1)+a_0(t-t_1)^\mu+\mathcal{O}(|t-t_1|^{\mu+1}) .
\ee
\item The function $b$ is smooth on a neighborhood of $t_1$ and there exist $\nu\geq 1$ and $b_0\neq 0$ such that as $t\to t_1$ 
\be 
 b(t)= b_0(t-t_1)^{\nu-1}+\mathcal{O}(|t-t_1|^{\nu})\quad  \textrm{with $b_0\neq 0$.}
\ee 
\item We have $a(t)>a(t_1)$ for all $t\in (t_1, t_2)$ and for every $\delta>0$ the infimum of $a(t)-a(t_1)$ in $[t_1+\delta,t_2)$ is positive. 
\item The integral $I(\lambda)$ converges absolutely for all sufficiently large $\lambda.$
\end{enumerate}
Then we have the asymptotic
\be\label{eq:Laplaceasymptotic} I(\lambda) = e^{-\lambda a(t_1)}\cdot 
   \frac{1}{\lambda^{\nu/\mu}}\cdot \frac{b_0\Gamma\left(\frac{\nu}{\mu}\right)}{a_0^{\nu/\mu} \mu}\cdot
       \left(1+\mathcal{O}(\lambda^{-(1+\nu)/\mu})\right)\quad \textrm{as $\lambda\to \infty$}.
\ee
\end{thm}

\begin{remark}
\begin{enumerate}
\item 
The hypotheses of \cite[Section II, Theorem 1]{Wong} are weaker than those of 
Theorem~\ref{thm:Laplace}: it is only required that $a$ and $b$ are continuous in a neighborhood of $t_1$ 
(except possibly at $t_1$) and that they have asymptotic series at $t_1$ (which is certainly true if they are both smooth).
Moreover the conclusion given there 
is also stronger than what we stated here:
it is proved that $I(\lambda)$ also has an asymptotic series in descending powers of $\lambda$, 
while we are just writing the leading order term and the order of the error of this series. 

\item Note that 
the exponent of the exponential growth in~\eqref{eq:Laplaceasymptotic} is determined 
by the minimum $a(t_1)$, the growth of the pre-exponential factor is determined by 
the orders $\mu,\nu$ of the expansions of the functions $a$ and $b$, 
and the leading constant in~\eqref{eq:Laplaceasymptotic} involves the constants $a_0$ and~$b_0$.
\end{enumerate}
\end{remark}

\subsection{An asymptotically sharp upper bound for the expected degree}

From Lemma~\ref{le:radius-D}, we see that the Segre zonoid $C(k,m)$ is contained in the ball 
\be
 B(k,m) := \Big\{ X\in\R^{k\times m} \mid \|X\|_F \le \frac{1}{\sqrt{2\pi}} \frac{\rho_k}{\sqrt{k}} \Big\} .
\ee
In particular, 
$|C(k,m)| \le |B(k,m)|$,
which implies an upper bound for $\alpha(k,m)$. 
We next show that this inequality, for fixed $k$,  
is asymptotically sharp up to a subexponential factor.

\begin{thm}\label{thm:estim}
For fixed $k$ we have 
$\log|C(k,m)|=\log |B(k,m)|-\mathcal{O}(\log m)$
for $m\to\infty$.		
\end{thm}

\begin{proof}
We introduce first some useful notation that we will use in the proof. Recall from Lemma~\ref{le:radius-D} that 
$R_k = (2\pi)^{-\frac12} \rho_k/\sqrt{k}$ is the maximum of the 
radial function $r_k$ on $S^{k-1}_+$, 
and that the maximum is attained exactly at $\s_u :=\frac{1}{\sqrt{k}}(1, \ldots, 1)$.
It is convenient to define the normalized radius function
$\tilde{r}_k := r_k/R_k$, which has the maximum value~$1$. 

Theorem \ref{thm:volconvex} and Remark~\ref{re:any-invar-body} imply now that we can rewrite 
the two volumes $|C(k,m)|$ and $|B(k,m)|$ as
\be \label{eq:integrals2}
 |C(k,m)|=\frac{|O(k)|\, |S(k,m)|}{km\, 2^k} R_k^{km}\cdot I_C(k,m),\quad
 |B(k,m)|=\frac{|O(k)|\, |S(k,m)|}{km\, 2^k}  R_k^{km}\cdot I_B(k,m) ,
\ee
where 
\be\label{eq:integrals}
 I_C(k,m) := \int_{S^{k-1}_+} \left( \tilde{r}_k^k\, p_k \right)^m q_{k} \, dS^{k-1}
\quad \textrm{and} \quad 
 I_B(k,m) := \int_{S^{k-1}_+} p_k^m q_{k}\, dS^{k-1}.
\ee

The idea of the proof is now to compute the asymptotics of two integrals $I_B(k,m)$ and $I_C(k,m)$, 
the first one explicitly and the second one using Laplace's method, and then to compare them. 
Because of $|C(k,m)| \le |B(k,m)|$, and  
since the factors in front of the two integrals in~\eqref{eq:integrals2} are the same, 
we just have to verify that
\be\label{eq:lowerest2} 
 \log I_C(k,m)\geq \log I_B(k,m) - \mathcal{O}(\log m)\quad  \textrm{for $m\to \infty$.}
\ee
We first look at the  asymptotic growth of $I_B(k,m)$. Note that 
by rearranging \eqref{eq:integrals2}, we get 
\be 
 I_B(k,m)=\frac{k m 2^k}{|O(k)|\, |S(k,m)|} \cdot |B_{\R^{km}}(0,1)| ,
\ee
since $R_k$ is the radius of the ball $B(k,m)$. 
Using the explicit formulas
\be 
|B_{\R^{km}}(0,1)|=\frac{\pi^{km/2}}{\Gamma\left(1+\frac{km}{2}\right)},\quad \quad 
 |S(k,m)|=\frac{2^k\pi^{km/2}}{\Gamma_k\left(\frac{m}{2}\right)}
\ee
and the asymptotics $\log \Gamma(z)=z\log z-z+\mathcal{O}(\log z)$ and $\log \Gamma_k(z)=k(z\log z-z)+\mathcal{O}(\log z)$ for $z\to \infty$, 
we easily deduce that 
\be\label{eq:esti1} 
 \log I_{B}(k,m)=km \log\left(\frac{1}{\sqrt{k}}\right)+\mathcal{O}(\log m)\quad \textrm{as $m\to \infty$}.
\ee
For the asymptotic growth of $I_C(k,m)$ we will need the following result, whose proof we postpone.

\begin{lemma}\label{prop:step2}
There exist $\kappa_0, \kappa_1, \delta>0$, depending on $k$, such that $\kappa_1\delta \le 1$ and for all~$m$, 
\be
 I_{C}(k,m)\geq \kappa_0 \int_{0}^{\delta}\left( k^{-\frac{k}{2}} (1-\kappa_1\rho) \right)^{m} \rho^{\frac{k^2+k-4}{2}} \, d\rho .
\ee
\end{lemma}

We rewrite the right-hand side integral in Lemma~\ref{prop:step2} as 
$\int_{0}^{\delta_0}e^{-ma(\rho)}b(\rho) d\rho$
with the smooth functions 
\be 
 a(\rho) := -\log\left( k^{-k/2} (1-\kappa_1\rho) \right) = \frac{k}{2} \log k + \kappa_1 \rho +\mathcal{O}(\rho^2), 
\quad b(\rho) := \rho^{\frac{k^2+k-4}{2}} ,
\ee
where the expansion is for $\rho\to 0$.

Note now that both functions $a$ and $b$ are smooth ($b$ is a polynomial since for $k\geq 2$ 
the exponent $\frac{k^2+k-4}{2}$ is a positive natural number). 
We can then apply Theorem \ref{thm:Laplace} with the choices 
$a_0=\kappa_1, \mu=1, b_0=1, \nu=\frac{k^2+k-4}{2} +1$
and obtain for $m\to\infty$
\be  
 \int_{0}^{\delta_0} \left( k^{-\frac{k}{2}} (1-\kappa_1\rho) \right)^{m} \rho^{\frac{k^2+k-4}{2}} \, d\rho
  =\left(\frac{1}{\sqrt{k}}\right)^{km} \frac{1}{m^\nu} \frac{\Gamma(\nu)}{\kappa_1^\nu} \left(1+\mathcal{O}(m^{-(\nu+1)})\right).
\ee
In particular, combining this asymptotic with Lemma~\ref{prop:step2}, 
we deduce that for $m\to\infty$ 
\begin{align}
\log I_{C}(k,m)&\geq\log\left(\kappa_0 
   \int_{0}^{\delta_0} \left( k^{-\frac{k}{2}} (1-\kappa_1\rho) \right)^{m} \rho^{\frac{k^2+k-4}{2}} \, d\rho \right)\\
& \geq km \log\left(\frac{1}{\sqrt{k}}\right) - \mathcal{O}(\log m) .
\end{align}
Together with \eqref{eq:esti1}, this shows~\eqref{eq:lowerest2} and proves Theorem~\ref{thm:estim}. 
\end{proof}


\begin{proof}[Proof of Lemma~\ref{prop:step2}]
It will be convenient to replace the singular value zonoid~$D(k)$ by a convex body $Y\subseteq D(k)$ 
that is simpler to analyze.
By Corollary~\ref{cor:int-point-C} (see also Remark \ref{remark:interior}), 
the zero vector is an interior point of~$D(k)$, 
hence there exists $\epsilon>0$ such that $B_{\R^k}(0, \epsilon)\subseteq D(k)$. 
We define $Y$ as the convex hull of $B_{\R^k}(0, \epsilon)$ and $\overline{y}:= R_k \s_u \in D(k)$. 
Thus $Y$ is the cone with base $B_{\R^k}(0, \epsilon)$ and 
apex $\overline{y}$.
Note that $Y$ is rotation symmetric with respect to the line passing through~$0$ and $\s_u$. 

In the definition~\eqref{eq:integrals} of $I_C(k,m)$, we can can integrate over the full sphere $S^{k-1}$,  
$$
 I_C(k,m) =  \int_{S^{k-1}_+} \left( \tilde{r}_k^k\, p_k \right)^m q_{k} \, dS^{k-1} 
   = \frac{1}{k! 2^k}  \int_{S^{k-1}} \left( \tilde{r}_k^k\, p_k \right)^m q_{k} \, dS^{k-1} ,
$$
by the $H_k$-invariance of the integrand. 
Recall that $\tilde{r}_k = R_k^{-1} r_k$ is the normalized radial function of $D(k)$. 
Denoting by $r_Y$ the radial function of $Y$ and 
setting $\tilde{r}_Y := R_k^{-1} r_Y$, we have 
$\tilde{r}_k\geq \tilde{r}_{Y}$, since $Y\subseteq D(k)$. 
Replacing $\tilde{r}_k$ by the smaller $\tilde{r}_Y$, we obtain 
\be 
 I_C(k,m) \ \ge\ 
  \frac{1}{k! 2^k}  \int_{S^{k-1}} \left( \tilde{r}_Y^k\, p_k \right)^m q_{k} \, dS^{k-1} .
\ee 
For the integration of the right-hand side, we use a coordinate system adapted to the situation. 
The exponential map 
$$ 
 \big\{x\in T_{\s_u}S^{k-1} \mid \|x\| < \pi \big\} \to S^{k-1}\setminus \{-\s_u\},\, x\mapsto \s(x)
$$ 
is a diffeomorphism that maps $0$ to $\s_u$. 
We describe tangent vectors in $T_{\s_u}S^{k-1}$ with polar coordinates 
$(\rho,\theta) \in [0,\infty)\times S^{k-2}$, where $S^{k-2}$ stands here for the unit sphere 
of $T_{\s_u}S^{k-1}$. 
This way, we obtain a parametrization $\s(\rho,\theta)$ of the points in $S^{k-1}\setminus\{-\s_u\}$. 
It is easy to see that its Jacobian is of the form $\rho^{k-2} w(\rho)$ with a smooth function~$w$ 
satisfying $w(0) >0$. By the transformation theorem we have
\be
 \int_{S^{k-1}} \left( \tilde{r}_Y^k\, p_k \right)^m q_{k} \, dS^{k-1} =  
 \int_0^\pi \int_{S^{k-2}} \left( \tilde{r}_Y(\s(\rho,\theta))^k\, p_k(\s(\rho,\theta) \right)^m 
   q_{k}(\s(\rho,\theta))\, \rho^{k-2} w(\rho)\, dS^{k-2} \, d\rho .
\ee
\medskip

\noindent {\bf Claim.} There are $c_1,c_2,c_3,c_4>0$, $\pi>\delta_1>0$ 
and a nonzero continuous function 
$\eta\colon S^{k-2} \to\R$ such that for all $\rho\le\delta_1$ we have  $w(\rho) \ge c_1$ and
\be
 \tilde{r}_Y(\s(\rho,\theta)) \ge 1 - c_2 \rho,\quad
 p_k(\s(\rho,\theta)) \ge k^{-k/2} \big(1- c_3 \rho^2 \big), \quad 
 q_k(\s(\rho,\theta)) \ge \rho^{\frac{k(k-1)}{2}}\, \big( |\eta(\theta)| - c_4 \rho\big) . \quad 
\ee

\begin{proof}
1. The lower bound on $w(\rho)$ follows from the continuity of $w$ and $w(0)>0$. 

2. By the axial symmetry of $Y$, the function $\tilde{r}_Y(\s(\rho,\theta))$ only depends on $\rho$.   
The lower bound on $\tilde{r}_Y(\rho)$ easily follows from the construction of $Y$. 

3. The function $p_k(\s(x))$ is smooth in a neighborhood of the zero vector and 
has the expansion
\be 
 p_k(\s(x))=p_k(\s(0))+\frac{1}{2}xHx^T+\mathcal{O}(\|x\|^3) 
\ee
for $x\in T_{\s_u} S^{k-1}$, $x\to 0$.  
The matrix $H$ is negative semidefinite since $x=0$ 
is a local maximum of $p_k(\s(x))$. 
Note that $p_k(\s(0)) = p_k(\s_u) = k^{-k/2}$. 
In particular, there exists a constant $c'>0$ such that for small enough $\rho=\|x\|$ we have 
$p_k(\s(x))\geq p_k(0)-c'\|x\|^2$, which shows the stated lower bound on $p_k$.  

Recall from \eqref{eq:q} that the function $q_k(\s)$ is defined as the modulus of the function 
$f_k(\sigma) := \prod_{i=1}^k \s_i^{-k} \cdot\prod_{i<j}( \s_i^2-\s_j^2)$. 
The function $f_k$ is smooth at $\sigma_u$ and it vanishes at $\sigma_u$ to order $\ell:={k\choose 2}$. 
We consider the Taylor expansion 
$f_k(\s(x))=h(x)+\mathcal{O}\left(\|x\|^{\ell+1}\right)$, 
where $h$ is a nonzero homogeneous polynomial of degree $\ell$. 
In particular, we get in polar coordinates
\be 
 f_k(\s(\rho, \theta)) = \rho^{\ell} \eta(\theta) + \mathcal{O}\left(\rho^{\ell+1}\right)
\ee
where the function $\eta:S^{k-2}\to \R$ is nonzero. 
The stated lower bound on $q_k$ follows and the claim is shown. 
\end{proof}

To complete the proof of Lemma~\ref{prop:step2}, 
we put $\kappa:= \int_{S^{k-2}} |h| dS^{k-2}$, which is positive. 
Using the claim, we obtain for 
$0 < \delta \le  \min\Big\{\delta_1,\frac{\kappa}{2 c_4|S^{k-2}|} \Big\}$ that 
\begin{eqnarray*}
\lefteqn{\int_{S^{k-1}} \left( \tilde{r}_Y^k\, p_k \right)^m q_{k} \, dS^{k-1} }\\
 &\ge&  
 \int_0^{\delta} \left( (1-c_2\rho)^k\, k^{-k/2}\, (1-c_3\rho^2) \right)^m \rho^{\frac{k(k-1)}{2}+k-2} c_1  
   \int_{S^{k-2}} \Big(|h| -c_4\rho\Big)\, dS^{k-2}\, d\rho \\
 &\ge&  \frac{c_1\kappa}{2} \int_0^{\delta} \left(  k^{-k/2} (1-c_2\rho)^k (1-c_3\rho^2) \right)^m \rho^{\frac{k^2+k-4}{2}} \, d\rho \\
 &\ge&  \frac{c_1\kappa}{2} \int_0^{\delta} \left( k^{-k/2} (1-2kc_2\rho) \right)^m \rho^{\frac{k^2+k-4}{2}} \, d\rho ,
\end{eqnarray*}
where the last inequality holds if $\rho$ is sufficiently small. 
The assertion follows with 
$\kappa_0 := \frac12 c_1\kappa$ and $\kappa_1 := 2kc_2$.
\end{proof}

Finally, we arrive at the main result of this section, an asymptotically sharp upper bound 
for the expected degree of real Grassmannians. 
The result should be compared with the corresponding statement \eqref{eq:vol-deg-G}. 
Recall the quantity $\rho_k \le\sqrt{k}$ from \eqref{eq:def-rho} and \eqref{eq:rho-estim}. 

\begin{thm}\label{thm:ased}
Let $N:=k(n-k) = \dim G(k,n)$. The following statements are true:
\begin{enumerate}
\item For all $n\geq k>0$ we have
\be
 \edeg G(k,n)\ \leq\ E(k,n) := \frac{|G(k,n)|}{|\RP^N|}\cdot 
  \left(\sqrt{\frac{\pi}{2}}\, \frac{\rho_k}{\sqrt{k}}\right)^N .
\ee

\item 
For fixed $k$ and $n\to \infty$, the above inequality is asymptotically sharp in the sense that 
\be 
 \log \edeg G(k,n) = \log E(k,n) - \mathcal{O}(\log n) .
\ee

\item 
For fixed $k$ and $n\to \infty$, we have 
\be 
\log E(k,n) = kn \log\left(\frac{\sqrt{\pi}\Gamma \left(\frac{k+1}{2}\right)}{\Gamma\left(\frac{k}{2}\right)}\right) -\mathcal{O}(\log n).
\ee
\end{enumerate}
\end{thm}

\begin{proof}
Recall the formula for $\edeg G(k,n)$ given in Corollary~\ref{cor:edeg-im}.  
In this formula, we replace $|C(k,n-k)|$ by the larger volume of the ball~$B(k,n-k)$, 
obtaining the quantity 
\be\label{eq:defE}
 E(k,n) := |G(k,n)| \cdot\frac{N!}{2^{N}} \cdot |B(k,n-k)| ,
\ee
which satisfies $\edeg G(k,n) \le E(k,n)$. 
Theorem~\ref{thm:estim} implies that 
$\log E(k,n) - \log\edeg G(k,n) = \log |B(k,n-k)| - \log |C(k,n-k)| = \mathcal{O}(\log n)$ 
for fixed~$k$ and $n\to\infty$.

We verify now that the $E(k,n)$ defined in \eqref{eq:defE} indeed has the form stated in assertion~(1). 
Recall that $B(k,n-k)$ is the ball in $\R^N$ with the radius $R_k=\rho_k/\sqrt{2\pi k}$; cf.\ Lemma~\ref{le:radius-D}. 
We have $|B(k,n-k)| = R_k^N |B_{\R^N}(0,1)|$ and 
$$
  |B_{\R^N}(0,1)| = \frac{1}{N} |S^{N-1}| = \frac{2}{N} \frac{\pi^{\frac{N}{2}}}{\Gamma(N/2)}
           = \frac{\pi^{\frac{N}{2}}}{\Gamma(\frac{N+2}{2})} .
$$ 
Plugging this into \eqref{eq:defE}, and simplifying with the help of the duplication formula
$N!=\frac{2^N}{\sqrt{\pi}}\cdot \Gamma\left(\frac{N+1}{2}\right)\cdot\Gamma\left(\frac{N+2}{2}\right)$,  
yields
$$
 E(k,n) = |G(k,n)| \cdot \frac{\Gamma\left(\frac{N+1}{2}\right)}{\pi^{\frac{N+1}{2}}} \cdot \left(\rho_k\sqrt{\frac{\pi}{2k}}\right)^{N} 
 = |G(k,n)| \cdot \frac{1}{|\RP^N|} \cdot \left(\sqrt{\frac{\pi}{2}}\, \frac{\rho_k}{\sqrt{k}}\right)^N ,
$$
as claimed in assertion~(1).  

It remains to prove the asymptotic of $E(k,n)$ stated in assertion~(3). By the first assertion and \eqref{eq:def-rho} 
\begin{align} 
\log E(k,n)
 &=\log\left(\frac{|G(k,n)|}{|\RP^N|}\cdot \frac{1}{k^{N/2}}\cdot 
 \left(\frac{\sqrt{\pi}\,\Gamma \left(\frac{k+1}{2}\right)}{\Gamma\left(\frac{k}{2}\right)}\right)^N\right)\\
 &=\label{eq:logedeg}k n\log\left(\frac{\sqrt{\pi}\,\Gamma \left(\frac{k+1}{2}\right)}{\Gamma\left(\frac{k}{2}\right)}\right)
 +\log \left(\frac{|G(k,n)|}{|\RP^N|}\cdot \frac{1}{k^{N/2}}\right)+\mathcal{O}(\log n).
\end{align}
Lemma~\ref{le:volG} states that for fixed $k$ and $n\to\infty$ 
\be \label{eq:lG}
 \log |G(k,n)| = -\frac12 kn\log n + \frac12 kn\log (2e\pi)+\mathcal{O}(\log n).
\ee
Moreover, for $|\RP^N|=\frac{1}{2}|S^N|$ we obtain that, using 
$\log\Gamma(x) = x\log x - x + \mathcal{O}(\log x)$ for $x\to\infty$, 
\be \label{eq:lP}
 \log |\RP^N|=-\frac12 kn\log n+\frac12 kn\log\left(\frac{2e\pi}{k}\right)+\mathcal{O}(\log n)
\ee
Combining \eqref{eq:lG} and \eqref{eq:lP}, it follows that
$\log\left(\frac{|G(k,n)|}{|\RP^N|}\cdot \frac{1}{k^{N/2}} \right)=\mathcal{O}(\log n)$,
which completes the proof. 
\end{proof}

We compare now the expected degree of the real Grassmannian with the 
degree of the corresponding complex Grassmannian. 

\begin{cor}\label{cor:exponent}
For $k\ge 2$ we have 
$$
 \e_k := \lim_{n\to\infty} \frac{\log\edeg G(k,n)}{\log\sqrt{\deg G_\C(k,n)}} 
          = \log_k \left(\frac{\pi\rho_k^2}{2}\right) .
$$
Moreover, the sequence $\e_k$ is monotonically decreasing and $\lim_{k\to\infty} \e_k = 1$. 
\end{cor}


\begin{proof}
The formula for $\e_k$ follows from Theorem~\ref{thm:ased} 
and the asymptotic
$\edeg G_\C(k,n) = kn\log k + \mathcal{O}(\log n)$ 
from Lemma~\ref{le:volG}.
We have $\lim_{k\to\infty} \e_k = 1$ due to~\eqref{eq:rho-estim}.  
The verification of the monotonicity of $\e_k$ is left to the reader.
\end{proof}

Corollary~\ref{cor:exponent} implies that for fixed $k$: 
\be\label{eq:sqr-law}
 \edeg G(k,n) = \deg G_\C(k,n)^{\frac12 \e_k + o(1)} \quad \mbox{for $n\to\infty$.}
\ee 
This means that for large $n$, the expected degree of the real Grassmannian
exceeds the square root of the degree of the corresponding complex Grassmannian 
and in the exponent, the deviation from the square root is measured by $\frac12 (\e_k - 1)$. 
For example, $\e_2= 2\log_2(\pi/2)\approx 1.30$. 
For large~$k$, the exponent~$\e_k$ goes to $1$ so that we get the 
{\em square root law} 
\be\label{eq:squarerootlaw} 
 \edeg G(k,n) = \deg G_\C(k,n)^{\frac12 +o(1)}\quad\mbox{for $k,n\to \infty$.}  
\ee

\subsection{The Grassmannian of lines}\label{se:G-lines}

\subsubsection{The expected degree of $G(2,4)$}
We can express this as an integral of the modulus of the function 
\begin{align} 
a(t_1, t_2, t_3, s_1, s_2, s_3)
 =& \cos s_2 \sin s_1\sin s_3 \sin t_2 \sin(t_1-t_3)\\
   &-(\sin s_2)\left(\cos s_1\sin s_3 \sin t_1\sin(t_2-t_3)+\cos s_3 \sin s_1\sin t_3\sin(t_1-t_2) \right) ,
\end{align}
which can be written in the following symmetric way (this obervation is due to Chris Peterson) 
\be 
| a(t, s)|=\left|\det \left(\begin{array}{ccc}\sin t_1\sin s_1 & \sin t_2\sin s_2 & \sin t_3\sin s_3 \\
 \cos t_1 \sin s_1 & \cos t_2\sin s_2 & \cos t_3 \sin s_3 \\
 \sin t_1 \cos s_1 & \sin t_2\cos s_2 & \sin t_3\cos s_3\end{array}\right)\right|.
\ee

\begin{prop}\label{prop:twolines}
We have $\edeg G(2,4) = \frac{1}{2^{13}}\int_{[0, 2\pi]^6}\left| a(t, s)\right|dtds=1.72....$.
\end{prop}

\begin{proof}
Formula~\eqref{eq:ided} implies that 
\be
 \edeg G(2,4) = |G(2,4)| \cdot \left(\frac{|\Sigma(2,4)|}{|G(2,4)|} \right)^{4}  \cdot \alpha(2,2).
\ee
By \eqref{eq:Gr-vol} we have $|G(2,4)|=2\pi^2$ and from Theorem~\ref{prop:vspe} 
we deduce $\frac{|\Sigma(2,4)|}{|G(2,4)|} =\frac{\pi}{4}$.
By definition, we have 
$\alpha(2,2)= \E \| (u_1\ot v_1) \wedge \ldots \wedge (u_{4} \ot v_{4})\|$, 
where the $u_1,u_2,u_3,v_1,v_2,v_3\in S^1$ are independently and uniformly chosen at random, 
and we can assume $u_4=v_4=(1,0)$ by the invariance of the problem. 
We can thus write 
$u_i=(\cos t_i, \sin t_i)$ and $v_i=(\cos s_i, \sin s_i)$ 
where $(t_i, s_i)\in [0, 2\pi]^2$ are random with the uniform density.
Setting $w_i :=u_i\ot v_i$, we have 
by \eqref{eq:norm-fomula} that 
$\| (u_1\ot v_1) \wedge \ldots \wedge (u_{4} \ot v_{4})\|^2=\det \left(\langle w_i, w_j\rangle\right)$.
Moreover, 
\be 
\langle w_i, w_j\rangle=\langle u_i, u_j\rangle \langle v_i, v_j\rangle=(\cos t_i \cos t_j+\sin t_i \sin t_j)(\cos s_i \cos s_j+\sin s_i\sin s_j).
\ee
Expanding the determinant, we see after a few simplifications that 
$\det \left(\langle w_i, w_j\rangle\right)=a^2$
with the function~$a$ defined above. 
Consequently:
\begin{align}
 \edeg G(2,4) &= 2\pi^2\cdot \left(\frac{\pi}{4}\right)^4\cdot \frac{1}{(2\pi)^6}\int_{[0, 2\pi]^6}\left| a(t,s)\right|dt ds ,
\end{align}
which proves the claim.
\end{proof}

\subsubsection{The general case}

In the case of $G(2, n+1)$, Theorem \ref{thm:ased} provides the asymptotic
\be 
 \log \edeg G(2,n+1) = 2n \log\frac{\pi}{2}+\mathcal{O}(\log n)\quad \textrm{as $n\to \infty$}.
\ee
We sharpen this by proving the following result. 

\begin{thm}\label{thm:edlines}
The expected degree of the Grassmannian of lines satisfies 
\be 
 \edeg G(2,n+1)= \frac{8}{3\pi^{5/2}}\frac{1}{\sqrt{n}}  \left(\frac{\pi^2}{4} \right)^n \left(1+\mathcal{O}(n^{-1})\right).
\ee
\end{thm}

\begin{proof}
By Corollary \ref{cor:edeg-im} we have 
\be\label{eq:690} 
 \edeg G(2, n+1)=|G(2, n+1)|\cdot \frac{(2n-2)!}{2^{2n-2}}\cdot |C(2, n-1)|.
\ee
The volume of the Grassmannian can be computed using \eqref{eq:Gr-vol}:
\be\label{eq:691} 
  |G(2, n+1)| = \frac{|O(n+1)|}{|O(2)|\cdot |O(n-1)|} = \frac{|S^n|\cdot |S^{n-1}|}{|O(2)|} 
   =\frac{\pi^{n-\frac12}}{\Gamma\left(\frac{n}{2}\right)\Gamma\left(\frac{n+1}{2}\right)}=\frac{(2\pi)^{n-1}}{(n-1)!} ,
\ee
where we used the duplication formula for the last inequality. Moreover, 
$$
 |S(2,n-1)| \stackrel{\eqref{eq:Ovol}}{=} \frac{4\pi^{n-1}}{\pi^\frac12\Gamma(\frac{n-1}{2})\Gamma(\frac{n-2}{2})} 
  = \frac{2^{n-1}\pi^{n-2}}{(n-3)!} .  
$$
Theorem \ref{thm:volconvex} describes the volume of $C(2, n-1)$:
\begin{align} 
 |C(2,n-1)|&=\frac{|O(2)|\cdot|S(2, n-1)|}{(2n-2)2^2}
   \int_{0}^{\frac{\pi}{4}}\left(r_2(\theta)^2\cos \theta\sin \theta\right)^{n-1} 
           \frac{(\cos \theta)^2-(\sin \theta)^2}{\left(\cos \theta \sin \theta\right)^2}d\theta\\
\label{eq:692}&=\frac{2^{n-2}\pi^{n-1}}{(n-1)\Gamma(n-2)}\int_{0}^{\frac{\pi}{4}}\left(r_2(\theta)^2\cos \theta\sin \theta\right)^{n-1} 
   \frac{(\cos \theta)^2-(\sin \theta)^2}{\left(\cos \theta \sin \theta\right)^2}d\theta,
\end{align}
where here and below, abusing notation, 
we denote by $r_2(\theta)$ the function $r_2(\cos\theta, \sin \theta)$. 

Plugging \eqref{eq:691} and \eqref{eq:692} into \eqref{eq:690}, and simplifying, we obtain:
\be\label{eq:693}  
 \edeg G(2, n+1)=\frac{\pi^{2n-2}\Gamma(2n-1)}{(2n-2)\Gamma(n-2)\Gamma(n)}
   \int_{0}^{\frac{\pi}{4}}\left(r_2(\theta)^2\cos \theta\sin \theta\right)^{n-1} 
   \frac{(\cos \theta)^2-(\sin \theta)^2}{\left(\cos \theta \sin \theta\right)^2}d\theta.
\ee
We will now use Laplace's method for the evaluation of the integral in the above equation; first we will write it as:
\be 
 I(\lambda)=\int_{0}^{\frac{\pi}{4}}e^{-\lambda a(\theta)}b(\theta)d\theta,
\ee
where $\lambda=n-1$ and where we have set
\be 
  a(\theta) := -\log(r_2(\theta )^2\cos \theta \sin \theta), \quad 
  b(\theta) :=\frac{(\cos \theta)^2-(\sin \theta)^2}{\left(\cos \theta \sin \theta\right)^2}.
\ee
In order to apply Laplace's method, we verify now that the hypotheses (1)--(4) of Theorem~\ref{thm:Laplace} are satisfied. 
(Note that the minimum of $a$ is attained at the {\em right} extremum of the interval of integration, 
which leads to minor modifications.) 
\begin{itemize}
\item[(1)] 
By Proposition~\ref{prop:r2} below, 
the function $r_2(\theta)^2$ is smooth in a neighborhood of $\frac{\pi}{4}$ and 
has the following asymptotic expansion for $\theta\to\frac{\pi}{4}$
\be\label{eq:seriesr2} 
  r_2(\theta)^2=\frac{1}{8}-\frac{1}{8}(\theta-\frac{\pi}{4})^2+\mathcal{O}\left( (\theta-\frac{\pi}{4})^3\right).
\ee
Consequently the function $a(\theta)=\log (r_2(\theta)^2\cos \theta \sin \theta)$ 
is smooth in a neighborhood of $\frac{\pi}{4}$ and, 
as $\theta\to \frac{\pi}{4}$, we can easily deduce
\be 
 a(\theta)=4\log 2+3 \left(\theta-\frac{\pi }{4}\right)^2 + \mathcal{O}\left( (\theta-\frac{\pi}{4})^3\right). 
\ee
\item[(2)] 
The function $b$ is smooth in a neighborhood of $\frac{\pi}{4}$, and, as $\theta\to\frac{\pi}{4}$, 
\be 
 b(\theta)=-8\left(\theta-\frac{\pi}{4}\right) + \mathcal{O}\left( (\theta-\frac{\pi}{4})^2\right). 
\ee
\item [(3)] 
The function $\cos \theta\sin\theta$ is monotonically increasing on $[0, \pi/4]$, 
hence it has a unique maximum at $\frac{\pi}{4}$; moreover, by Lemma~\ref{le:radius-D}, 
the function $r_2(\theta)$ also has a unique maximum on $[0, \pi/4]$ at $\frac{\pi}{4}$. 
Therefore, $a$ has a unique minimum at $\theta=\frac{\pi}{4}$ on the interval $[0, \pi/4]$. 

\item[(4)] 
The integrand in $I(\lambda)$ is nonnegative and $I(\lambda)<\infty$ for every $\lambda$, since $\edeg G(2,n+1)$ is finite. 
\end{itemize}

We can now apply Theorem \ref{thm:Laplace} with $\mu=\nu=2$, $a_0=3$ and $b_0=8$ 
(we have to change the sign of $b_0$ because the minimum is now attained at the \emph{right} extremum of the interval of integration). 
This implies for $\lambda\to\infty$: 
\begin{align}
 I(\lambda)&=e^{-\lambda 4\log 2}\frac{1}{\lambda}\cdot \frac{8}{3\cdot 2}\cdot \left(1+\mathcal{O}(\lambda^{-3/2})\right)\\
\label{eq:694}
 &=2^{-4(n-1)}\frac{1}{n-1}\cdot \frac{4}{3}\cdot \left(1+\mathcal{O}(n^{-3/2})\right).
\end{align}
We plug in now \eqref{eq:694} into \eqref{eq:693}, obtaining 
\be\label{eq:edeg2}
 \edeg  G(2, n+1)=\frac{\pi^{2n-2}\Gamma(2n-1)}{(2n-2)\Gamma(n-2)\Gamma(n)}\cdot 2^{-4(n-1)}\frac{1}{n-1}\cdot 
    \frac{4}{3}\cdot \left(1+\mathcal{O}(n^{-3/2})\right).\ee
We use the duplication formula for the Gamma function and write
\be\label{eq:duplication} 
 \Gamma(2n-1)=\frac{\Gamma(2n)}{(2n-1)}=\frac{\Gamma(n)\Gamma\left(n+\frac{1}{2}\right)2^{2n-1}}{\sqrt{\pi}(2n-1)}.
\ee
Using \eqref{eq:duplication} in\eqref{eq:edeg2}, we get
\begin{align}\label{eq:edeg3}
 \edeg  G(2, n+1)&=\left(\frac{\pi}{2}\right)^{2n}\frac{8}{3\pi^{5/2}}\frac{1}{n^3}\frac{\Gamma\left(n+\frac{1}{2}\right)}{\Gamma(n-2)}  
   \left(1+\mathcal{O}(n^{-1})\right)\\
&=\left(\frac{\pi}{2}\right)^{2n}\frac{8}{3\pi^{5/2}}\frac{1}{\sqrt{n}} \left(1+\mathcal{O}(n^{-1})\right) ,
\end{align}
where in the last step we have used the asymptotic $\frac{\Gamma\left(n+\frac{1}{2}\right)}{\Gamma(n-2)}=n^{5/2}(1+\mathcal{O}(n^{-1}))$. 
This concludes the proof.
\end{proof}

\begin{figure}
\includegraphics[scale=0.4]{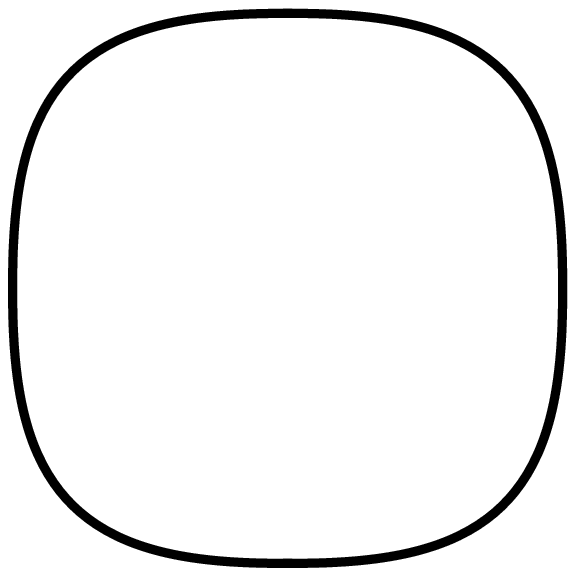} \hspace{2cm}
\includegraphics[scale=0.4]{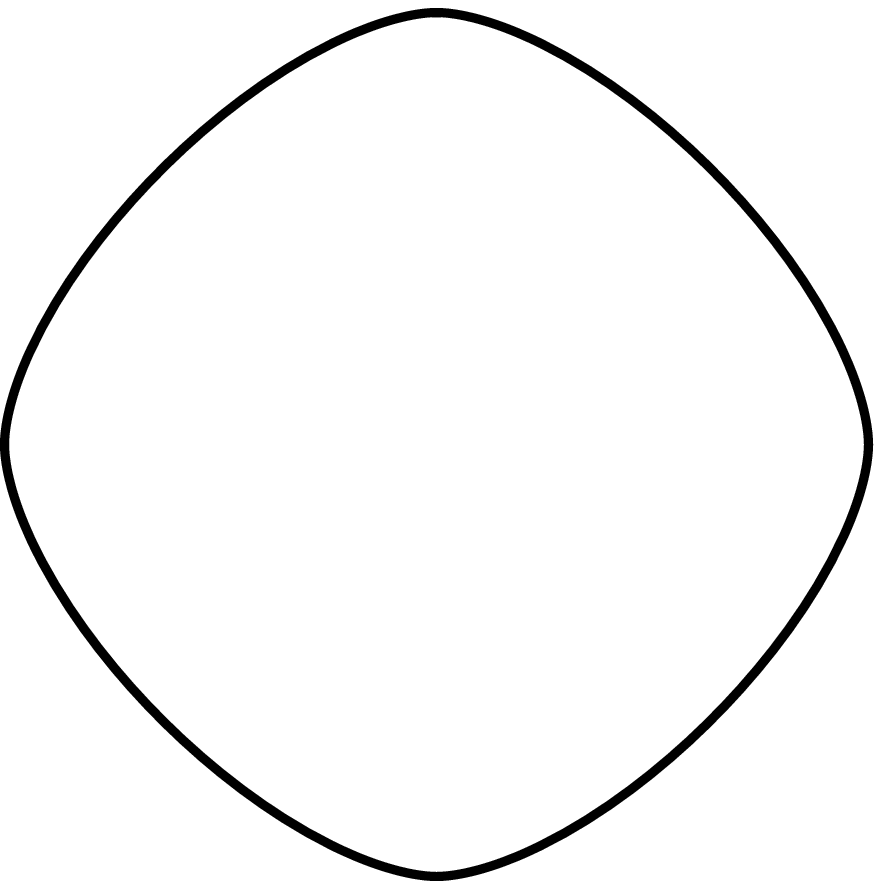}
\caption{\small The convex body $D(2)$ on the left and its polar $D(2)^\circ$ on the right 
 (the bodies are centered at the origin, and the two pictures have the same scale).}\label{fig:D2}
\end{figure}

It remains to show the proposition used in the proof of Theorem \ref{thm:edlines}. 

\begin{prop}\label{prop:r2}
The radial function $r_2(\theta)$ of the singular value zonoid $D(2)$,
parameterized by the angle~$\theta$, 
is smooth in a neighborhood of $\frac{\pi}{4}$.  
Moreover, we have the following expansion  
\be 
r_2(\theta)^2=\frac{1}{8}-\frac{1}{8}\big(\theta -\frac{\pi}{4}\big)^2 
 + \mathcal{O}\left(\big(\theta -\frac{\pi}{4}\big)^4\right) 
\quad \textrm{as $\theta\to \frac{\pi}{4}$}.
\ee
\end{prop}


\begin{proof}
Let us first outline the idea. 
The main difficulty is that $r_2(\theta)$ is only implicitly defined.  
However, by Proposition~\ref{pro:prop-D}, 
we explicitly know the support function $h$ of $D(2)$. 
(See Figure~\ref{fig:D2} for the shape of $D(2)$ and its polar body). 

Below we will show that $h$ is smooth on $\mathbb{R}^2\backslash\{0\}$, 
hence on this open set its gradient $\nabla h$ exists and it is a smooth function. 
Proposition~\ref{prop:nabla} tells us that 
$\textrm{im} (\gamma)\subseteq \partial D(2)$, 
where the function~$\gamma\colon (0,\frac{\pi}{2}) \to\R^2$ is defined by 
\be\label{eq:defgamma} 
  \gamma(t) := (\nabla h)(\cos t, \sin t).
\ee 
We use $\gamma$ to obtain local parametrization of $\partial D(2)$ around $\pi/4$. 
Note that $\gamma(\pi/4)= \frac{1}{4}(1,1)$ by Lemma~\ref{le:radius-D}.
Let $\theta(t)$ denote the angle that $\gamma(t)$ makes with the $x$-axis, that is, 
$\theta(t) := \arctan(\gamma_2(t)/\gamma_1(t))$.
Then we have
\be\label{eq:r_2char} 
 r_2(\theta(t)) = \|\gamma(t)\|  
\ee 
and $\gamma(\pi/4)=\pi/4$.  
We will show that the derivative of $\theta(t)$ does not vanish at $\pi/4$. 
Then the implicit function theorem implies that we can locally invert 
the function $\theta(t)$ around $\pi/4$ to obtain a smooth function $t(\theta)$,  
and it follows from~\eqref{eq:r_2char} that $r_2(\theta)$ is indeed smooth in a neighborhood of $\pi/4$. 
The Taylor expansion of $r_2(\theta)$ around $\pi/4$ will be derived by calculating 
the expansions of the gradient of $h$ and the resulting expansion of $\theta(t)$
around $\pi/4$. 

We proceed now to prove the smoothness of $h$. 
By Proposition~\ref{pro:prop-D}, $h$ can be written as
\be \label{eq:ell}
h(\sigma_1, \sigma_2) =\frac{1}{\sqrt{2\pi}}\, g_2(\sigma_1, \sigma_2) 
 =\frac{1}{(2\pi)^{3/2}}\int_{\R^2}\left(\sigma_1^2 u_1^2+\sigma_2^2u_2^2\right)^{\frac{1}{2}} e^{-\frac12(u_1^2+u_2^2)}\,du_1 du_2.
\ee 
The above integral can be expressed in terms of the complete elliptic integral of the second kind
\be 
 E(s) :=\int_{0}^{\pi/2}\sqrt{1-s (\sin t)^2}dt, \quad 0\le s\leq 1,
\ee 
which is a smooth function of $s$.
We obtain after a short calculation that 
\be\label{eq:hseriesE} 
 h(\sigma_1,\sigma_2) = \frac{1}{\pi}\, |\sigma_1|\, E\left(1-\frac{\sigma_2^2}{\sigma_1^2}\right) .
\ee 
This formula, combined with 
$h(\sigma_1,\sigma_2) =h(\sigma_2,\sigma_1)$, 
implies that $h$ is a smooth function on $\R^2\backslash \{0\}$.

The elliptic integral $E(s)$ has the following expansion for $s\to 0$
\be 
\label{eq:ellexpa} E(s)=\frac{\pi}{2}-\frac{\pi}{8}s-\frac{3\pi}{128}s^2-\frac{5 \pi}{512}s^3+\mathcal{O}(s^4).
\ee
We can derive from this the third order Taylor expansion of $h$ at $(\mu,\mu)$,
where $\mu:=\sqrt{2}/2$.
Denoting $x_i := \sigma_i -\mu$ and $x:=(x_1,x_2)$, 
we obtain after some calculation the third order Taylor expansion of $h$ 
around $x=(0,0)$: 
\be
\begin{split}
 h(x_1+ \mu,x_2 +\mu) &= \frac{\mu}{2} + \frac{1}{4}x_1 + \frac{1}{4}x_2 
  + \frac{\mu}{16} x_1^2 - \frac{\mu}{8} x_1x_2 + \frac{\mu}{16} x_2^2  \\
 & -\frac{1}{32} x_1^3 + \frac{1}{32}x_1^2 x_2 + \frac{1}{32} x_1 x_2^2 - \frac{1}{32} x_2^3 +\mathcal{O}(\|x\|^4).
\end{split}
\ee
From this, we obtain the Taylor expansion for the gradient~$\nabla h$ around $x=(0,0)$: 
\begin{equation}\label{eq:nabla} 
\begin{split}
\partial_{x_1}h(x_1+ \mu,x_2+ \mu) 
   &=\frac{1}{4}+\frac{\mu}{8}x_1 -\frac{\mu}{8}x_2 -\frac{3}{32}x_1^2+\frac{1}{16}x_1 x_2+\frac{1}{32}x_2^2 + \mathcal{O}(\|x\|^3)\\
\partial_{x_2}h(x_1+ \mu,x_2+ \mu) 
   &=\frac{1}{4}-\frac{\mu}{8}x_1+\frac{\mu}{8}x_2+\frac{1}{32}x_1^2+\frac{1}{16}x_1 x_2-\frac{3}{32} x_2^2 + \mathcal{O}(\|x\|^3) .
\end{split}
\end{equation}
Taking into account the definition $\gamma_i(t) = \partial_{x_i}h(\cos t,\sin t)$ 
and substituting $x_1= \cos t - \mu$, $x_2= \sin t - \mu$, we obtain 
from the above expansion of $\nabla h$ by a straightforward calculation (best done with a computer) 
the following expansions for $\tau\to 0$, 
\be\label{eq:gammas}
 \gamma_1(\tau + \pi/4) = \frac{1}{4} - \frac{1}{8} \tau - \frac{1}{16} \tau^2 +\mathcal{O}(\tau^3),\quad 
 \gamma_2(\tau + \pi/4) = \frac{1}{4} +\frac{1}{8} \tau - \frac{1}{16} \tau^2 +\mathcal{O}(\tau^3) .
\ee
This implies 
$\frac{\gamma_2}{\gamma_1}(\tau + \pi/4) = 1 +\tau + \frac12 \tau^2 + \mathcal{O}(\tau^3)$
for $\tau\to 0$, and hence 
\be\label{eq:exptheta}
 \theta(\tau + \pi/4) = \arctan\left(\frac{\gamma_2(\tau + \pi/4)}{\gamma_1(\tau + \pi/4)}\right) 
 =  \frac{\pi}{4} + \frac12 \tau + \mathcal{O}(\tau^3) .
\ee
In particular, we see that $\theta'(\pi/4)=1/2\neq 0$, hence the function $\theta(t)$ 
has a smooth inverse function $t(\theta)$ around $\pi/4$.
Equation~\eqref{eq:gammas} implies with 
$r_2(\theta(t))^2 = \gamma_1(t)^2 + \gamma_2(t)^2$ (see~\eqref{eq:r_2char}) 
that for $\tau\to 0$
\be\label{eq:expr2} 
 r_2(\theta(\tau + \pi/4))^2 = \frac{1}{8} - \frac{1}{32} \tau^2 + \mathcal{O}(\tau^4) .
\ee 
From the expansion~\eqref{eq:exptheta} we obtain for $\theta\to\pi/4$ 
\be\label{eq:exp-t}
 \tau(\theta) := t(\theta) -\frac{\pi }{4} = 2 \left(\theta-\frac{\pi }{4}\right) + \mathcal{O}\left  ((\theta-\frac{\pi }{4})^3\right) .
\ee
Plugging this into~\eqref{eq:expr2}, we finally arrive at
\be 
 r_2(\theta)^2=\frac{1}{8}-\frac{1}{8}\left(\theta -\frac{\pi}{4}\right)^2+\mathcal{O}\left((\theta-\frac{\pi }{4})^4\right)
\quad \textrm{as $\theta\to \frac{\pi}{4}$},
\ee
which concludes the proof.
\end{proof}

\section{Appendix}

\subsection{Proof of Theorem \ref{thm:density}}\label{app:density}

Since the density of $A$ is invariant under the orthogonal group $O(n)$, 
we can assume that $B$ is spanned by the first $l$ standard vectors, i.e.,
\be 
	B=\left[
	\begin{array}{c}
	\textbf{1}\\
	\textbf{0}
	\end{array}\right]\quad \textrm{and} \quad
	A=\left[
	\begin{array}{c}
	A_1\\
	A_2\end{array}\right] ,
\ee
where $\textbf{1}$ is the $\ell\times \ell$ identity matrix, $A_1$ is an $\ell\times k$ matrix 
and $A_2$ an $(n-\ell)\times k$ matrix. 
(Again, abusing notation, we denote by $A$ and $B$ also matrices whose columns span 
the corresponding spaces.)
If we sample $A$ with i.i.d. normal gaussians, 
the corresponding probability distribution for the span of its columns is $O(n)$ invariant, 
and consequently it coincides with the uniform distribution. 
In order to compute the principal angles between $A$ and $B$ using~\eqref{eq:principal},  
we need to orthonormalize the columns of $A$. Defining 
\be 
 \hat{A} :=\left[
	\begin{array}{c}
	A_1\\
	A_2\end{array}\right](A_1^TA_1+A_2^TA_2)^{-1/2}
\ee
we see that the span of the columns of $A$ and $\hat{A}$ is the same, and the columns of $\hat{A}$ are orthonormal. 
The cosines of the principal angles between $A$ and $B$ are the singular values 
$1\geq \sigma_1\geq \cdots \geq\sigma_k\geq 0$ of the matrix:
\be 
  \hat{A}^TB=(A_1^TA_1+A_2^TA_2)^{-1/2}A_1^T.
\ee
The $\sigma_1,\ldots,\sigma_k$ coincide with the square roots of the eigenvalues $1\geq u_1\geq \cdots\geq u_k\geq 0$ 
of the positive semidefinite matrix 
\be 
(A_1^TA_1+A_2^TA_2)^{-1/2}A_1^TA_1(A_1^TA_1+A_2^TA_2)^{-1/2} ,
\ee
which are the same as the eigenvalues of $N=(A_1^TA_1+A_2^TA_2)^{-1}A_1^TA_1$ 
(eigenvalues are invariant under cyclic permutations). 
Consider the Cholesky decomposition $M^TM=A_1^TA_1+A_2^TA_2$. 
Then the eigenvalues of $N$ equal the eigenvalues of
\be 
U=(M^T)^{-1}A_1^TA_1M^{-1}.
\ee
We use now some facts about the multivariate Beta distribution (see \cite[\S3.3]{Muirhead}).
By its definition, 
the matrix $U$ has a $\textrm{Beta}_{k}(\frac{1}{2}l, \frac{1}{2}(n-l))$ distribution. 
\cite[Theorem 3.3.4]{Muirhead} states that 
the joint density of the eigenvalues of~$U$ is given by 
\be\label{eq:joint1}
\frac{ \pi^\frac{k^2}{2} \Gamma_{k}\left(\frac{n}{2}\right)}{\Gamma_{k}\left(\frac{k}{2}\right)\Gamma_{k}\left(\frac{l}{2}\right)
      \Gamma_{k}\left(\frac{n-l}{2}\right)}\prod_{j=1}^ku_j^{(l-k-1)/2}(1-u_j)^{(n-l-k-1)/2}\prod_{i<j}(u_i-u_j).
\ee
Recall now that $u_i=(\cos \theta_i)^2$ for $i=1,\ldots, k$, 
which implies the change of variable $du_i = -2\cos \theta_i \sin \theta_i d\theta_i$, 
and thus \eqref{eq:joint1} becomes the stated density in \eqref{eq:density}.

\subsection{Proof of Lemma~\ref{le:tube}}\label{app:tech-lemma}

We begin with a general reasoning. 
Assume that $A\in e(B)$. The unit normal vectors of $e(B)$ at $A$, up to a sign, are uniquely determined by $A$. 
Lemma~\ref{le:prop-Sigmakn} provides an explit description for them as follows. 
Let $a_1,a_2,\ldots,a_k$  and $a_1,b_2,\ldots,b_{n-k}$ be the orthonormal bases 
given by Lemma~\ref{le:bases} for $A$ and~$B$, respectively;
note that $\dim (A\cap B) =1$ since $A\in e(B)$. 
In particular, $\langle a_i,b_i \rangle = \cos\theta_i$ for $i\le k$, 
where $\theta_1\le\ldots\le\theta_k$ are the principal angles between $A$ and $B$. 
Let $\R f = (A+B)^\perp$ with $\|f\|=1$. 
According to Lemma~\ref{le:prop-Sigmakn}, 
the unit vector $\nu := f\wedge a_2\wedge\cdots\wedge a_k$ 
spans the normal space of $e(B)$ at $A$. 
(We use here the representation of elements of 
$G(k,n)$ and its tangent spaces by vectors in $\Lambda^k \R^n$; 
cf.~\S\ref{se:GrassR}.)

Fix now $B\in G(n-k,n)$ and 
recall that the functions $\vartheta_1,\vartheta_2\colon G(k,n)\to \R$ give 
the smallest and second smallest principal angle, respectively, 
between $A\in G(k,n)$ and~$B$. 
Let $0<\epsilon \le \delta<\pi/2$ and put 
$T := \{ A\in G(k,n) \mid \vartheta_1(A) \le \epsilon,\ \vartheta_2(A)\ge\delta\}$. 
We first prove that 
$\Tu(e(B)_\delta, \epsilon)\subseteq T$.

For $A\in e(B)$ we consider the curve
$N_A(t) := (a_1\cos t + f \sin t )\wedge a_2 \wedge\cdots\wedge a_k$
for $t\in\R$. 
We note that $N_A(t)$ arises from $A$ by a rotation with the angle~$t$ 
in the (oriented) plane spanned by $a_1$ and $f$, fixing the vectors 
in the orthogonal complement spanned by $a_2,\ldots,a_k$. 
We have $N_A(0)=A$ and $\dot{N}_{A}(0) = \nu$.
It is well known (see, e.g., \cite{edelman-arias-smith:99}) 
that $N_A(t)$ is the geodesic through $A$ with 
speed vector~$\nu$, that is, $N_A(t) = \exp_A (t\nu)$ and $N_A(0)=A$. 
From the definition of the $\epsilon$-tube, we therefore have 
\be\label{eq:tubedes}
 \Tu(e(B)_\delta, \epsilon) = \{ N_A(t) \mid A \in e(B)_\delta,\ |t| \le \epsilon \} .
\ee
Since both $a_1$ and~$f$ are orthogonal to $b_2, \ldots, b_{n-k}$,
the principal angles between $N_A(t)$ and $B$ are 
$|t|,\theta_2,\ldots,\theta_k$, where 
$\langle a_j,b_j \rangle = \cos\theta_j$. 
By our assumption $A\in e(B)_\delta$, we have 
$\delta\le\theta_2\le\ldots\le\theta_k$. 
Hence we see that $|t|$ is the smallest principle angle 
between $N_A(t)$ and $B$ if $|t| \le \delta$. 
We have thus verified that $\Tu(e(B)_\delta, \epsilon)\subseteq T$.

For the other inclusion, let $C\in T$; assume  $C\notin e(B)_\delta$, otherwise clearly $C\in\Tu(e(B)_\delta, \epsilon)$. 
Let $(c_1, \ldots, c_k)$ and $(b_1, \ldots, b_{n-k})$ be orthonormal bases of $C$ and $B$, 
respectively, as provided by Lemma~\ref{le:bases}. So we have 
$\langle b_j,c_j \rangle = \cos\vartheta_j(C)$ for all~$j$, 
where $\vartheta_1(C)\le\epsilon$ and 
$\delta\le\vartheta_2(C)\le\ldots\le\vartheta_k(C)$ by the assumption $C\in T$.
In particular, $b_1$ is orthogonal to $c_2,\ldots,c_k$ and $b_1,c_1$ are linearly independent.  
We define $A\in G(k,n)$ as the space spanned by $b_1, c_2, \ldots, c_k$. 
By construction, $\vartheta_1(A)=0$ and $\vartheta_j(A) = \vartheta_j(C)$ for $j\ge 2$. 
Hence, $A\in e(B)_\delta$.  
Let $N_A(t)\in G(k,n)$ be defined as above as the space resulting from $A$ 
by a rotation with the angle~$t$ in the oriented plane spanned by $b_1$ and $c_1$.
By construction, we have $C=N_A(\tau)$ where $\tau = \vartheta_1(C) \le \epsilon$.  
Therefore, we indeed have $C\in \Tu(e(B)_\delta, \epsilon)$ by \eqref{eq:tubedes}.
\hfill$\qed$

\subsection{Proof of Lemma~\ref{pro:TangSpaceChow}}\label{app:Chow}

More generally, we consider a semialgebraic set $Y\subseteq \RP^{n-1}$ 
of dimension $d\leq n-k-1$. Consider the semialgebraic set
\be 
 C(Y) :=\{(A, y)\in G(k,n)\times Y\,|\, y\in \PP(A) \},
\ee
together with the projections on the two factors: $\pi_1\colon C(Y)\to G(k,n)$ and $\pi_2\colon B(Y)\to Y$.  
Since $Z(Y)=\pi_1(C(Y))$, the set $Z(Y)$ is semialgebraic.
Note that $C(Y)$ is compact if $Y$ compact, and $C(Y)$ is connected if $Y$ is connected 
(since $\pi_1$ is continuous). 
In order to determine the dimension of $C(Y)$, we note that 
the fiber $\pi_2^{-1}(y)$ over $y\in Y$ is isomorphic to $G(k-1, n-1)$. 
As a consequence, we get 
\be\label{eq:dimCY}
  \dim C(Y) =\dim Y+(k-1)(n-k).
\ee
The fibers $\pi_1^{-1}(A)$ over $A\in  Z(Y)$ consist of exactly one point, except for the 
$A$ lying in the exceptional set 
\be 
 Z^{(2)}(Y) := \{A\in Z(Y) \,|\, \PP(A)\cap Y\textrm{ consists of at least two points}\} .
\ee
Note that $\PP(A)\cap Y$ consists of one point only, 
for all $A\in Z(Y)\setminus Z^{(2)}(Y)$.

In order to show that $\dim Z^{(2)}(Y) < \dim Z(Y)$,  
we consider the semialgebraic set 
\be 
 C^{(2)}(Y) :=\{(A, y_1, y_2)\in G(k,n)\times Y\times Y\,|\, y_1, y_2\in\PP( A), y_1\ne y_2\}
\ee
with the corresponding projections 
$\pi_3\colon C^{(2)}(Y)\to G(k,n)$ and $\pi_4 \colon C^{(2)}(Y)\to Y\times Y$. 
Note that $Z^{(2)}(Y) = \pi_3(C^{(2)}(Y))$.
For $(y_1, y_2)\in Y\times Y$ such that $y_1\ne y_2$, we have 
$\dim \pi_4^{-1}(y_1, y_2) =k(n-k)-2(n-k)$,
since the fibers of $\pi_4$ are isomorphic to $G(k-2,n-2)$.
Therefore, 
\begin{align} \notag
 \dim C^{(2)}(Y) &= 2\dim Y+k(n-k)-2(n-k) \leq \dim Y+n-k-1+k(n-k)-2(n-k) \\ \label{eq:dC}     
                      &= \dim Y + (k-1)(n-k) -1 \stackrel{\eqref{eq:dimCY}}{=} \dim C(Y) - 1 , 
\end{align}
and we see that $\dim (C(Y) \setminus C^{(2)}(Y)) = \dim C(Y)$. 
For the projection $f\colon C^{(2)}(Y)\to C(Y)$ defined by $f(A, y_1, y_2):=(A, y_1)$, 
we have $\pi_1^{-1}(Z^{(2)}(Y)) = f(C^{(2)}(Y))$, hence 
\be
 \dim \pi_1^{-1}(Z^{(2)}(Y)) \le \dim f(C^{(2)}(Y)) \le \dim C^{(2)}(Y) < \dim C(Y) .
\ee
Moreover, the projection 
$C(Y)\setminus C^{(2)}(Y) \to Z(Y) \setminus Z^{(2)}(Y),\, (A,y)\mapsto y$ is bijective, 
hence 
$$
 \dim (Z(Y) \setminus Z^{(2)}(Y)) = \dim (C(Y) \setminus C^{(2)}(Y)) \stackrel{\eqref{eq:dC}}{=} \dim C(Y) .
$$
Using $\dim Z(Y) \le \dim C(Y)$, we see that 
\be\label{eq:dimZY}
 \dim Z(Y) = \dim C(Y) = \dim Y+(k-1)(n-k) ,
\ee
\be\label{eq:dimZY2}
 \dim Z^{(2)}(Y) \le \dim C^{(2)}(Y) < \dim C(Y) = \dim Z(Y) . 
\ee
In the special case where $\dim X = n-k-1$, we conclude that 
$Z(X)$ is a hypersurface.

We consider now the following set of ``bad'' $A\in Z(X)$:
\be 
 S(X) :=Z^{(2)}(X)\cup Z(\Sing(X))\cup \Sing (Z(X)) \cup \pi_1(\Sing(C(X))) .
\ee
We claim that this semialgebraic set has dimension strictly less than $Z(X)$.  
This follows for $Z^{(2)}(X)$ from~\eqref{eq:dimZY2}, 
for $Z(\Sing X)$ from \eqref{eq:dimZY} applied to $Y=\Sing (X)$, 
for $\pi_1(\Sing(C(X))$ from Proposition \ref{prop:singular} and \eqref{eq:dimZY}, 
and finally for $\Sing(Z(X))$ by Proposition \ref{prop:singular}.

Thus generic points $A\in Z(X)$ are not in $S(X)$ and hence satisfy the following:
\begin{enumerate}
\item the intersection $\PP(A)\cap X$ consists of one point only (let us denote this point by $p$);
\item the point $p$ is a smooth point of $X$;
\item the point $A$ is a regular point of $Z(X)$,
\item the point $(A, p)$ is a regular point of $C(X)$.
\end{enumerate}

It remains to prove that for every $A\in Z(X)\backslash S(X)$ we have \eqref{eq:Tchow}.
Let us take $(A, p)\in C(X)$ with $A\not\in S(X)$. 
We work in local coordinates $(w,y)\in \R^{N}\times \R^{n-1}$ 
on a neighborhood $U$ of $(A,p)$ in $G(k,n)\times \RP^{n-1}$, 
where $N=k(n-k)$. For simplicity we center the coordinates on the origin, 
so that $(0,0)$ are the coordinates of $(A,p)$. 
In this coordinates, the set $C(X)\cap U$ can be described as:
\be \label{eq:CU}C(X)\cap U= \{(w, x)\in \R^{N}\times \R^{n-1}\,|\, F(w,x)=0, G(x)=0\},\ee
where $F(w,x)=0$ represents the reduced local equations describing the condition $x\in \PP(W)$ and $G(x)=0$ 
the reduced local equations giving the condition $x\in X.$ Since $(A, p)$ is a regular point of $C(X)$, 
the tangent space of $C(X)$ at $(A,p)$ is described by:
\begin{align} T_{(A,p)}C(X)&=\{(\dot w, \dot x)\,|\, (D_{(0,0)}F)(\dot w, \dot x)=0, \, (D_{0}G)\dot x=0\}.
\end{align}
Note that $p$ is a smooth point of $X$, hence $(D_{0}G)\dot x=0$ is the equation for the tangent space to~$X$ at $p$.  
On the other hand, since $Z(X)=\pi_1(C(X))$, we have:
\begin{align}\label{eq:TT} 
  T_{A}Z(X)&=D_{(0,0)}\pi_1\left(T_{(A,p)}C(X)\right) . 
\end{align}
Let now $B\subseteq\R^n$ be the linear space corresponding to $T_pX$ and 
let us write the equations for $C(B)$ in the same coordinates 
as above (by construction we have $(A,p)\in C(B)$):
\be C(B)\cap U=\{(w, x)\,|\, F(w,x)=0, (D_{0}G)x=0\}.\ee
Note that the same equation $F(w,x)=0$ as in \eqref{eq:CU} appears here (recall that this is the equation describing $x\in\PP(W)$), 
but now $G=0$ is replaced with its linearization \emph{at zero}.
In particular:
\be T_{(A, p)}C(B)=\{(\dot w, \dot x)\,|\, (D_{(0,0)}F)(\dot w, \dot x)=0, \, (D_{0}G)\dot x=0\},\ee
which coincides with \eqref{eq:TT}. Since $\Omega(B)=\pi_1(C(B))$, this finally implies:
\be 
 T_A\Omega(B)=D_{(0,0)}\pi_1\left(T_{(A,p)}C(B)\right)=D_{(0,0)}\pi_1\left(T_{(A,p)}C(X)\right)=T_{A}Z(X),
\ee
which finishes the proof. \hfill $\qed$

\subsection{Proof of Proposition~\ref{prop:integration}}\label{app:integration} 

We extend the function $f$ to $\R^{k\ti m}\setminus\{0\}$ by setting $f(X) := f(X/\|X\|)$,  
denoting it by the same symbol. Similarly, we extend $g$ by setting $g(\s):=f(\diag_{k,m}(\s))$. 
We assume now that $X\in \R^{k\times m}$ has i.i.d.\ standard Gaussian entries. Then we can write 
\be 
  \frac{1}{|S^{km-1}|} \int_{S^{km-1}} f(X)\, dS^{km-1}  = \E f(X) 
  = \int_{\s_1>\ldots>s_k} g(\s)\, p_{\textrm{SVD}}(\sigma_1, \ldots, \sigma_k)\, d\sigma_1\cdots d\sigma_k ,
\ee
where $p_{\textrm{SVD}}$ denotes the joint density of the \emph{ordered} singular values of $X$. 
This density can be derived as follows.
The joint density of the ordered eigenvalues $\lambda_1>\cdots>\lambda_k>0$ of the Wishart distributed matrix~$XX^T$  
is known to be \cite[Corollary 3.2.19]{Muirhead} 
\be 
p(\lambda_1, \ldots, \lambda_k)
  = c_{k,m} \cdot e^{-\frac{1}{2}\sum_{i=1}^k\lambda_i}\prod_{i=1}^k\lambda_i^{\frac{m-k-1}{2}}\prod_{i<j}(\lambda_i-\lambda_j), 
\ee
where 
\be
 c_{k,m} :=\frac{2^k\pi^{\frac{k^2}{2}}}{2^{\frac{km}{2}}\Gamma_k\left(\frac{k}{2}\right) \Gamma_k\left(\frac{m}{2}\right)} .
\ee
From this, using $\lambda_i(XX^T)=\sigma_i(X)^2$ and the change of variable $d\lambda_i=2\sigma_id\sigma_i$, 
we obtain 
\be 
p_{\textrm{SVD}}(\sigma_1, \ldots, \sigma_k) 
 = c_{k,m} \cdot e^{-\frac{1}{2}\sum_{i=1}^k\sigma_i^2}\prod_{i=1}^k\sigma_i^{m-k}\prod_{i<j}(\sigma_i^2-\sigma_j^2).
\ee
As a consequence, we obtain 
\begin{align}
\E f(X) 
 & = c_{k,m}\int_{\sigma_1>\cdots>\sigma_k>0} g(\sigma)\, e^{-\frac{1}{2}\sum_{i=1}^k\sigma_i^2}
   \prod_{i=1}^k\sigma_i^{m-k}\prod_{i<j}(\sigma_i^2-\sigma_j^2) \, d\sigma_1\cdots d\sigma_k\\
 &=c_{k,m}\int_0^\infty r^{km-1} e^{-\frac12 r^2}\, dr\, \int_{\theta\in S^{k-1}_+} g(\theta)
   \prod_{i=1}^k \theta_i^{m-k} \prod_{i<j} (\theta_i^2-\theta_j^2) \, dS^{k-1}(\theta),
\end{align} 
where in the second line we have switched to polar coordinates 
$\s = r \theta$ with $\theta \in S^{k-1}$ and $r\ge 0$. 
Note that the power of the $r$-variable arises as 
\be 
 k(m-k)+2{k\choose 2}+k-1=km-1.
\ee
Using 
$\int_0^\infty r^{km-1} e^{-\frac12 r^2}\, dr = \Gamma(\frac{km}{2}) \, 2^{\frac{km}{2}-1}$,  
we obtain
\be 
  \E f(X) = c_{k,m}\,\Gamma\left(\frac{km}{2}\right) 2^{\frac{km}{2}-1} 
 \int_{S^{k-1}_+} g(\theta) \prod_{i=1}^k\theta_i^{m-k} \prod_{1\leq i<j\leq k}(\theta_i^2-\theta_j^2)\, dS^{k-1}.
\ee
It is immediate to verify that:
\be  
  |S^{km-1}| \cdot c_{k,m}\, \Gamma\left(\frac{km}{2}\right)2^{\frac{km}{2}-1}= \frac{|O(k)||S(k,m)|}{2^k},
\ee
which completes the proof. \hfill$\qed$ 

\subsection{Generalized Poincar\'e's formula in homogeneous spaces}\label{se:GIGF}

The purpose of this section is to prove the kinematic formula in homogeneous spaces for multiple intersections 
and to derive Theorem \ref{thm:IGF}. The proofs are similar to \cite{Howard}, to which we refer the reader for more details. 

\subsubsection{Definitions and statement of the theorem}

In the following, $G$ denotes a compact Lie group with a left and right invariant Riemannian metric.\footnote{The compactness assumption 
is not essential, but simplifies the statements, for example the modular function of~$G$ is constant 
and does not have to be taken into account, see \cite[\S2.3]{Howard}.} 
See~\cite{Broecker} for background on Lie groups. 
We denote by $e\in G$ the identity element and by $L_{g}\colon G\to G,x\mapsto gx$ 
the left translation by $g\in G$.  
The derivatives of $L_g$ will be denoted by 
$g_*\colon T_e G\to T_gG$. By assumption, this map is isometric. 

In~\eqref{eq:def-sigma-many} we defined a quantity for capturing the relative position of 
linear subspaces of a Euclidean vector space. We can extend this notion to linear 
subspaces $V_i\subseteq T_{g_i}G$ in tangent spaces of $G$ at any points 
$g_1, \ldots, g_m\in G$, assuming $\sum_i\dim V_i \le \dim G$. 
This is done by left-translating the~$g_i$ to the identity $e$: so we define 
\be 
 \sigma(V_1, \ldots, V_m) := \sigma\big( (g_1)_*^{-1}V_1,\ldots,(g_m)_*^{-1}V_m\big) .
\ee

Let now $K\subseteq G$ be a closed Lie subgroup and denote by $p\colon G\to G/K$ the quotient map. 
We endow $K$ with the Riemannian structure induced by its inclusion in $G$, and $G/K$ with 
the Riemannian structure defined by declaring $p$ to be a Riemannian submersion. 
For example, when $G=O(n)$ with the invariant metric defined in Section~\ref{se:vols} and $K=O(k)\times O(n-k)$, 
then $G/K$ with the quotient metric is isometric to the Grassmannian $G(k,n)$ 
with the metric defined in Section~\ref{se:GrassR}. 

Note that $G$ acts naturally by isometries on $G/K$; 
if $g\in G$ and $y\in G/K$, we denote by $g y$ the result of the action. 
Further, we denote by $y_0=p(e)$ the projection of the identity element.
The multiplication with an element $k\in K$ fixes the point $y_0$; as a consequence, 
the differential of $k$ induces a map denoted 
$k_*\colon T_{y_0}G/K\to T_{y_0}G/K$, 
so that we have an induced action of $K$ on $T_{y_0}G/K$. 

Given a submanifold $X$ of a Riemannian manifold $M$, we denote by $NX$ its normal bundle in $M$ 
(i.e., for all $x\in X$ the vector space $N_xX$ is the orthogonal complement to $T_xX$ in $T_xM$). 
Also, the restriction of the Riemannian metric of $M$ to $X$ allows to define a volume density on~$X$; 
if $f\colon X\to \R$ is an integrable function, we denote its integral with respect to this density by $\int _X f(x)\,dx$.

\begin{defi}\label{def:sigma-many}
For given submanifolds $Y_1, \ldots, Y_m\subseteq G/K$, we define the function
\be 
 \sigma_K:Y_1\times \cdots \times Y_m\to \R
\ee
as follows. For $(y_1, \ldots, y_m)\in Y_1\times \cdots \times Y_m$ 
let $\xi_i\in G$ be such that $\xi_i y_i=y_0$ for all~$i$. We define 
\be
 \sigma_K(y_1, \ldots, y_m):=\E_{(k_1,\ldots, k_m)\in K^m}\, \sigma({k_1}_*{\xi_1}_*N_{y_1}Y_1, \ldots,{k_m}_{*}{\xi_m}_*N_{y_m}Y_m) ,
\ee
where the expectation is taken over a uniform $(k_1, \ldots, k_m)\in K\times \ldots\times K$.
\end{defi}

The reader should compare this definition with \cite[Def.~3.3]{Howard}, which is just a special case.
The main result of this section is the following generalization of Poincar\'e's kinematic formula for homogeneous spaces,
as stated in \cite[Thm.~3.8]{Howard} for the intersection of two manifolds. 
We provide a proof, since the more general result is crucial for our work and 
we were unable to find it in the literature. 

\begin{thm}\label{thm:GIGF}
Let $Y_1, \ldots, Y_m$ be submanifolds of $G/K$ such that $\sum_{i=1}^m\codim_{G/K} Y_i \leq \dim G/K$. 
Then, for almost all $(g_1, \ldots, g_m)\in G^{m}$, the manifolds $g_1Y_1, \ldots, g_mY_m$ intersect transversely, and
\be
 \E_{(g_1, \ldots, g_m)\in G^m}|g_1Y_1\cap \cdots\cap g_m Y_m|
   =\frac{1}{|G/K|^{m-1}} \int_{Y_1\times \cdots\times Y_m}\sigma_K(y_1, \ldots, y_m)\, dy_1\cdots dy_m,
\ee 
where the expectation is taken over a uniform $(g_1, \ldots, g_m)\in G\times \cdots\times G.$
\end{thm}

We note that when $G$ acts transitively on the tangent spaces to each $Y_i$ 
(formally defined as in Definition~\ref{def:trans-act}), 
then the function $\sigma_K\colon Y_1\times \cdots \times Y_m\to \R$ 
introduced in Definition~\ref{def:sigma-many} is constant.
As a consequence we obtain: 

\begin{cor}\label{cor:GIGF}
Under the assumptions of Theorem \ref{thm:GIGF}, 
if moreover $G$ acts transitively on the tangent spaces to $Y_i$ for $i=1, \ldots, m$, 
then we have 
\be  
  \E_{(g_1, \ldots, g_m)\in G^m}|g_1Y_1\cap \cdots\cap g_m Y_m|
     =\sigma_K(y_1, \ldots, y_m)\cdot |G/K|\cdot \prod_{i=1}^n\frac{|Y_i|}{|G/K|} ,
\ee
where $(y_1, \ldots, y_m)$ is any point of $Y_1\times \cdots\times Y_m.$
\end{cor}

Let us look now at the special case $G=O(n)$ and $K=O(k)\times O(n-k)$. 
If $Y_1, \ldots, Y_{m}$ are coisotropic hypersurfaces of $G(k,n)$, then $G$ acts transitively on their tangent spaces 
by Proposition \ref{pro:Chow-TA}. Moreover, when $m =k(n-k)$,
it is easy to check that the constant value of $\sigma_K$ 
equals the real average scaling factor $\alpha(k, n-k)$ defined in Definition~\ref{def:alpha}. 
Hence in this case, the statement of 
Corollary~\ref{cor:GIGF} coincides with the statement of Theorem~\ref{thm:IGF}.


\subsubsection{The kinematic formula in $G$}

As before, $G$ denotes a compact Lie group.
We derive first Theorem~\ref{thm:GIGF} in the special case $K=\{e\}$,  
which is the following result (we can w.l.o.g.\ assume $g_m=e$). 

\begin{lemma}\label{le:GGG}
Let $X_1,\ldots,X_m$ be submanifolds of $G$ such that $\sum_{i=1}^m\codim_{G} X_i \leq \dim G$. 
Then 
\be
  \int_{G^{m-1}} |g_1 X_1\cap \cdots\cap g_{m-1} X_{m-1} \cap X_m| \, dg_1\cdots dg_{m-1} 
   = \int_{X_1\times \cdots\times X_m} \sigma(N_{x_1} X_1, \ldots, N_{x_m} X_m)\, dx_1\cdots dx_m,
\ee 
\end{lemma}

In the special case of intersecting two submanifolds, 
this is an immediate consequence of the 
following ``basic integral formula'' from~\cite[\S2.7]{Howard} 
(take $h=1$). 


\begin{prop}\label{prop:basic}
Let $M_1, M_2$ be submanifolds of $G$ such that $\codim_G M_1 +\codim_G M_2 \leq \dim G$. 
For almost all $g\in G$, the manifolds $M_1$ and $gM_2$ intersect transversely, 
and if $h$ is an integrable function on $M_1\times M_2$, then
\be
  \int_G \int_{gM_1\cap M_2} h(\varphi_g(y))\, dy\,dg  
   =\int_{M_1\times M_2}h(x_1, x_2)\, \sigma( N_{x_1}M_1, N_{x_2}M_2)\, dx_1dx_2,
\ee
where $\varphi_g\colon gM_1\cap M_2\to M_1\times M_2$ is the function given by $\varphi_g(y):=(g^{-1}y,y)$. 
\end{prop}


In order to reduce the general case to that of intersecting 
two submanifolds, we first establish a linear algebra identity. 

\begin{lemma}\label{eq:VWZ} 
For subspaces $V_1, \ldots V_\ell, W, Z$ of a Euclidean vector space, we have
\be
 \sigma(V_1, \ldots, V_\ell, W, Z) = \sigma(V_1, \ldots, V_\ell, W+ Z) \cdot \sigma(W,Z) .
\ee
\end{lemma}

\begin{proof}
We may assume that $W\cap Z = 0$, since otherwise both sides of the identity are zero.
Let us denote by $(v_{i,1}, \ldots, v_{i,d_i})$, $(w_{1}, \ldots, w_{a})$, and $(z_{1}, \ldots , z_{b})$ 
orthonormal bases of $V_i$, $W$, and $Z$ respectively. 
Moreover, we denote by 
$(w_{1}, \ldots,w_{a}, \tilde{z}_1,\ldots ,\tilde{z}_b)$ an orthonormal basis for $Z+W$ 
obtained by completing $(w_{1}, \ldots, w_{a})$.
By definition we have:
\be\label{eq:VWZ1} 
\sigma(V_1,\ldots, V_\ell, W+Z)
 =\Big\|\big(\bigwedge_{i,j} v_{i,j}\big)\wedge w_1\wedge \cdots\wedge w_a\wedge \tilde{z}_1\wedge \cdots \wedge\tilde{z}_b\Big\|.
\ee
On the other hand, since $W+Z=\textrm{span}\{w_1, \ldots, w_a, z_1, \ldots, z_b\}$, we have 
\be  
 w_1\wedge \cdots\wedge w_a\wedge \tilde{z}_1\wedge \cdots \wedge\tilde{z}_b 
 =\pm \frac{ w_1\wedge \cdots\wedge w_a\wedge z_1\wedge \cdots \wedge z_b}{\| 
                              w_1\wedge \cdots\wedge w_a\wedge z_1\wedge \cdots \wedge z_b\|} 
 = \pm \frac{ w_1\wedge \cdots\wedge w_a\wedge z_1\wedge \cdots \wedge z_b}{\sigma(W, Z)}.
\ee
Substituting the last line into~\eqref{eq:VWZ1} and recalling the definition of $\sigma(V_1, \ldots, V_\ell, W, Z)$ 
from~\eqref{eq:def-sigma-many}, the assertion follows.
\end{proof}


\begin{cor}\label{le:bif} 
Let $M_1, M_2$ be submanifolds of $G$ and $V_1, \ldots, V_\ell$ be linear subspaces 
of tangent spaces of $G$ (possibly at different points), such that 
$\codim_G M_1 +\codim_G M_2 + \sum_i \dim V_i \leq \dim G$. 
Then we have 
\be
  \int_G \int_{gM_1\cap M_2}\sigma(V_1, \ldots, V_\ell, N_{y}(gM_1\cap M_2))\, dy\,dg 
   =\int_{M_1\times M_2}\sigma(V_1, \ldots, V_\ell, N_{x_1}M_1, N_{x_2}M_2)\, dx_1dx_2 .
\ee
\end{cor}

\begin{proof}
We apply Proposition~\ref{prop:basic} with the function 
$h\colon M_1\times M_2\to \R$ defined by
\be 
 h(x_1, x_2) := \sigma\big(V_1, \ldots, V_\ell, (x_1)_*^{-1} N_{x_1}M_1 + (x_2)_*^{-1} N_{x_2}M_2\big) .
\ee
When $gM_1$ and $M_2$ intersect transversely, we have 
$N_y(gM_1\cap M_2) = N_y(gM_1) + N_y(M_2)$, 
which implies
$h(\varphi_g(y)) = \sigma(V_1, \ldots, V_\ell, N_{y}(gM_1\cap gM_2))$
for $y\in gM_1 \cap M_2$. 
Hence we obtain with Proposition~\ref{prop:basic}, 
\begin{align}
 \int_G\int_{gM_1\cap M_2}\sigma(V_1, \ldots, V_\ell, N_{y}(gM_1\cap M_2))\, dy \,dg 
 &= \int_{G}\left(\int_{gM_1\cap M_2} h(\varphi_g(y))\, dy\right)\,dg\\
 &=\int_{M_1\times M_2} h(x_1, x_2) \sigma(N_{x_1}M_1, N_{x_2}M_2)\, dx_1dx_2\\
 &=\int_{M_1\times M_2}\sigma(V_1, \ldots, V_\ell, N_{x_1}M_1, N_{x_2}M_2)\, dx_1dx_2 ,
\end{align}
where the last equality is due to Lemma~\ref{eq:VWZ}. 
\end{proof}

\begin{proof}[Proof of Lemma~\ref{le:GGG}]
Recall that we already established Lemma~\ref{le:GGG} 
in the case $m=2$ as a consequence of Proposition~\ref{prop:basic}. 
Let $m\ge 3$ and abbreviate 
$Y:= g_2 X_2 \cap Z$, where
$Z:= g_3 X_3 \cap\ldots\cap g_{m-1} X_{m-1} \cap X_m$.
Then we have
\begin{align*}
&\int_{G^{m-2}} \left(\int_{g_1\in G} \big|g_1 X_1 \cap \big( g_2 X_2 \cap\ldots g_{m-1} X_{m-1} \cap X_m \big) \big| \, dg_1
 \right) dg_2\ldots dg_{m-1} \\
&= \int_{G^{m-2}} \left(\int_{X_1} \int_{Y} \s(N_{x_1}X_1, N_y Y) \, dx_1 dy \right) dg_2\ldots dg_{m-1}\\
& = \int_{X_1} \int_{G^{m-3}} \left(\int_{g_2\in G} \int_{Y} \s(N_{x_1}X_1, N_y(g_2 X_2 \cap Z) \, dg_{2}dy \right) dg_3\ldots dg_{m-1} dx_{1} \\
&= \int_{X_1} \int_{G^{m-3}} \left( \int_{x_2\in X_2} \int_{z\in Z} \s(N_{x_1}X_1,N_{x_2}X_2,N_z Z)\, dg_2 \right) dg_3\ldots dg_{m-1} ,
\end{align*}
where we first applied Lemma~\ref{le:GGG} in the case $m=2$, then interchanged the order 
of integration, and after that used Corollary~\ref{le:bif}. 
Proceeding analogously, we see that the above integral indeed equals 
$$
 \int_{X_1\times\cdots \times X_m} \s(N_{x_1}X_1,\ldots,N_{x_m}X_m)\, dx_1\cdots dx_{m},
$$
which completes the proof. 
\end{proof}
 
\begin{proof}[Proof of Theorem~\ref{thm:GIGF}]
We consider $X_i := p^{-1}(Y_i)$, which is a submanifold, 
since the projection $p\colon G\to G/K$ is a submersion, 
cf.~\cite[Thm.~A.15]{BuCu}. 
Moreover, $g_1X_1,\ldots,g_mX_m$ intersect transversally if 
$Y_1,\ldots,Y_m$ do so. 
We can rewrite the integral in the statement as
\begin{align*} 
E &:=\E_{(g_1, \ldots, g_m)\in G^m}|g_1Y_1\cap \cdots\cap g_m Y_m|\\
&=\E_{(g_1, \ldots, g_{m-1})\in G^{m-1}}|g_1Y_1\cap \cdots\cap g_{m-1} Y_{m-1}\cap Y_m|\\
&=\frac{1}{|K|}\frac{1}{|G|^{m-1}}\int_{G^{m-1}}|g_1X_1\cap \cdots\cap g_{m-1} X_{m-1}\cap X_m|\, dg_1\cdots dg_{m-1} ,
\end{align*}
For justifying the last equality, note that 
the coarea formula~\cite[Thm.~17.8]{BuCu} 
yields $|p^{-1}(Y)|=|K||Y|$ 
for any submanifold $Y$ of $G/K$, since 
$p$ is a Riemannian submersion.
Applying Lemma~\ref{le:GGG} to the integral in the last line, we obtain
\be
E=\frac{1}{|K|}\frac{1}{|G|^{m-1}}\int_{X_1\times \cdots\times X_m}\sigma(N_{x_1}X_1,\ldots, N_{x_ m}X_m)\, dx_1\cdots dx_{m} .
\ee
The projection 
$P\colon X_1\times \cdots \times X_m\to Y_1\times \cdots\times Y_m$ 
defined by 
$p(x_1, \ldots, x_m) := (p(x_1),\ldots, p(x_m))$  
is a Riemannian submersion with fibers isometric to $K^m$.
Using the coarea formula
\begin{align*}
& \int_{X_1\times\cdots\times X_m} \s(N_{x_1}X_1,\ldots,N_{x_m}X_m)\, dx_1\cdots dx_{m} \\
&= \int_{y\in Y_1\times\cdots Y_m} \int_{x\in P^{-1}(y)} \s(N_{x_1}X_1,\ldots,N_{x_m}X_m)\, dx\, dy \\
&= |K|^m \int_{y\in Y_1\times\cdots Y_m} \s_K(y_1,\ldots,y_m)\, dy ,
\end{align*}
where the last equality is due to Definition~\ref{def:sigma-many}. (Note that here is where we use the right invariance of the metric under the action of the elements in $K$, as one can verify by a careful inspection of the last steps.  The reader can see the proof of \cite[Thm. 3.8]{Howard}, which is almost identical and where all the details of the calculation are shown).
This completes the proof. 
\end{proof}


\bibliographystyle{plain}
\bibliography{PSCarxiv2}

\end{document}